\documentclass[11pt]{amsart}
\usepackage{amsmath,amssymb,amsfonts,exscale}

\usepackage[curve,matrix,arrow,cmtip]{xy}
\usepackage{verbatim}

\newdir^{ (}{{}*!/-3pt/\dir^{(}}    
\newdir_{ (}{{}*!/-3pt/\dir_{(}}

\def\Frob{\mathop{\rm Frob}\nolimits}
\def\Hom{\mathop{\rm Hom}\nolimits}

\def\Spec{\mathop{\rm Spec}\nolimits}

\def\deg{\mathop{\rm deg}\nolimits}

\def\im{\mathop{\rm im}\nolimits}

\def\supp{\mathop{\rm supp}\nolimits} 
 
\def\Mat{\mathop{\rm Mat}\nolimits}

\def\GL{\mathop{\rm GL}\nolimits}
\def\SL{\mathop{\rm SL}\nolimits}

\def\Bun{\mathop{\rm Bun}\nolimits}

\def\Isom{\mathop{\rm Isom}\nolimits}
\def\Maps{\mathop{\rm Maps}\nolimits}

\def\VinBun{\mathop{\rm VinBun}\nolimits}
\def\barVinBun{\mathop{\overline{\rm VinBun}}\nolimits}

\def\Qellbar{\mathop{\overline{\BQ}_\ell}\nolimits}
\def\IC{\mathop{\rm IC}\nolimits}
\def\can{\mathop{\rm can}\nolimits}
\def\var{\mathop{\rm var}\nolimits}
\def\sign{\mathop{\rm sign}\nolimits}
\def\tr{\mathop{\rm tr}\nolimits}
\def\Gr{\mathop{\rm Gr}\nolimits}
\def\CT{\mathop{\rm CT}\nolimits}

\def\add{{\rm add}}

\def\id{{\rm id}}

\def\Sym{\mathop{\rm Sym}\nolimits}

\newbox\starbox 
\setbox\starbox=\hbox{\vphantom{\underline{\ }}\rlap{\kern4pt\small?}\smash{\lower2pt\hbox{\Large$\square$}}}

\def\hatE{{\mathchoice
  {\hbox{\rlap{\smash{\kern1pt\lower1pt\hbox{$\widehat{\phantom{\hbox{$E$}}}$}}}$E$}}
  {\hbox{\rlap{\smash{\kern1pt\lower1pt\hbox{$\widehat{\phantom{\hbox{$E$}}}$}}}$E$}}
  {\widehat E}
  {\widehat E}}}

\def\hatW{\hbox{\rlap{\smash{\lower1pt\hbox{$\widehat{\phantom{\hbox{$W$}}}$}}}$W$}}

\def\tildeW{\hbox{\rlap{\smash{\lower1pt\hbox{$\widetilde{\phantom{\hbox{$W$}}}$}}}$W$}}

\newbox\checkWbox
\setbox\checkWbox\hbox{\rlap{\smash{\kern.8pt\lower4pt\hbox{\huge \v{}}}}$W$}

\def\barBun{{\overline{\Bun}}}

\def\gr{{\rm gr}}

\def\D{{\rm D}}
\def\sl{{\mathfrak{sl}}}

\def\circV{{\mathchoice{\circVbig}{\circVbig}{\circVscript}{\circVscriptscript}}}
\def\circVbig{\hbox{\text{\it\r{V}}}}
\def\circVscript{\hbox{\scriptsize\text{\it\r{V}}}}
\def\circVscriptscript{\mbox{\tiny\text{\it\r{V}}}}
\def\circVprime{{\mathchoice{\circV\kern1.8pt{}^\prime}{\circV\kern1.8pt{}^\prime}
                            {\circVscript\kern1.3pt{}^\prime}{\circVscriptscript\kern1pt{}^\prime}}}

\def\circVpprime{{\mathchoice{\circV\kern1.8pt{}^{\prime \prime}}{\circV\kern1.8pt{}^{\prime \prime}}
                            {\circVscript\kern1.3pt{}^{\prime \prime}}{\circVscriptscript\kern1pt{}^{\prime \prime}}}}

\let\epsilon\varepsilon
\let\setminus\smallsetminus

\let\leq\leqslant
\let\geq\geqslant

\newtheorem{theorem}[subsubsection]{Theorem}

\newtheorem{corollary}[subsubsection]{Corollary}

\newtheorem{proposition-definition}[subsubsection]{Proposition-Definition}
\newtheorem{theorem-definition}[subsubsection]{Theorem-Definition}

\newtheorem{lemma}[subsubsection]{Lemma}

\newtheorem{proposition}[subsubsection]{Proposition}
\newtheorem{question}[subsubsection]{Question}
\newtheorem{remark}[subsubsection]{Remark}

\newcommand{\BA}{{\mathbb{A}}}

\newcommand{\BD}{{\mathbb{D}}}

\newcommand{\BF}{{\mathbb{F}}}
\newcommand{\BG}{{\mathbb{G}}}

\newcommand{\BP}{{\mathbb{P}}}
\newcommand{\BQ}{{\mathbb{Q}}}

\newcommand{\BY}{{\mathbb{Y}}}
\newcommand{\BZ}{{\mathbb{Z}}}

\newcommand{\FB}{{\mathfrak{B}}}

\newcommand{\CB}{{\mathcal B}}

\newcommand{\CO}{{\mathcal O}}
\newcommand{\CP}{{\mathcal P}}

\newcommand{\CU}{{\mathcal U}}

\newcommand{\CY}{{\mathcal Y}}
\newcommand{\CZ}{{\mathcal Z}}

\newcommand{\ssec}{\subsection}
\newcommand{\sssec}{\subsubsection}

\def\longto{\longrightarrow}
\def\into{\hookrightarrow}
\let\onto\twoheadrightarrow

\def\longinto{\lhook\joinrel\longrightarrow}
\def\longintointo{\lhook\joinrel\lhook\joinrel\lhook\joinrel\longrightarrow}
\def\longotni{\longleftarrow\joinrel\rhook}
\def\longonto{\ontoover{\ }}

\newbox\mybox
\def\arrover#1{\mathrel{
       \setbox\mybox=\hbox spread 1.4em
              {\hfil$\scriptstyle#1\vphantom{g}$\hfil}
       \vbox{\offinterlineskip\copy\mybox
             \hbox to\wd\mybox{\rightarrowfill}}}}
\def\larrover#1{\mathrel{
       \setbox\mybox=\hbox spread 1.4em{\hfil$\scriptstyle#1$\hfil}
       \vbox{\offinterlineskip\copy\mybox
             \hbox to\wd\mybox{\leftarrowfill}}}}

\def\ontoover#1{\mathrel{
       \setbox\mybox=\hbox spread 1.4em{\hfil$\scriptstyle#1$\hfil}
       \vbox{\offinterlineskip\copy\mybox
             \hbox to\wd\mybox{\rightarrowfill\hskip-2.8mm
                               $\rightarrow$}}}}
\def\leftontoover#1{\mathrel{
       \setbox\mybox=\hbox spread 1.4em{\hfil$\scriptstyle#1$\hfil}
       \vbox{\offinterlineskip\copy\mybox
             \hbox to\wd\mybox{$\leftarrow$\hskip-2.8mm
                               \leftarrowfill}}}}

\newbox\invlimsymbol
\setbox\invlimsymbol=\vtop{\hbox{\rm lim}\vskip-8pt
        \hbox{\hskip1pt$\scriptstyle\longleftarrow$}\vskip-1pt}

\newbox\dirlimsymbol
\setbox\dirlimsymbol=\vtop{\hbox{\rm lim}\vskip-8pt
        \hbox{\hskip1pt$\scriptstyle\longrightarrow$}\vskip-1pt}

\begin{document}

\title[Picard-Lefschetz oscillators]{Picard-Lefschetz oscillators for the Drinfeld-Lafforgue-Vinberg degeneration \\ for $SL_2$}

\author{Simon Schieder}
\thanks{Department of Mathematics, MIT, Cambridge, MA 02139, USA}

\address{Department of Mathematics, MIT, Cambridge, MA 02139, USA}
\email{schieder@mit.edu}

\begin{abstract}
Let $G$ be a reductive group and let $\Bun_G$ denote the moduli stack of $G$-bundles on a smooth projective curve.
We begin the study of the singularities of a canonical compactification of $\Bun_G$ due to V. Drinfeld (unpublished), which we refer to as the \textit{Drinfeld-Lafforgue-Vinberg compactification} $\barBun_G$.
For $G=\GL_2$ and $G=\GL_n$ certain smooth open substacks of this compactification have already appeared in the work of Drinfeld and of L.~Lafforgue on the Langlands correspondence for function fields. The stack $\barBun_G$ is however already singular for $G=\SL_2$; questions about its singularities arise naturally in the geometric Langlands program, and form the topic of the present article.

Drinfeld's definition of $\barBun_G$ for a general reductive group~$G$ relies on the Vinberg semigroup of~$G$, and will be given and studied in \cite{Sch1}. In the present paper we focus on the case $G = \SL_2$. In this case the compactification can alternatively be viewed as a canonical one-parameter \textit{degeneration} of the moduli space of $\SL_2$-bundles. We study the singularities of this one-parameter degeneration via the weight-monodromy theory of its nearby cycles: We give an explicit description of the nearby cycles sheaf together with its monodromy action in terms of certain novel perverse sheaves which we call ``Picard-Lefschetz oscillators'' and which govern the singularities of $\barBun_G$. We then use this description to determine the intersection cohomology sheaf of $\barBun_G$ and other invariants of its singularities. Our proofs rely on the construction of certain local models which themselves form one-parameter families of spaces which are \textit{factorizable} in the sense of Beilinson and Drinfeld.

We also briefly discuss the relationship of our results for $G=\SL_2$ with the \textit{miraculous duality} of Drinfeld and Gaitsgory in the geometric Langlands program, as well as two applications of our results to the classical theory on the level of functions: To Drinfeld's and Wang's \textit{strange invariant bilinear form} on the space of automorphic forms; and to the categorification of the \textit{Bernstein asymptotics map} studied by Bezrukavnikov and Kazhdan as well as by Sakellaridis and Venkatesh.
\end{abstract}

\maketitle

\bigskip
\bigskip
\bigskip
\bigskip
\bigskip

\newpage

\tableofcontents

\section{Introduction}

\ssec{Overview}

\sssec{\textbf{The relative compactification $\barBun_G$}}
For the purpose of this introduction, let $\mathsf{k}$ be an algebraically closed field of arbitrary characteristic; see Subsection \ref{Notation and conventions} below for the conventions we adopt in the main body of the article.
Let $X$ be a smooth projective curve over $\mathsf{k}$, let $G$ be a reductive group over $\mathsf{k}$, and let $\Bun_G$ denote the moduli stack of $G$-bundles on $X$. In this article we begin the study of the singularities of a canonical relative compactification $\barBun_G$ of $\Bun_G$ discovered by V. Drinfeld (unpublished). Recall first that the diagonal morphism
$$\Delta: \ \Bun_G \ \stackrel{\Delta}{\longto} \ \Bun_G \times \Bun_G$$
of $\Bun_G$ is not proper. Drinfeld has hence defined a larger stack $\barBun_G$ together with a factorization of the diagonal $\Delta$ as
$$\xymatrix@+10pt{
\Bun_G \ar[r] \ar@/_-2pc/[rr]^{\Delta} & \barBun_G \ar^{\bar\Delta \ \ \ \ \ }[r] & \Bun_G \times \Bun_G \, .
}$$
where the map $\bar\Delta$ is proper; thus $\barBun_G$ is a compactification of $\Bun_G$ relative to the diagonal morphism $\Delta$ of $\Bun_G$. (As a technical comment, we remark that in the definition of $\barBun_G$ used in the present article, the map $\Bun_G \to \barBun_G$ is not quite an open immersion; rather, the stack $\barBun_G$ contains as a dense open substack the product $\Bun_G \times \cdot/Z_G$, where $Z_G$ denotes the center of $G$).

\medskip

For $G = \GL_2$ and for $G = \GL_n$ certain smooth open substacks of $\barBun_G$ were used by Drinfeld and by L. Lafforgue in their seminal work on the Langlands correspondence for function fields (\cite{Varieties of modules of F -sheaves}, \cite{Cohomology of compactified manifolds of modules of F -sheaves of rank 2}, \cite{Lafforgue JAMS}). For a general reductive group $G$, Drinfeld's definition of $\barBun_G$ has never been published, and relies on the Vinberg semigroup of $G$;
it will appear in \cite{Sch1}.
While the open substacks used by Drinfeld and Lafforgue for $G=\GL_n$ are smooth, the stack $\barBun_G$ is already singular for $G = \SL_2$.

\medskip

The goal of the present article and the future article \cite{Sch1} is the study of the singularities of $\barBun_G$. As is explained in Subsections \ref{Relation to the geometric Langlands program} through \ref{An application to Bernstein asymptotics} of this introduction, this study is motivated by the geometric Langlands program as well as by applications to the classical theory on the level of functions.

\medskip

\sssec{\textbf{The degeneration $\VinBun_G$}}
In the present article we study the singularities of $\barBun_G$ in the special case $G = \SL_2$; the case of a general reductive group will be treated in \cite{Sch1}. For most of the present article it will be more convenient for us to work with a minor modification of $\barBun_G$ which we denote by $\VinBun_G$ and refer to as the \textit{Drinfeld-Lafforgue-Vinberg degeneration} of $\Bun_G$. The space $\VinBun_G$ can be viewed as the total space of a canonical $\BG_m$-bundle on $\barBun_G$ (see Subsection \ref{The definition for G=SL_2} below), so that from the viewpoint of singularities it is equivalent to study either $\barBun_G$ or $\VinBun_G$.

\medskip

As indicated by its name, the modification $\VinBun_G$ of $\barBun_G$ naturally forms a one-parameter degeneration of $\Bun_G$: It admits a natural map to the affine line
$$v: \ \VinBun_G \ \longto \ \BA^1$$
whose fibers away from $0 \in \BA^1$ are isomorphic to $\Bun_G$, and whose fiber over~$0$ is singular. This degeneration of $\Bun_G$ is in fact canonical: It is induced by the canonical Vinberg semigroup degeneration of the group $G$, found by Vinberg \cite{V}, which will however not appear explicitly in this article.

\medskip

\sssec{\textbf{The definition for $G=\SL_2$}}
\label{The definition for G=SL_2}

For $G = \SL_2$ it is in fact easy to state the definition of $\VinBun_G$ as a moduli space: It parametrizes triples $(E_1, E_2, \varphi)$ consisting of two $\SL_2$-bundles $E_1$, $E_2$ on the curve $X$ together with a morphism of the associated vector bundles $\varphi: E_1 \to E_2$ which is required to be not the zero map. Taking the determinant of the map $\varphi$ defines the desired map
$$v: \ \VinBun_G \ \longto \ \BA^1 \, .$$
In these terms, the moduli space $\barBun_G$ can be recovered as the quotient of $\VinBun_G$ by the $\BG_m$-action given by scalar multiplication of the map $\varphi$. In other words, the spaces $\VinBun_G$ and $\barBun_G$ naturally fit into the cartesian square
$$\xymatrix@+10pt{
\VinBun_G \ar[r] \ar[d]_{v  } &    \barBun_G \ar[d]^{\bar{v}}          \\
\BA^1 \ar[r]                            &    \BA^1/\BG_m             \\
}$$
where the action of $\BG_m$ on $\BA^1$ in the lower right corner is the quadratic action.

\bigskip

\sssec{\textbf{Motivation and the contents of this introduction}}
The main geometric results of this article deal with the description of the singularities of the family
$$v: \ \VinBun_G \ \longto \ \BA^1$$
as well as of the singularities of the total space $\VinBun_G$: We determine the nearby cycles perverse sheaf of the family $v$ together with its monodromy action; the intersection cohomology sheaf of the total space $\VinBun_G$; as well as the stalks of the $*$-extension of the constant sheaf from the open locus $\VinBun_G|_{\BA^1 \setminus \{0\}}$.
Our main theorems describing these invariants of the singularities are sketched below in Subsection \ref{Main results about the singularities of VinBun_G} of this introduction; we refer the reader to Section \ref{Statement of main theorems} for the precise statements of these theorems.

\medskip

In Subsection \ref{Relation to the geometric Langlands program} of this introduction we discuss how our results are related to the geometric Langlands program and Drinfeld's and Gaitsgory's \textit{miraculous duality}. In Subsection \ref{An application to Drinfeld's strange invariant bilinear form} we discuss an application of our results to Drinfeld's and Wang's \textit{strange invariant bilinear form} on the space of automorphic forms; this corresponds to Section \ref{An application: Computation of Drinfeld's function} of the main text. In Subsection \ref{An application to Bernstein asymptotics} we discuss another application of our results on the level of functions, to the \textit{Bernstein asymptotics map} studied by Bezrukavnikov and Kazhdan and by Sakellaridis and Venkatesh; in particular, we explain how a special case ($G=\SL_2$) of a conjecture of Sakellaridis can be deduced from our results. In Subsection \ref{Proofs via local models} of this introduction we briefly discuss one ingredient of our proofs, the construction of certain \textit{local models} for the degeneration $\VinBun_G$.

\bigskip

\ssec{Main results about the singularities of $\VinBun_G$}
\label{Main results about the singularities of VinBun_G}

\mbox{}
\medskip

In this subsection we briefly sketch the main geometric results of this article.
In order to do so, we first discuss a natural stratification of the special fiber $v^{-1}(0)$ of the family $\VinBun_G$.

\sssec{\textbf{The defect stratification}}
\label{introduction --- The defect stratification}
To each point $\varphi: E_1 \to E_2$ of the special fiber $v^{-1}(0)$ of the family $\VinBun_G$ one can naturally associate an effective divisor $D$ on the curve $X$ which measures the ``defect'' of the map $\varphi$, i.e., it yields a measure of ``how singular'' the point $\varphi: E_1 \to E_2$ is.
One may consider $D$ the divisor of zeroes of the map $\varphi$; see Subsection \ref{Definition of the defect} below.
The special fiber of $\VinBun_G$ can then be stratified into strata $\sideset{_k}{_G}\VinBun$ on which the \textit{defect} $k$, i.e., the degree of the associated \textit{defect divisor} $D$, is equal to $k$. Associating to each point its defect divisor we obtain natural maps
$$p_k: \ \sideset{_k}{_G}\VinBun \ \longto \ X^{(k)}$$
where $X^{(k)}$ denotes the $k$-th symmetric power of the curve $X$. Finally, let
$$j_k: \ \sideset{_k}{_G}\VinBun \ \longinto \ \overline{\sideset{_k}{_G}\VinBun}$$
denote the inclusion of a stratum into its closure, and let $(j_k)_{!*}$ denote the corresponding intermediate extension functor for perverse sheaves.

\medskip

\sssec{\textbf{Main theorem about nearby cycles}}
Broadly speaking, our main theorem about the nearby cycles $\Psi$ of $\VinBun_G$, Theorem \ref{main theorem for nearby cycles}, expresses the associated graded $\gr \Psi$ of the weight-monodromy filtration on $\Psi$ as the direct sum
$$\gr \, \Psi \ \ = \ \ \bigoplus_{k \in \BZ_{\geq 0}} \ \, (j_k)_{!*} \ p_k^* \ \CP_k$$
where the $\CP_k$ denote certain novel perverse sheaves on $X^{(k)}$ which we call \textit{Picard-Lefschetz oscillators}. Their definition is rather simple and originates in the classical Picard-Lefschetz theory (see Subsection \ref{Picard-Lefschetz oscillators}). Our theorem includes a description of the monodromy action on the nearby cycles $\Psi$: By construction the Picard-Lefschetz oscillators come equipped with actions of the Lefschetz-$\sl_2$, and our theorem in fact asserts that the above isomorphism identifies the induced action of the Lefschetz-$\sl_2$ on the right hand side with its monodromy action on the left hand side. Finally, in the actual formulation of our theorem, the intermediate extensions $(j_k)_{!*}$ are explicitly computed via certain finite resolutions of singularities of the strata closures, which we construct using the smooth relative compactifications of the map $\Bun_B \to \Bun_G$ found by Drinfeld (\cite{BG1}) and Laumon (\cite{Lau}); see Theorem \ref{main theorem for nearby cycles} for the precise statement.

\medskip

We furthermore remark that the nearby cycles of $\VinBun_G$ are well-behaved in the sense that they \textit{factorize}, in the sense of Beilinson and Drinfeld (\cite{BD1}, \cite{BD2}), with respect to the defect: Broadly speaking, this means that the stalks of the perverse sheaf $\Psi$ decompose into a tensor product of \textit{local factors}, each of which is associated with a point on the curve $X$ occurring in the defect divisor. Somewhat surprisingly, this factorization property does not hold for the IC-sheaf of $\VinBun_G$; see Subsection \ref{Intersection Cohomology} below.

\medskip

The singularities of the total space $\VinBun_G$, and hence of the compactification $\barBun_G$, appear to be most easily studied via the above nearby cycles theorem, i.e., via the Picard-Lefschetz oscillators. We now illustrate this phenomenon in the example of the computation of the IC-sheaf of $\VinBun_G$:

\medskip

\sssec{\textbf{Intersection Cohomology}}
\label{Intersection Cohomology}

Our second main theorem, Theorem \ref{main theorem for intersection cohomology}, describes the IC-sheaf of $\VinBun_G$. We first point out that, as remarked above, the IC-sheaf of $\VinBun_G$ exhibits the peculiarity that it \textit{does not factorize} in the sense of Beilinson and Drinfeld, as we now explain; this sets the spaces $\barBun_G$ and $\VinBun_G$ apart from similarly defined singular moduli spaces in the geometric Langlands program such as Drinfeld's relative compactifications $\barBun_B$ (\cite{BG1}, \cite{BFGM}).

\medskip

As was mentioned in Subsection \ref{introduction --- The defect stratification} above, one can naturally associate to any singular point of $\VinBun_G$ its defect divisor. The IC-stalk at the singular point will in fact essentially only depend on this divisor; let us hence temporarily denote the IC-stalk at a singular point with associated defect divisor $D$ by $\IC_D$. A natural expectation is then that the IC-sheaf \textit{factorizes}, i.e., if $D = D_1 + D_2$ for two effective divisors $D_1$ and $D_2$ on $X$ with disjoint supports, then (up to shifts and twists) we have
$$\IC_D \ = \ \IC_{D_1} \otimes \IC_{D_2} \, .$$
Less formally, one might expect that distinct points on the curve ``cannot see each other'' in the sense that they contribute to the IC-stalk independently; this is indeed the case for Drinfeld's $\barBun_B$ (see \cite{FFKM}, \cite{BFGM}, and Subsection \ref{Proofs via local models} below).

\medskip

It however turns out that the IC-sheaf of $\VinBun_G$ does not factorize in this sense; this appears to make it difficult to carry out the approach of \cite{BFGM} to compute the IC-stalks directly. For this reason, we instead access the IC-sheaf via our understanding of the monodromy action on the nearby cycles: From the nearby cycles theorem we deduce a formula for the weight filtration of the restriction of the IC-sheaf of $\VinBun_G$ to the special fiber (Theorem \ref{main theorem for intersection cohomology}). Extracting this description from the formula for the nearby cycles amounts to computing the perverse kernel of the monodromy operator on the Picard-Lefschetz oscillators, which can be done systematically using the classical Schur-Weyl duality. Furthermore, exploiting the geometry of the defect stratification it is also possible to compute the IC-stalks from this description of the weight filtration (see Remark \ref{IC stalks}).

\medskip

For an example where a direct computation of the IC-stalks is possible in our setting, we refer the reader to Subsection \ref{Picard-Lefschetz oscillators via intersection cohomology}, where the case of defect $\leq 2$ is treated and used for our proof of the nearby cycles theorem. A geometric explanation for the lack of factorization of the IC-sheaf can be found in the nature of the local models for $\VinBun_G$ which we discuss in Subsection \ref{Proofs via local models} of this introduction.

\medskip

\sssec{\textbf{Stalks of the $*$-extension of the constant sheaf}}
\label{Stalks of the *-extension of the constant sheaf}
Let
$$j: \ \VinBun_G \big|_{\BA^1 \setminus \{0\}} \ \longinto \ \VinBun_G$$
denote the open inclusion of the inverse image of $\BA^1 \setminus \{0\}$ under $v$. One ingredient in our proof of the nearby cycles theorem which might be of independent interest is the computation of the $*$-stalks of the $*$-extension $j_*$ of the constant sheaf on $\VinBun_G \big|_{\BA^1 \setminus \{0\}}$: Theorem \ref{ij} describes these stalks in terms of the cohomology
of the ``open'' Zastava spaces defined and studied in \cite{FM}, \cite{FFKM}, \cite{BFGM}, and \cite{BG2}. An application of this stalk computation to the classical theory of automorphic forms is discussed in Subsection \ref{An application to Drinfeld's strange invariant bilinear form} of this introduction, and carried out in Section \ref{An application: Computation of Drinfeld's function} of the main text.

\bigskip

\ssec{Relation to the geometric Langlands program}
\label{Relation to the geometric Langlands program}

\mbox{} \medskip

We now briefly recall Drinfeld's and Gaitsgory's \textit{miraculous duality functor} (\cite{DrG1}, \cite{DrG2}, \cite{G2}) in the geometric Langlands program, and indicate how our results maybe be used in its study.

\sssec{\textbf{Drinfeld's and Gaitsgory's miraculous duality functor}}

This subsection is of motivational nature. For the purpose of exposition we implicitly make various simplifying assumptions and suppress any technicalities.
We refer the reader to the original articles (\cite{DrG1}, \cite{DrG2}, \cite{G2}) for a rigorous treatment.
Let $D(\Bun_G)$ denote the derived category of constructible complexes on $\Bun_G$. In the articles \cite{DrG1} and \cite{G2} Drinfeld and Gaitsgory have introduced a natural self-equivalence of $D(\Bun_G)$, the ``miraculous duality functor" or ``pseudo-identity functor''
$$\text{Ps-Id}: \ D(\Bun_G) \ \stackrel{\cong}{\longto} \ D(\Bun_G).$$
(Without going into detail in this introduction, we remark that the miraculous duality really yields a self-equivalence only on the subcategory of compact objects; on the full DG-categories it defines an equivalence between $D(\Bun_G)$ and its dual category.)
This functor plays an important role in the geometric Langlands program: For example, it arises naturally when attempting to relate Verdier duality on the automorphic side of the categorical geometric Langlands correspondence to Serre duality on the spectral side via the conjectural geometric Langlands transform (\cite{G2}). The functor $\text{Ps-Id}$ furthermore plays a key role in Drinfeld's and Gaitsgory's \textit{strange functional equations} for the Eisenstein Series and Constant Term functors (\cite{DrG2}, \cite{G2}), and is closely related to Drinfeld's and Wang's \textit{strange invariant bilinear form} on the space of automorphic forms (\cite{DrW}).
Finally, we remark that the functor $\text{Ps-Id}$ acts as the identity functor on the subcategory of $D(\Bun_G)$ consisting of cuspidal objects; this might explain why the need to study the functor $\text{Ps-Id}$ has not arisen earlier historically.

\bigskip

\sssec{\textbf{Relation to $\barBun_G$}}
To explain the relationship of the miraculous duality functor $\text{Ps-Id}$ with the present work, let $\Delta$ denote the diagonal morphism of $\Bun_G$ and let ${\Qellbar}$ denote the constant sheaf on $\Bun_G$. Then the functor $\text{Ps-Id}$ is constructed as an integral transform defined by the kernel $\Delta_! ({\Qellbar}) \in D(\Bun_G \times \Bun_G)$. Note that since $\Bun_G$ is a non-separated algebraic stack, the diagonal morphism $\Delta$ is not a closed immersion; in fact, it is neither a monomorphism nor proper. Unlike in the separated case, the objects $\Delta_! ({\Qellbar})$ and $\Delta_* ({\Qellbar})$ thus do not agree, and form interesting complexes on $\Bun_G \times \Bun_G$. To study the functor $\text{Ps-Id}$, one is naturally led to studying these complexes; we restrict our attention to $\Delta_! ({\Qellbar})$ for the purpose of this introduction.

\medskip

Using the factorization of the diagonal $\Delta$ as
$$\xymatrix@+10pt{
\Bun_G \ar^{j}[r] \ar@/_-2pc/[rr]^{\Delta} & \barBun_G \ar^{\bar\Delta \ \ \ \ \ }[r] & \Bun_G \times \Bun_G \,
}$$
and exploiting the properness of the map $\bar\Delta$, one can reduce questions about the complex $\Delta_! ({\Qellbar})$ to questions about the complex $j_!({\Qellbar})$ on $\barBun_G$; see Subsection \ref{Reduction to a trace computation on barBun_G} for an example of such a reduction step. In the next subsection we explain why some of the main results of the present article can be viewed as giving a description of precisely this complex $j_!({\Qellbar})$ of interest.

\sssec{\textbf{The complex $j_!({\Qellbar})$}}

Let $i$ denote the inclusion of the boundary
$$\barBun_G \setminus \Bun_G \ \longinto \ \barBun_G \, .$$
Motivated by the above, one would like to understand the complex $i^!j_!{\Qellbar}$. On the one hand, one would like to compute its stalks. As was already discussed in Subsection \ref{Stalks of the *-extension of the constant sheaf} above, this computation is achieved by our Theorem \ref{ij} : We introduce a natural stratification of the boundary and compute the restriction of the complex $i^!j_!{\Qellbar}$ to the strata in terms of the cohomology of the open Zastava spaces.

\medskip

On the other hand, one might also want to obtain a description of the entire complex $i^!j_!{\Qellbar}$ ``at once'', rather than just of its restrictions to the various strata of the boundary as above. Such a description can indeed also be extracted from our main results: It follows from the cartesian square in Section \ref{The definition for G=SL_2} above that the complex $i^!j_!{\Qellbar}$ on $\barBun_G \setminus \Bun_G$ is \textit{Koszul dual} to the nearby cycles perverse sheaf $\Psi$ of the degeneration $\VinBun_G$, in a sense that will be explained in detail in the future work \cite{Sch1}.
But a description of the nearby cycles as a perverse sheaf (and not only of its stalks) is in turn provided by our Theorem \ref{main theorem for nearby cycles}.

\medskip

To state one concrete application to the miraculous duality functor, we remark that from our above-mentioned results about the complex $i^!j_!{\Qellbar}$ one can extract the geometric input needed for Gaitsgory's proof that the miraculous duality functor acts as the identity on the subcategory of $D(\Bun_G)$ consisting of all cuspidal objects (\cite{G2}). The geometric input required for this proof will be supplied in \cite{Sch1}, where we return to this topic in the context of an arbitrary reductive group.

\bigskip

\ssec{An application to Drinfeld's strange invariant bilinear form}
\label{An application to Drinfeld's strange invariant bilinear form}

\mbox{} \medskip

In the article \cite{DrW} Drinfeld and Wang have introduced, for $G=\SL_2$, a novel invariant symmetric bilinear form on the space of automorphic forms. This article forms part of Drinfeld's \textit{geometric functional analysis} and \textit{geometric theory of automorphic forms}: While the article is written for an arbitrary global field $F$, the motivation for its constructions stems from the geometric Langlands program. In this subsection we discuss an application of the space $\barBun_G$ and the results of the present article to the work \cite{DrW}, which is carried out in Section \ref{An application: Computation of Drinfeld's function} of the main text.

\medskip

Let $G=\SL_2$, let $F$ be a global field, let $\BA$ denote the adele ring of $F$, and let $K$ denote the standard maximal compact subgroup of $G(\BA)$. Drinfeld's and Wang's \textit{strange bilinear form} $\FB$ on the space of $K$-finite smooth compactly supported functions on $G(\BA)/G(F)$ is defined via the formula
$$\FB(f_1,f_2) \ \ := \ \ \FB_{naive}(f_1, f_2) - \langle M^{-1} \CT(f_1), \CT(f_2) \rangle \, .$$
Here $\FB_{naive}$ denotes the usual scalar product; $M^{-1}$ denotes the inverse of the long intertwining operator $M$; and $\CT$ denotes the constant term operator. This definition is inspired by the work of Drinfeld and Gaitsgory on the miraculous duality functor in the geometric Langlands program from Subsection \ref{Relation to the geometric Langlands program} above. We refer to \cite{DrW} for how this new bilinear form ties in with the existing theory of automorphic forms; see also \cite{W}.

\medskip

The main theorem of \cite{DrW} asserts that the bilinear form $\FB$ ``comes from geometry'', as we now discuss. Let $X$ be a smooth projective curve over a finite field $\BF_q$, and as before let
$$\Delta: \ \ \Bun_G \ \longto \ \Bun_G \times \Bun_G$$
denote the diagonal morphism of $\Bun_G$. Consider the constructible $\ell$-adic complex $\Delta_*(\Qellbar)$ on $\Bun_G \times \Bun_G$, and let $f_{\textrm{Drinfeld}}$ denote \textit{Drinfeld's function}, i.e., the function on the set of isomorphism classes of $\BF_q$-points of $\Bun_G \times \Bun_G$ associated to this complex; note that since $\Delta$ is not proper, the explicit computation of this function in terms of a combinatorial formula is not immediate.

\medskip

Let $F$ denote the function field of $X$, and denote the function induced by $f_{\textrm{Drinfeld}}$ on the product $G(\BA_F)/G(F) \times G(\BA_F)/G(F)$ by $f_{\textrm{Drinfeld}}$ as well. The main theorem of \cite{DrW} then asserts that the bilinear form $\FB$ is equal to the bilinear form defined by integrating the product of two given automorphic forms multiplied with the function $f_{\textrm{Drinfeld}}$. To prove this theorem, Drinfeld and Wang required a combinatorial formula for $f_{\textrm{Drinfeld}}$. This combinatorial formula is established in Section \ref{An application: Computation of Drinfeld's function} of the present article; its deduction is based on Drinfeld's idea to compactify the diagonal $\Delta$ using the space $\barBun_G$; more precisely, the formula is deduced from one of our main theorems, Theorem \ref{ij} below. We will return to the calculation of the function $f_{\textrm{Drinfeld}}$ for an arbitrary reductive group in future work.

\medskip
\bigskip

\ssec{An application to Bernstein asymptotics}
\label{An application to Bernstein asymptotics}

\mbox{} \medskip

Some of our results in the present work (for $G=\SL_2$) can also be viewed as a step towards a geometric construction, or rather a categorification, of the \textit{Bernstein asymptotics map} appearing in the representation theory of reductive groups over non-archimedean local fields and the theory of automorphic forms (\cite{BK}, \cite{SakV}, \cite{Sak1}, \cite{Sak2}, \cite{DrW}, \cite{CY}); our results for $G=\SL_2$ thus confirm a special case of a conjecture of Sakellaridis, briefly outlined below. These results hence fit into the general framework of lifting various constructions in the classical theory to the geometric and categorical level, for example with the goal of developing a theory of character sheaves for loop groups and \textit{p}-adic groups.

\medskip

Let $G$ now be a reductive group over a non-archimedean local field. Let $N$ denote the unipotent radical of the Borel $B$ of $G$. The \textit{Bernstein asymptotics map} is a map of $G \times G$-modules
$$C^{\infty}(G) \ \longto \ C^{\infty}((G/N \times G/N^-)/T)$$
from the space of functions on the group $G$ to the space of functions on the \textit{boundary degeneration} $(G/N \times G/N^-)/T$. It can be characterized either by a universal property related to the asymptotics of matrix coefficients,
or as a composition of the \textit{orispheric transform} with the inverse of the \textit{intertwining operator} (see \cite{BK}, \cite{Sak2}).

\medskip

In the finite-dimensional setting (i.e., over an algebraically closed field instead of over a local field), a categorical construction similar to the Bernstein map has already appeared in \cite{BFO} (see also \cite{ENV}, \cite{CY}), where the authors study the \textit{twisted Harish-Chandra functor} between certain categories of D-modules. The authors show that, in this setting, the twisted Harish-Chandra functor can be realized as the functor of Verdier specialization in the De Concini-Procesi wonderful compactification. Motivated by this result, Bezrukavnikov and Kazhdan remark in their work \cite{BK}, which uses the Bernstein map to prove Bernstein's second adjointness theorem for reductive groups over non-archimedean local fields, that the Bernstein map bears some resemblance to a ``nearby cycles procedure'' in the setting over a local field, and pose the question whether these heuristics can be made precise. This question has also been raised by \cite{CY}.

\medskip

A similar prediction has been made by Sakellaridis and Venkatesh. They have constructed a more general \textit{asymptotics map} in their work \cite{SakV} on harmonic analysis on spherical varieties over non-archimedean local fields, forming part of their \textit{relative Langlands program}. More precisely, associated to certain natural degenerations of spherical varieties, they have constructed asymptotics maps on the level of functions, which reduce to the above case of Bernstein asymptotics when the spherical variety is the group itself. Sakellaridis has given a precise conjecture (see below, and see \cite{Sak2} for details) relating the asymptotics map on the level of functions to the nearby cycles functor of a global model for the degeneration.

\medskip

The results in this article indicate that the degeneration $\VinBun_G$ may be used to categorify the Bernstein map; in particular, our results confirm the conjecture of Sakellaridis when the spherical variety is the group $G=\SL_2$.

\medskip

To sketch the formulation of Sakellaridis's conjecture, we first recall that, unlike the IC-sheaf of $\VinBun_G$, the nearby cycles $\Psi_{\VinBun_G}$ of $\VinBun_G$ do possess the \textit{factorization property}: The stalks of $\Psi_{\VinBun_G}$ decompose into a tensor product of \textit{local factors}, each of which is associated with a point on the curve $X$ occurring in the defect divisor. Next, recall that our main results provide an explicit description of $\Psi_{\VinBun_G}$ for $G=\SL_2$.
Working over a finite field and with $\ell$-adic sheaves, one easily computes (for example from the formulas in Theorem \ref{main theorem for nearby cycles} or Theorem \ref{ij}) the function corresponding to the perverse sheaf $\Psi_{\VinBun_G}$ under the sheaf-function dictionary. Due to the factorization of $\Psi_{\VinBun_G}$, this function splits into a product of local factors as well. Our formulas for $\Psi_{\VinBun_G}$ in the present article then imply the special case $G=\SL_2$ of Sakellaridis's conjecture:

\begin{proposition}
\label{Bernstein via Psi}
For $G=\SL_2$, the local factors of the function corresponding to the nearby cycles perverse sheaf $\Psi_{\VinBun_G}(\IC_{\Bun_G})$ agree with the corresponding Bernstein asymptotics of the canonical Schwartz function.
\end{proposition}

This proposition can be established by comparing our formulas for the local factors with the explicit formulas given by Sakellaridis for the Bernstein asymptotics of the canonical Schwartz function on the group $G$ (\cite[Section 6]{Sak1}).
The use of the global object $\VinBun_G$ in this local context is completely analogous to the use of a global curve in the study of the affine Grassmannian $\Gr_G$ (\cite{BD2}, \cite{MV}), or the use of the Zastava spaces in the study of the semi-infinite flag variety (\cite{FFKM}, \cite{BFGM}).

\medskip
Since we can already prove Proposition \ref{Bernstein via Psi} above for an arbitrary reductive group $G$ by similar methods as for $G=\SL_2$, its precise formulation and proof will be discussed in the future article \cite{Sch1}, where the space $\VinBun_G$ will be studied for an arbitrary reductive group $G$. We plan to return to the full categorification of the Bernstein map in separate work in the future.

\bigskip

\ssec{Proofs via local models}
\label{Proofs via local models}

\mbox{} \medskip

A powerful technique in the study of singular moduli spaces in the geometric Langlands program is to construct \textit{local models} which feature the same singularities but have the advantage of being \textit{factorizable} in the sense of Beilinson and Drinfeld (\cite{BD1}, \cite{BD2}). One common application of the factorization property is to perform inductive calculations of factorizable sheaves on the moduli spaces in question. A prime example of this technique, and a major influence on the present article, is the computation of the IC-sheaf of Drinfeld's relative compactifications $\barBun_B$ by Braverman, Finkelberg, Gaitsgory, and Mirkovic in \cite{BFGM}; in this case the local models are the Zastava spaces introduced by Drinfeld, Feigin, Finkelberg, Kuznetsov and Mirkovic (\cite{FFKM}, \cite{BFGM}).

\medskip

\sssec{\textbf{Local models for $\VinBun_G$}}
To prove our main theorems we construct certain local models $(Y^n)_{n \in \BZ_{\geq 1}}$ for the degeneration $\VinBun_G$, which themselves form one-parameter families
$$v: \ Y^n \ \longto \ \BA^1 \, .$$
Their relationship with $\VinBun_G$ is completely analogous to the relationship between the Zastava spaces and Drinfeld's $\barBun_B$ (\cite{BFGM}, \cite{BG2}):
Broadly speaking, the local model $Y^n$ features the same singularities as the open substack $\sideset{_{\leq n}}{_G}\VinBun$ of defect $\leq n$; hence the validity of our main theorems for $\VinBun_G$ is equivalent to the validity of the corresponding assertions for the local models $Y^n$ for all $n \geq 1$. For example, we establish our nearby cycles theorem for $\VinBun_G$ by establishing it for all local models $Y^n$. To prove it for any given $Y^n$, we proceed by induction on $n$, making use of the following factorization property of our local models:

\medskip

\sssec{\textbf{Factorization in families}}
The main difference between our local models $Y^n$ and the Zastava spaces is that our local models actually do not factorize, but rather ``factorize in families'', i.e., the fibers of the map $v: Y^n \to \BA^1$ are factorizable in compatible ways. In fact, our local models can be also viewed as canonical one-parameter ``Vinberg degenerations'' of the Zastava spaces. The fact that factorization holds only in families is a natural explanation for the fact that the IC-sheaf of the total space $\VinBun_G$ does not factorize. The ``factorization in families'' of our local models however implies that the nearby cycles sheaf does factorize, making it amenable to an inductive computation.

\medskip

\sssec{\textbf{The local models $Y^n$ for small $n$}}
As another ingredient of our proofs, we mention the possibility to describe our local models $Y^n$ in very concrete terms; this is a special feature of the case $G = \SL_2$ considered in this article. For example, we construct natural embeddings of our local models $Y^n$ into certain products of Beilinson-Drinfeld affine Grassmannians to derive explicit equations in coordinates. In the simplest case of defect degree $\leq 1$, which is simultaneously the base case of the inductive proof, these formulas show that the resulting one-parameter degeneration
$$Y^1 \ \longto \ \BA^1$$
has the same singularities as the Picard-Lefschetz family $x \cdot y = t$ of hyperbolas degenerating to a node. Similarly, a somewhat more involved analysis of the equations in the case of defect degree $\leq 2$ can be used to prove the appearance of the Picard-Lefschetz oscillators in the formula for the nearby cycles. Although it is possible to give more abstract and possibly quicker proofs of these statements, we have tried in the current article to include a concrete proof when possible, and have postponed using more abstract methods to our future work \cite{Sch1}, \cite{Sch2} which deals with the case of an arbitrary reductive group $G$.

\medskip

We refer the reader to Section \ref{Nearby cycles} for a more detailed outline of the structure of the proof of the nearby cycles theorem.

\bigskip

\ssec{Structure of the article}

\mbox{} \medskip

We now briefly discuss the content of the individual sections.

\medskip

In Section \ref{The Drinfeld-Lafforgue-Vinberg compactification} we define the compactification $\barBun_G$ and the degeneration $\VinBun_G$ and explain their relationship. We then focus on $\VinBun_G$ and introduce the aforementioned defect stratification. To construct the stratification, but also to prepare for the statement of the main theorem about nearby cycles, we use Drinfeld's and Laumon's relative compactifications $\barBun_B$ to compactify the inclusion maps of the individual strata.

\medskip

In Section \ref{Statement of main theorems} we first recall some facts about nearby cycles, the weight-monodromy filtration, the action of the Lefschetz-$\sl_2$, and the relationship between the nearby cycles and the IC-sheaf. We then define the Picard-Lefschetz oscillators and state our main theorems about the nearby cycles, the IC-sheaf, and the stalks of the $*$-extension of the constant sheaf.

\medskip

In Sections \ref{local models for VinBun} and \ref{Geometry of the local models} we first construct the local models for $\VinBun_G$ and restate the analogous theorem about nearby cycles in this context. We then study their geometry: We discuss the aforementioned factorization in families, and construct embeddings into a product of Beilinson-Drinfeld affine Grassmannians. We use these embeddings on the one hand to construct $\BG_m$-actions which contract the local models onto the strata of maximal defect, and on the other hand to derive the explicit equations for the local models mentioned above.

\medskip

In Section \ref{Nearby cycles} we give the proof of the main theorem about nearby cycles. In Section \ref{Intersection cohomology} we deduce from it the description of the IC-sheaf. In Section \ref{An application: Computation of Drinfeld's function} we provide the aforementioned application on the level of functions related to Drinfeld's and Wang's strange bilinear form on the space of automorphic forms.

\bigskip

\ssec{Notation and conventions}
\label{Notation and conventions}

\mbox{} \medskip

Since we will use a formalism of mixed sheaves, we for concreteness choose the following setup: We assume the curve $X$ is defined over a finite field, and work with Weil sheaves over the algebraic closure of the finite field. For a scheme or stack $Y$, we will denote by $D(Y)$ the derived category of constructible $\Qellbar$-sheaves on $Y$. We will frequently abuse terminology and refer to its objects as sheaves. We fix once and for all a square root $\Qellbar(\tfrac{1}{2})$ of the Tate twist $\Qellbar(1)$. We normalize all IC-sheaves to be pure of weight $0$; for example, on a smooth variety $Y$ the IC-sheaf is equal to $\Qellbar[\dim Y](\tfrac{1}{2} \dim Y)$. Given a local system $E$ on a smooth dense open subscheme $U$ of a scheme~$Y$, we refer to the intermediate extension of the shifted and twisted local system $E[\dim Y](\tfrac{1}{2} \dim Y)$ to $Y$ as the IC-extension of $E$. Our conventions for nearby cycles are stated in Subsection \ref{Conventions for nearby cycles} below.

\medskip

Although we restrict to the case $G=\SL_2$ throughout the article, we will continue to use the symbol $G$; we denote by $B$ and $B^-$ the standard Borel and opposite Borel subgroups of $G = \SL_2$, and by $T$ the standard maximal torus. The arrow $F \longintointo E$ denotes the inclusion of a subbundle $F$ of a vector bundle $E$; a usual injective arrow $F \longinto E$ stands for an injection of coherent sheaves.

\medskip

We will indicate the restriction of a space or a sheaf to a ``disjoint locus'' by the symbol $\circ$, whenever there is no confusion about what the disjointness is referring to. For example, we denote by
$$X^{(n_1)} \stackrel{\circ}{\times} X^{(n_2)}$$
the open subset of the product $X^{(n_1)} \times X^{(n_2)}$ of symmetric powers of the curve $X$ consisting of those pairs of effective divisors with disjoint support, and refer to it as the disjoint locus of $X^{(n_1)} \times X^{(n_2)}$. Similarly, for objects $F_1 \in D(X^{(n_1)})$ and $F_2 \in D(X^{(n_1)})$ we denote by
$$F_1 \ \overset{\circ}{\boxtimes} \ F_2$$
the restriction of the exterior product $F_1 \boxtimes F_2$ to the disjoint locus of the above product. Finally, we denote by $\overset{\circ}{X} {}^{(n)}$ the open subscheme of $X^{(n)}$ obtained by removing all diagonals, i.e., the open subscheme consisting of all effective divisors of the form $\sum_{i=1}^n x_i$ with all $x_i$ distinct.

\medskip
\bigskip

\ssec{Acknowledgements}

\mbox{} \medskip

First and foremost I would like to thank my doctoral advisor Dennis Gaitsgory, for suggesting to study the compactification $\barBun_G$ as well as for his invaluable support throughout this project: The present article has benefitted from his suggestions and from conversations with him in too many ways to mention and would not exist without him. I would furthermore like to thank Vladimir Drinfeld for generously sharing his ideas and for his continuing support; the definition of the main object of this paper, as well as many of the ideas and techniques in this area of mathematics, are of course due to him. Finally, I would like to thank Sophie Morel, Anand Patel, Sam Raskin, Yiannis Sakellaridis, and Jonathan Wang for useful conversations related to this article.

\bigskip\bigskip\bigskip

\section{The Drinfeld-Lafforgue-Vinberg compactification}
\label{The Drinfeld-Lafforgue-Vinberg compactification}

\ssec{The definition of the degeneration and the compactification}

\sssec{Definition of $\VinBun_G$}

We now define the Drinfeld-Lafforgue-Vinberg degeneration $\VinBun_G$ for $G=\SL_2$. An $S$-point of $\VinBun_G$ consists of the data of two vector bundles $E_1$, $E_2$ of rank $2$ on $X \times S$, together with trivializations of their determinant line bundles $\det E_1$ and $\det E_2$, and a map of coherent sheaves
$$\varphi: \ E_1 \ \longto \ E_2$$
satisfying the following condition: For each geometric point $\bar{s} \to S$ we require that the map
$$\varphi|_{X \times \bar{s}} : \ \ E_1|_{X \times \bar{s}} \ \longto \ E_2|_{X \times \bar{s}}$$
is not the zero map; in other words, the map $\varphi|_{X \times \bar{s}}$ is required to not vanish generically on the curve $X \times \bar{s}$.

\medskip

The stack $\VinBun_G$ admits a natural map
$$v: \ \VinBun_G \ \longto \ \BA^1$$
which sends an $S$-point as above to the determinant
$$\det \varphi \ \in \ \Gamma(\CO_{X \times S}) \ = \ \Gamma(\CO_S) \ = \ \BA^1(S).$$
It will follow from Lemma \ref{torsor} below that $\VinBun_G$ is indeed an algebraic stack.

\bigskip

\sssec{Definition of $\barBun_G$}

We now give the definition of the Drinfeld-Lafforgue-Vinberg compactification $\barBun_G$ for $G = \SL_2$, following Drinfeld.
An $S$-point of $\barBun_G$ consists of the following data: Two vector bundles $E_1$ and $E_2$ of rank~$2$ on $X \times S$ together with trivializations of their determinant line bundles $\det E_1$ and $\det E_2$; a line bundle $L$ on $S$; and a map of coherent sheaves
$$\varphi: \ E_1 \ \longto \ E_2 \otimes \text{pr}^*L,$$
where $\text{pr}^*L$ denotes the pullback of $L$ along the projection map $\text{pr}: X \times S \to S$. Similarly to above we require the above data to satisfy the following condition: For each geometric point $\bar{s} \to S$ we require that the map
$$\varphi|_{X \times \bar{s}} : \ \ E_1|_{X \times \bar{s}} \ \longto \ (E_2 \otimes p^*L)|_{X \times \bar{s}}$$
is not the zero map; in other words, the map $\varphi|_{X \times \bar{s}}$ is required to not vanish generically on the curve $X \times \bar{s}$.

\medskip

Similarly to $\VinBun_G$, the stack $\barBun_G$ admits a natural map
$$\bar v: \ \barBun_G \ \longto \ \BA^1/\BG_m$$
to the quotient of $\BA^1$ by $\BG_m$ with respect to the quadratic action, defined by remembering only the line bundle $L$ together with the global section of its square
$$\det \varphi: \ \CO_S \ \longto \ L^{\otimes 2}.$$

\medskip

\sssec{Compactifying the diagonal of $\Bun_G$}
\label{Compactifying the diagonal}
We now explain why one may call the stack $\barBun_G$ a compactification (though it is not literally so; see below). To do so, let
$$b: \Bun_G \ \longto \ \barBun_G \ \ \ \ \ \text{and} \ \ \ \ \ \bar \Delta: \ \barBun_G \ \longto \ \Bun_G \times \Bun_G$$
denote the natural maps. If the characteristic is not equal to $2$, the map $b$ is an etale map onto its image in $\barBun_G$ of degree $2$. If the characteristic is equal to $2$, the map $b$ is radicial onto its image in $\barBun_G$. On the one hand, the diagonal morphism of $\Bun_G$
$$\Delta: \ \Bun_G \ \longto \ \Bun_G \times \Bun_G$$
naturally factors as
$$\xymatrix@+10pt{
\Bun_G \ar^{b}[r] \ar@/_-2pc/[rr]^{\Delta} & \barBun_G \ar^{\bar\Delta \ \ \ \ \ }[r] & \Bun_G \times \Bun_G \, .
}$$

\medskip

\noindent On the other hand, we have the following lemma, which can be easily checked from the definitions:

\medskip

\begin{lemma}
\label{delta bar}
The map
$$\bar\Delta: \ \barBun_G \ \longto \ \Bun_G \times \Bun_G$$
is schematic and proper. The fiber of the map $\bar\Delta$ over a point $(E_1, E_2)$ in $\Bun_G \times \Bun_G$ is equal to the projectivization $\BP (\Hom(E_1, E_2))$ of the vector space of all homomorphisms of coherent sheaves $E_1 \to E_2$.
\end{lemma}

\bigskip

\sssec{The relation between $\VinBun_G$ and $\barBun_G$}
\label{The relation}

Next consider the natural map
$$\VinBun_G \ \longto \ \barBun_G$$
defined by taking $L$ to be the trivial line bundle $\CO_S$ on $S$ and by not changing the remaining data. Then the square
$$\xymatrix@+10pt{
\VinBun_G \ar[r] \ar[d]_{v  } &    \barBun_G \ar[d]^{\bar{v}}          \\
\BA^1 \ar[r]                            &    \BA^1/\BG_m              \\
}$$
commutes, where the bottom arrow is the natural projection map. In fact one sees directly from the definitions:

\medskip

\begin{lemma}
\label{torsor}
The above square is cartesian. Thus the map
$$\VinBun_G \ \longto \ \barBun_G$$
is a $\BG_m$-bundle, and in particular the stacks $\barBun_G$ and $\VinBun_G$, as well as the maps $v$ and ${\bar{v}}$, are smooth-locally isomorphic.
\end{lemma}

\medskip

Note that Lemmas \ref{torsor} and \ref{delta bar} imply that $\barBun_G$ and $\VinBun_G$ are indeed algebraic stacks.

\medskip

\begin{remark}
Because of Lemma \ref{torsor} above, we will restrict our attention to the Drinfeld-Lafforgue-Vinberg degeneration $\VinBun_G$ for the entire article. The study of the singularities of $\barBun_G$, or the study of the map $\bar v$, immediately reduces to the study of $\VinBun_G$ and the study of the map $v$ due to the cartesian square of Lemma \ref{torsor}.
\end{remark}

\bigskip

\ssec{The $G$-locus, the $B$-locus, and the defect-free locus}
\label{The $G$-locus and the $B$-locus}

\mbox{} \medskip

Consider again the natural map $v: \VinBun_G \to \BA^1$, which on the level of $k$-points is defined by sending a triple $(E_1, E_2, \varphi)$ to the determinant $\det \varphi$ of the map $\varphi$.
We will call the fiber of the map $v$ over $0 \in \BA^1$ the $B$-\textit{locus} of $\VinBun_G$, and denote it by $\VinBun_{G,B}$. We will call the inverse image of $\BA^1 \setminus \{0\}$ under~$v$ the $G$-\textit{locus} of $\VinBun_G$ and denote it by $\VinBun_{G,G}$. Thus, on the level of $k$-points, the $G$-locus $\VinBun_{G,G}$ consist precisely of those triples $(E_1, E_2, \varphi)$ for which the map $\varphi$ is an isomorphism. Similarly, the $B$-locus $\VinBun_{G,B}$ consist precisely of those triples for which the determinant
$$\det \varphi: \ \CO_X \ \longto \ \CO_X$$
equals the zero map, i.e., for which the induced maps on fibers
$$\varphi|_x: \ E_1|_x \ \longto \ E_2|_x$$
have rank $\leq 1$ at every point $x \in X$. In other words, the $B$-locus consists of those triples for which the map $\varphi$ has generic rank $1$ on the curve $X$.

\medskip

The $G$-locus of $\VinBun_G$ in fact naturally decomposes as a product:

\medskip

\begin{lemma}
\label{G-locus product decomposition}
The natural map
$$\VinBun_{G,G} \ \ \ \longto \ \ \ \Bun_G \ \times \ (\BA^1 \setminus \{0\})$$
$$(E_1, E_2, \varphi) \ \ \ \longmapsto \  \  \ (E_1, \det \varphi) \ \ \ \ \ \ \ \ \ \ \  $$
is an isomorphism.
\end{lemma}

\medskip

\begin{proof}
By definition the $G$-locus of $\VinBun_G$ is equal to the mapping stack
$$\Maps(X \, , \, \SL_2 \backslash \GL_2 / \SL_2)$$
parametrizing maps from the curve $X$ into the quotient stack $\SL_2 \backslash \GL_2 / \SL_2$ for the action by left and right translations. But after identifying the quotient $\GL_2/\SL_2$ for the action from the right with $\BG_m$ via the determinant map, we see that the remaining action of $\SL_2$ from the left on this quotient is trivial, and the result follows.
\end{proof}

\bigskip

\sssec{The defect-free locus}
\label{The defect-free locus}

We now define an open substack
$$\sideset{_0}{_G}\VinBun \ \subset \ \VinBun_G$$
which will be referred to as the \textit{defect-free locus} of $\VinBun_G$. This terminology is in line with the notion of defect defined in Subsection \ref{The defect stratification} below. To define the open substack we require the triple $(E_1, E_2, \varphi)$ to additionally satisfy the following condition: For each $\bar{s} \to S$ we require the map
$$\varphi|_{X \times \bar{s}} : \ E_1|_{X \times \bar{s}} \longto (E_2)|_{X \times \bar{s}}$$
to not vanish at any point of the curve $X \times \bar{s}$. In particular the defect-free locus contains the $G$-locus $\VinBun_{G,G}$.

\medskip

\begin{proposition}
The restriction of the map $v: \VinBun_G \to \BA^1$ to $\sideset{_0}{_G}\VinBun$ is smooth; in particular the open substack $\sideset{_0}{_G}\VinBun$ is smooth.
\end{proposition}

\medskip

\begin{proof}
Let $\Mat_{2 \times 2}^{\geq 1}$ denote the variety of $2 \times 2$ matrices over $k$ of rank $\geq 1$, and abbreviate
$$Q \ := \ \SL_2 \backslash \Mat_{2 \times 2}^{\geq 1} / \SL_2 \, .$$
By definition, the defect-free locus$\sideset{_0}{_G}\VinBun$ is equal to the mapping stack $\Maps(X,Q)$, and the above map
$$v: \ \ \sideset{_0}{_G}\VinBun \ = \ \Maps(X,Q) \ \ \longto \ \ \Maps(X,\BA^1) \ = \ \BA^1$$
is equal to the map induced on mapping stacks by the determinant map
$$d: \ Q \ \longto \ \BA^1 \, .$$
The above mapping stacks are objects of classical algebraic geometry; we will now consider the corresponding derived mapping stacks, which are objects of derived algebraic geometry. In the present context, this should be considered as nothing more than a convenient formalism when dealing with tangent complexes. We denote the derived mapping stacks and the map between them by
$$v_{der}: \ \ \Maps_{der}(X, Q) \ \longto \ \Maps_{der}(X, \BA^1) \, ,$$
and will show that the map $v_{der}$ is smooth. Since the base change of $v_{der}$ along the natural map
$$\BA^1 \ = \ \Maps(X, \BA^1) \ \longto \ \Maps_{der}(X, \BA^1)$$
agrees with the map $v$ between classical mapping stacks, establishing that $v_{der}$ is smooth suffices to prove the lemma.

\medskip

To prove that the map $v_{der}$ is indeed smooth, we will show that the fiber of its relative tangent complex
$$T_{rel} \ \ := \ \ T_{Maps_{der}(X,Q) \, / \, Maps_{der}(X,\BA^1)}$$
at any geometric point of $\Maps_{der}(X,Q)$ is concentrated in degrees $-1$ and~$0$. To do so, let
$$T_{Q/\BA^1} \ \longto \ T_Q \ \longto \ d^*T_{\BA^1} \ \stackrel{+1}{\longto}$$
denote the usual tangent complex triangle on the double quotient $Q$ associated to the determinant map $d: Q \to \BA^1$. We claim that the complex $T_{Q/\BA^1}$ is concentrated in degree $-1$. Indeed, since the map $d$ is smooth, it suffices to show that the tangent complex of each fiber of $d$ is concentrated in degree~$-1$. But since the action of $\SL_2 \times \SL_2$ on any fiber of the determinant map
$$\Mat_{2 \times 2}^{\geq 1} \ \longto \ \BA^1$$
is transitive with smooth stabilizers, the fibers of the map $d$ are classifying stacks of smooth groups, proving the claim that $T_{Q/\BA^1}$ is concentrated in degree $-1$. We can now show that the fiber of the relative tangent complex $T_{rel}$ at any given geometric point $f: X \to Q$ of $\Maps_{der}(X,Q)$ is concentrated in degrees $-1$ and $0$. Namely, since we are using derived mapping stacks, taking the fiber of the usual tangent complex triangle
$$T_{rel} \ \longto \ T_{Maps_{der}(X,Q)} \ \longto \ v_{der}^*T_{\Maps_{der}(X, \BA^1)} \ \stackrel{+1}{\longto}$$
on $\Maps_{der}(X,Q)$ at the point $f$ yields a triangle
$$T_{rel}|_f \ \longto \ R\Gamma(X, f^*T_Q) \ \longto \ R\Gamma(X, f^*d^*T_{\BA^1}) \ \stackrel{+1}{\longto} \, .$$
But pulling back the tangent complex triangle on $Q$ above along the map $f$ and then applying $R\Gamma(X,-)$ yields the triangle
$$R\Gamma(X, f^*T_{Q/\BA^1}) \ \longto \ R\Gamma(X, f^*T_Q) \ \longto \ R\Gamma(X, f^*d^*T_{\BA^1}) \ \stackrel{+1}{\longto} \, ,$$
whose second map agrees with the second map of the previous triangle. Thus
$$T_{rel}|_f \ = R\Gamma(X, f^*T_{Q/\BA^1}) \, ,$$
and hence $T_{rel}|_f$ is indeed concentrated in degrees $-1$ and $0$ as desired.
\end{proof}

\medskip

\sssec{Remarks about the Vinberg semigroup}

The terminology ``$G$-locus'' and ``$B$-locus'' stems from the more general context of the Vinberg semigroup: The Vinberg semigroup of a reductive group admits a natural stratification indexed by the parabolic subgroups of the reductive group; this stratification induces a stratification of the degeneration $\VinBun_G$, which specializes to the stratification into the $G$-locus and the $B$-locus in the case of $G=\SL_2$. For further motivation for this notation see Subsection \ref{The defect stratification} below.

\bigskip

\ssec{The defect stratification}
\label{The defect stratification}

\sssec{Definition of the defect}
\label{Definition of the defect}

The $B$-locus $\VinBun_{G,B}$ possesses a natural stratification by the following notion of \textit{defect}. Let $(E_1, E_2, \varphi)$ be a $k$-point of $\VinBun_{G,B}$. Then the map $\varphi$ admits a unique factorization
$$E_1 \ \longonto \ M_1 \ \longinto \ M_2 \ \longintointo \ E_2$$
where $M_1$ and $M_2$ are line bundles on the curve $X$, the first map is surjective, the middle map is an injection of coherent sheaves, and the last map is a subbundle map. We call the effective divisor on the curve $X$ corresponding to the injection $M_1 \into M_2$ the \textit{defect divisor}; its degree will be called the \textit{defect}.

\medskip

\sssec{Stratification by defect}
\label{Stratification by defect}

We now stratify the $B$-locus $\VinBun_{G,B}$ into loci of constant defect, according to the factorization of the map $\varphi$ above. We first set up the notation. Recall that the moduli stack $\Bun_B$ classifying $B$-bundles on $X$ admits a natural map
$$\mathfrak{q}: \Bun_B \longto \Bun_T$$
which induces a bijection between the sets of connected components
$$\pi_0(\Bun_B) \ = \ \pi_0(\Bun_T) \ = \ \BZ.$$
Let $\Bun_{T,n}$ denote the connected component of $\Bun_T$ consisting of degree~$n$ line bundles, and define $\Bun_{B,n}$ and $\Bun_{B^-,n}$ in the same way. Furthermore, let $k \in \BZ_{\geq 0}$ be a non-negative integer and let $X^{(k)}$ denote the $k$-th symmetric power of the curve $X$.

\medskip

Next define a map
$$X^{(k)} \times \Bun_B \ \longto \ \Bun_T$$
as the composition
$$X^{(k)} \times \Bun_B \ \stackrel{\id \times \mathfrak{q}}{\longto} \ X^{(k)} \times \Bun_T \ \stackrel{\text{twist}}{\longto} \ \Bun_T \, ,$$
where the second map sends a pair $(D, L)$ consisting of an effective divisor~$D$ and a line bundle $L$ to the twisted line bundle $L(-D)$.

\medskip

Using the previous map we now form the fiber product
$$\Bun_{B^-} \ \underset{\Bun_T}{\times} \ \bigl( \, X^{(k)} \, \times \Bun_B \bigr) \, ,$$
from which we will now construct a map to the $B$-locus $\VinBun_{G,B}$. By definition, a point of this fiber product consists of a $B^-$-bundle $E_1 \longonto M_1$, an effective divisor $D$, a $B$-bundle $M_2 \longintointo E_2$, and an identification $M_1 \cong M_2(-D)$. Thus, given two integers $n_1, n_2$ with $n_1 = n_2 - k$ we can define a map
$$f_{n_1, k, n_2}: \ \Bun_{B^-, n_1} \ \underset{\Bun_T}{\times} \ \bigl( \, X^{(k)} \, \times \Bun_{B,n_2} \bigr) \ \longto \ \VinBun_{G,B}$$
by sending the above point to the triple $(E_1, E_2, \varphi)$ where the map $\varphi$ is defined as the composition
$$\varphi: \ \ E_1 \ \longonto \ M_1 =  M_2(-D) \ \longinto \ M_2 \ \longintointo \ E_2 \, .$$

\medskip

We then have the following stratification of the $B$-locus $\VinBun_{G,B}$:

\medskip

\begin{proposition}
\label{defect stratification proposition}
\begin{itemize}
\item[]
\item[]
\item[(a)] The map $f_{n_1, k, n_2}$ is a locally closed immersion and thus defines an isomorphism onto a smooth locally closed substack
$$\sideset{_{(n_1, k, n_2)}}{_{G,B}}\VinBun \ \longinto \ \VinBun_{G,B} \, .$$
\item[]
\item[(b)] On the level of $k$-points, the $B$-locus $\VinBun_{G,B}$ is equal to the disjoint union
$$\VinBun_{G,B} \ = \ \bigcup_{(n_1, k, n_2)} \ \sideset{_{(n_1, k, n_2)}}{_{G,B}}\VinBun \, ,$$
where the union runs over all triples $(n_1, k, n_2)$ with $n_1, n_2 \in \BZ$, $k \in \BZ_{\geq 0}$, and $n_1 = n_2 - k$.
\item[]
\item[(c)] On the level of $k$-points, the closure of a stratum $\sideset{_{(n_1, k, n_2)}}{_{G,B}}\VinBun$ is equal to the union of strata

$$\ \ \ \ \ \ \sideset{_{(n_1, k, n_2)}}{_{G,B}}\barVinBun \ \ = \ \ \bigcup_{\substack{d_1 \geq 0 \\ d_2 \geq 0}} \ \sideset{_{(n_1 - d_1, k + d_1 + d_2, n_2 + d_2)}}{_{G,B}}\VinBun \, .$$
\item[]
\item[(d)]
Let $k \in \BZ_{\geq 0}$. Then the locus in $\VinBun_{G,B}$ obtained by requiring that the defect is at most $k$ naturally forms an open substack
$$\sideset{_{\leq k}}{_{G,B}}\VinBun \ \ \subset \ \ \VinBun_{G,B} \, .$$
\item[]
\item[(e)]
The union of all strata $\sideset{_{(n_1, k, n_2)}}{_{G,B}}\VinBun$ of fixed defect $k \in \BZ_{\geq 0}$ naturally forms a locally closed substack 
$$\sideset{_k}{_{G,B}}\VinBun \ \ \longinto \ \ \VinBun_{G,B}$$
which is isomorphic as stacks to the disjoint union
$$\sideset{_k}{_{G,B}}\VinBun \ \ = \ \ \coprod_{n_1, n_2} \sideset{_{(n_1, k, n_2)}}{_{G,B}}\VinBun \, .$$
\end{itemize}
\end{proposition}

\bigskip

We will prove Proposition \ref{defect stratification proposition} in Subsection \ref{Proof of stratification results} below, using certain compactifications $\bar{f}_{n_1, k, n_2}$ of the maps $f_{n_1, k, n_2}$ that we introduce next. A posteriori, these compactified maps are in fact resolutions of singularities of the strata closures (see Corollary \ref{resolution of singularities} below).

\bigskip
\ssec{Compactifying the maps $f_{n_1, k, n_2}$}
\label{Compactifying the maps $f_{n_1, k, n_2}$}

\sssec{Overview}
The goal of this subsection is to compactify the maps $f_{n_1, k, n_2}$ introduced above. These compactifications will be used to to prove Proposition \ref{defect stratification proposition} above, and are constructed using the relative compactifications $\barBun_B$ of Drinfeld and Laumon. Since the compactifications of the maps $f_{n_1, k, n_2}$ are also used in the description of the nearby cycles sheaf in Section \ref{Statement of main theorems} below, we begin with a brief review of the relative compactifications of Drinfeld and Laumon.

\medskip

Recall first that the map $\Bun_B \to \Bun_G$ is schematic but not proper; relative compactifications have been defined by G. Laumon for $G = \GL_n$ (\cite{Lau}) and by V. Drinfeld for an arbitrary reductive group $G$ (\cite{BG1}), and have been of great importance in the geometric Langlands program. For $G=\GL_n$ and $n \geq 2$, Laumon's compactification and Drinfeld's compactification differ. However, in the case of interest $G = \SL_2$ of the present paper, the two compactifications agree; we will denote them by $\barBun_B$. We now recall the definition of $\barBun_B$ for $G = \SL_2$ and then use it to compactify the maps $f_{n_1, k, n_2}$ from Subsection \ref{Stratification by defect} above. For more details on $\barBun_B$ we refer the reader to \cite{BG1}.

\medskip

\sssec{Definition of $\barBun_B$}
Let $G = \SL_2$. An $S$-point of $\barBun_B$ consists of the data of a vector bundle $E$ of rank $2$ on $X \times S$ with trivialized determinant, a line bundle $L$ on $X \times S$, and an injection of coherent sheaves $L \longinto E$ which remains injective after being restricted to $X \times \bar{s}$ for any geometric point $\bar{s} \to S$. The definition of $\barBun_{B^-}$ is analogous.

\medskip

\sssec{Basic properties}
\label{Basic properties}

The open substack of $\barBun_B$ obtained by requiring that the above injection of sheaves is a subbundle map is naturally identified with $\Bun_B$ and is dense in $\barBun_B$. Furthermore, the maps $\Bun_B \to \Bun_G$ and $\Bun_B \to \Bun_T$ naturally extend to $\barBun_B$, and the extended map $\barBun_B \to \Bun_G$ is schematic and proper when restricted to connected components of $\barBun_B$. Finally, for $G = \SL_2$ the map $\barBun_B \to \Bun_T$ is in fact smooth.

\medskip

\sssec{Stratification of $\barBun_B$}
\label{Stratification of barBun_B}

The stack $\barBun_B$ possesses the following stratification. For a connected component $\barBun_{B,n}$ with
$$n \ \in \ \BZ \ = \ \pi_0(\Bun_B) \ = \ \pi_0(\barBun_B)$$
and an integer $k \in \BZ_{\geq 0}$, consider the map
$$X^{(k)} \times \, \Bun_{B,n+k} \ \, \longto \ \, \barBun_{B,n}$$
defined as
$$(D, L \longintointo E) \ \ \ \longmapsto \ \ \ (L(-D) \into L \longintointo E) \, .$$
This map is in fact a locally closed immersion, and as $k$ ranges over $\BZ_{\geq 0}$ the corresponding locally closed substacks stratify $\barBun_{B,n}$:

$$\barBun_{B,n} \ \ = \ \ \bigcup_{k \in \BZ_{\geq 0}} \ (X^{(k)} \times \Bun_{B,n+k})$$

\noindent Finally, note that the map
$$X^{(k)} \times \Bun_B \ \longto \ \Bun_T$$
from Subsection \ref{Stratification by defect} above is in fact equal to the composition
$$X^{(k)} \times \, \Bun_B \ \, \longto \ \, \barBun_B \ \stackrel{\bar{\mathfrak{q}}}{\longto} \ \Bun_T \, .$$

\bigskip

\sssec{Compactifying the maps $f_{n_1, k, n_2}$}
\label{compactifying the maps sssec}

We now define the above-mentioned compactifications $\bar f_{n_1, k, n_2}$ of the maps $f_{n_1, k, n_2}$ from Subsection \ref{Stratification by defect} above. To do so, observe first that the maps
$$X^{(k)} \times \, \Bun_{B,n+k} \ \, \longto \ \, \barBun_{B,n}$$
from Subsection \ref{Stratification of barBun_B} above naturally extend to maps
$$X^{(k)} \times \, \barBun_{B,n+k} \ \, \longto \ \, \barBun_{B,n} \, .$$
We can therefore enlarge the fiber product from Subsection \ref{Stratification by defect} by replacing $\Bun_B$ and $\Bun_{B^-}$ by $\barBun_B$ and $\barBun_{B^-}$, and define the compactified map

$$\bar{f}_{n_1, k, n_2}: \ \barBun_{B^-, n_1} \ \underset{\Bun_T}{\times} \ \bigl( \, X^{(k)} \, \times \barBun_{B,n_2} \bigr) \ \longto \ \VinBun_{G,B}$$

\medskip

\noindent in the exact same fashion as the map $f_{n_1, k, n_2}$. We then have:

\medskip
\medskip

\begin{lemma}
\label{finiteness of compactified maps}
The map $\bar{f}_{n_1, k, n_2}$ is finite.
\end{lemma}

\medskip

\begin{proof}
We first show that the map is quasifinite. This can easily be deduced from the definitions and from the stratification of $\barBun_B$ in Subsection \ref{Stratification of barBun_B}, as follows. Consider the induced stratification of the fiber product
$$\barBun_{B^-, n_1} \ \underset{\Bun_T}{\times} \ \bigl( \, X^{(k)} \, \times \barBun_{B,n_2} \bigr)$$
with the strata
$$\bigl( \Bun_{B^-, n_1 - d_1} \times \, X^{(d_1)} \bigr) \ \underset{\Bun_T}{\times} \ \bigl( \, X^{(k)} \, \times X^{(d_2)} \times \Bun_{B,n_2 + d_2} \big) \, ,$$
where the integers $d_1, d_2 \in \BZ_{\geq 0}$ are varying.
We claim that the fiber over any $k$-point of $\VinBun_{G,B}$ can meet at most finitely many of the above strata. Indeed, any $k$-point of $\VinBun_{G,B}$ admits a unique factorization
$$\varphi: \ E_1 \ \longonto \ M_1 \ \longinto \ M_2 \ \longintointo \ E_2$$
as in Subsection \ref{Definition of the defect} above; if $m \in \BZ_{\geq 0}$ denotes its defect, then one sees from the definition of $\bar{f}_{n_1, k, n_2}$ that only the strata with
$$d_1 + k + d_2 \ = \ m$$
can meet its fiber.

\medskip

Hence it suffices to prove that $\bar{f}_{n_1, k, n_2}$ is quasifinite when restricted to any such stratum; this follows from the unique factorization of the map $\varphi$ above together with the fact that the map
$$X^{(d_1)} \times X^{(k)} \times X^{(d_2)} \ \longto \ X^{(d_1 + k + d_2)}$$
defined by adding effective divisors is quasifinite.

\bigskip

To show that the map $\bar{f}_{n_1, k, n_2}$ is finite it now suffices to show that it is proper. To do so, consider first the fiber product
$$\barBun_{B^-, n_1} \underset{\Bun_G}{\times} \VinBun_G \underset{\Bun_G}{\times} \barBun_{B, n_2}$$
where the two maps $\VinBun_G \to \Bun_G$ are the two maps remembering only the bundles $E_1$ and $E_2$, respectively. Thus the fiber product parametrizes the data of a point $(E_1, E_2, \varphi)$ of $\VinBun_G$ together with a point $E_1 \to M_1$ of $\barBun_{B^-, n_1}$ and a point $M_2 \to E_2$ of $\barBun_{B, n_2}$.

\medskip

Consider now the closed substack $\CY$ of the above fiber product obtained by requiring that the map $\varphi$ factors through the map $E_1 \to M_1$ and also through the map $M_2 \to E_2$:

$$\xymatrix@+10pt{
E_1 \ar[d] \ar[rr]^{\varphi} \ar@{..>}[rrd] & & E_2 \\
M_1 \ar@{..>}[rru] & & M_2 \ar[u] \\
}$$

\medskip
\medskip

\noindent We claim that the closed substack $\CY$ is in fact isomorphic to the stack
$$\barBun_{B^-, n_1} \ \underset{\Bun_T}{\times} \ \bigl( \, X^{(k)} \, \times \barBun_{B,n_2} \bigr) \, .$$
Indeed, given an $S$-point of $\CY$ as above, the map $\varphi$ is forced to factor as
$$E_1 \longto M_1 \stackrel{i}{\longto} M_2 \longto E_2 \, ,$$
and the datum of the map $i: M_1 \to M_2$ is equivalent to the datum of the map $\varphi$. Moreover, the definition of $\VinBun_G$ forces the map $i: M_1 \to M_2$ to be injective when restricted to $X \times \bar{s}$ for any geometric point $\bar{s} \to S$. Since $M_1$ and $M_2$ have degrees $n_1$ and $n_2$ when restricted to each $X \times \bar{s}$ and since $n_1 = n_2 - k$, the datum of the map $i$ above is in turn equivalent to the datum of an $S$-point of $X^{(k)}$.

\medskip

Finally, note that the map $\bar{f}_{n_1, k, n_2}$ is equal to the composition of the inclusion map of
$$\CY \ = \ \barBun_{B^-, n_1} \ \underset{\Bun_T}{\times} \ \bigl( \, X^{(k)} \, \times \barBun_{B,n_2} \bigr)$$
into the fiber product
$$\barBun_{B^-, n_1} \underset{\Bun_G}{\times} \VinBun_G \underset{\Bun_G}{\times} \barBun_{B, n_2}$$
with the projection of the latter to $\VinBun_G$. Since the inclusion map is a closed immersion and the projection map is proper by Subsection \ref{Basic properties}, we conclude that $\bar{f}_{n_1, k, n_2}$ is proper, finishing the proof.
\end{proof}

\bigskip

\sssec{Proof of stratification results}
\label{Proof of stratification results}

Using Lemma \ref{finiteness of compactified maps} we can now prove the stratification results of Proposition \ref{defect stratification proposition} above. Before doing so, we state the following corollary, which follows from the fact that the map
$$\barBun_B \longto \Bun_T$$
is smooth for $G = \SL_2$ (see Subsection \ref{Basic properties} above).

\bigskip

\begin{corollary}[of Lemma \ref{finiteness of compactified maps}]
\label{resolution of singularities}

The compactified map

$$\bar{f}_{n_1, k, n_2}: \ \ \barBun_{B^-, n_1} \, \underset{\Bun_T}{\times} \bigl( \, X^{(k)} \times \barBun_{B,n_2} \bigr) \ \, \longto \ \, \sideset{_{(n_1, k, n_2)}}{_{G,B}}\barVinBun$$

\medskip

\noindent is a resolution of singularities of the closure of the stratum $\sideset{_{(n_1, k, n_2)}}{_{G,B}}\VinBun$.
\end{corollary}

\bigskip

Finally we prove Proposition \ref{defect stratification proposition}:

\medskip

\begin{proof}[Proof of Proposition \ref{defect stratification proposition} (a)]
We use the same notation as in Subsection \ref{Stratification by defect}. We first show that the map $f_{n_1, k, n_2}$ is a monomorphism of algebraic stacks. Thus we need to check that the data
$$E_1 \longonto M_1 \longto M_2 \longintointo E_2$$
on $X \times S$ can be reconstructed from the composite map $\varphi: E_1 \longto E_2$. Indeed, the line bundle $M_1$ can be recovered as the image $\im(\varphi)$, and the factorization through $M_1$ corresponds to the factorization
$$E_1 \longonto \im(\varphi) \longinto E_2 \, .$$
One can argue dually for $M_2$, and hence $f_{n_1, k, n_2}$ is a monomorphism.

\medskip

We now show that $f_{n_1, k, n_2}$ is in fact a locally closed immersion. Let $\CB$ denote the boundary of

$$\CY \ = \ \barBun_{B^-, n_1} \ \underset{\Bun_T}{\times} \ \bigl( \, X^{(k)} \, \times \barBun_{B,n_2} \bigr) \, ,$$

\noindent i.e., the closed complement in $\CY$ of the open substack

$$\Bun_{B^-, n_1} \ \underset{\Bun_T}{\times} \ \bigl( \, X^{(k)} \, \times \Bun_{B,n_2} \bigr) \, .$$

\medskip

\noindent Since the map $\bar{f}_{n_1, k, n_2}$ is proper, the image of the boundary $\CZ$ under $\bar{f}_{n_1, k, n_2}$ is a closed substack of $\VinBun_G$; let $\CU$ denote its open complement. We claim that taking the inverse image of $\CU$ under $\bar{f}_{n_1, k, n_2}$ yields the cartesian square

$$\xymatrix@+10pt{
\Bun_{B^-, n_1} \underset{\Bun_T}{\times} \bigl( \, X^{(k)} \times \Bun_{B,n_2} \bigr) \ar@{^(->}[r]^{\text{open}} \ar[d] \ar[rd]^{f_{n_1, k, n_2}} \ & \ \barBun_{B^-, n_1} \underset{\Bun_T}{\times} \bigl( \, X^{(k)} \times \barBun_{B,n_2} \bigr) \ar[d]^{\bar{f}_{n_1, k, n_2}} \\
\CU \ \ar@{^(->}[r]^{\text{open}}& \ \VinBun_G \\
}$$

\medskip
\medskip

\noindent This follows from the fact that any point of $\VinBun_G$ lying in the image of the boundary $\CZ$ must have defect strictly greater than $k$.

\medskip

The diagonal map of the above square is precisely the map $f_{n_1, k, n_2}$, which has already been shown to be a monomorphism. Thus the left vertical arrow is also a monomorphism; but being the base change of the proper map~$\bar{f}_{n_1, k, n_2}$, the left vertical arrow is also proper, and hence it must be a closed immersion. This establishes the desired factorization of the map~$f_{n_1, k, n_2}$, showing that it is indeed a locally closed immersion.

\medskip

Finally, the assertion about smoothness follows from the fact that the map $\Bun_{B^-} \longto \Bun_T$ is smooth.
\end{proof}

\bigskip

\begin{proof}[Proof of Proposition \ref{defect stratification proposition} (b) through (e)]
Part (b) follows immediately from the fact that every map $\varphi: E_1 \to E_2$ factors uniquely as
$$\varphi: \ E_1 \longonto M_1 \longinto M_2 \longintointo E_2$$
as in Subsection \ref{Definition of the defect} above.
Part (c) follows from the definition of the map $\bar{f}_{n_1, k, n_2}$ together with the stratifications of $\barBun_B$ and $\barBun_{B^-}$ in Subsection~\ref{Stratification of barBun_B}.
Part (d) follows from the formula for the strata closure in part (c).
For part~(e), note that by part (c) each stratum $\sideset{_{(n_1, k, n_2)}}{_{G,B}}\VinBun$ is closed in the open substack $\sideset{_{\leq k}}{_{G,B}}\VinBun$. Thus the natural map
$$\coprod_{n_1, n_2} \sideset{_{(n_1, k, n_2)}}{_{G,B}}\VinBun \ \ \longto \ \ \sideset{_{\leq k}}{_{G,B}}\VinBun$$
is a closed immersion, and the claim follows.
\end{proof}

\bigskip\bigskip

\section{Statement of main theorems}
\label{Statement of main theorems}

\ssec{Preliminaries about nearby and vanishing cycles}
\label{Preliminaries about nearby and vanishing cycles}

\mbox{} \vspace{0.3cm}

\sssec{Conventions}
\label{Conventions for nearby cycles}

Given a map $Y \to \BA^1$ we will denote by
$$\Psi: \ \D(Y|_{\BA^1 \setminus \{0\}}) \ \longto \ \D(Y|_{\{0\}})$$
the unipotent nearby cycles functor in the perverse and Verdier-self dual renormalization, i.e., we shift and twist the usual unipotent nearby cycles functor by $[-1](-\tfrac{1}{2})$ so that it is t-exact for the perverse t-structure and commutes with Verdier duality literally and not just up to twist. We will refer to~$\Psi$ simply as \textit{the nearby cycles}, and we denote the analogously shifted and twisted unipotent vanishing cycles functor simply by $\Phi$. We denote the logarithm of the unipotent part of the monodromy operator by
$$N: \ \Psi \ \longto \ \Psi(-1) \, ,$$
and will refer to it simply as \textit{the monodromy operator}. The monodromy operator $N$ admits the factorization $N = \var \circ \can$ into the natural maps
$$\can: \ \Psi \longto \Phi \ \ \ \ \ \ \text{and} \ \ \ \ \ \ \var: \ \Phi \longto \Psi(-1) \, .$$
In the above normalization, the usual triangle relating $\Psi$ and $\Phi$ reads
$$F|^*_{Y|_{\{0\}}}[-1](-\tfrac{1}{2}) \ \longto \ \Psi(F) \ \stackrel{\can \ }{\longto} \ \Phi(F) \ \stackrel{+1}{\longto}$$
for any object $F \in \D(Y)$. We refer the reader to \cite{gluing perverse sheaves} and \cite[Sec. 5]{Jantzen conjectures} for more background on unipotent nearby and vanishing cycles.

\bigskip

\sssec{Monodromy and weight filtrations and Gabber's theorem}
We now recall some facts about the monodromy and weight filtrations on nearby cycles; we refer the reader to \cite[Sec. 1.6]{Weil2} and \cite[Sec. 5]{Jantzen conjectures} for proofs.

\medskip

Given a perverse sheaf $F$ on $Y|_{\BA^1 \setminus \{0\}}$, the endomorphism $N$ acts nilpotently on the perverse sheaf $\Psi(F)$, and thus induces the \textit{monodromy filtration} on $\Psi(F)$. The latter filtration is the unique finite filtration
$$\Psi(F) = M_{n} \ \supseteq \ M_{n-1} \ \supseteq \ \cdots \ \supseteq \ M_{-n} \ \supseteq 0$$
by perverse sheaves $M_i$ satisfying that
$$N(M_i) \ \subset \ M_{i-2}(-1)$$
for all $i$, and that the induced maps
$$N^i: \ M_i/M_{i-1} \ \longto \ \bigl(M_{-i}/M_{-i-1}\bigr)(-i)$$
are isomorphisms for all $i \geq 0$.
In particular the operator $N$ acts on the associated graded perverse sheaf $\gr(\Psi(F))$, and we have the following well-known lemma:

\medskip

\begin{lemma}
\label{Lefschetz-sl_2}
The action of $N$ on the associated graded $\gr(\Psi(F))$ extends canonically to an action of the ``Lefschetz-$\sl_2$'', i.e.: There exists a unique action of the Lie algebra $\sl_2(\overline{\BQ}_\ell)$ on $\gr(\Psi(F))$ such that the action of the lowering operator of $\sl_2(\overline{\BQ}_\ell)$ coincides with the action of $N$, and such that the Cartan subalgebra of $\sl_2(\overline{\BQ}_\ell)$ acts on the summand $\gr(\Psi(F))_i = M_i/M_{i-1}$ with Cartan weight $i$. Thus the decomposition
$$\gr(\Psi(F)) \ = \ \bigoplus_{i} M_i/M_{i-1}$$
agrees with the decomposition of the $\sl_2(\overline{\BQ}_\ell)$-representation $\gr(\Psi(F))$ according to Cartan weights. We will refer to the Lie algebra $\sl_2(\overline{\BQ}_\ell)$ in this context as \textit{the Lefschetz}-$\sl_2$.
\end{lemma}

\bigskip

In the case that $F$ is a pure perverse sheaf, the monodromy filtration satisfies Gabber's theorem, which we state for the case of weight $0$:

\medskip

\begin{proposition}[Gabber]
\label{Gabber's theorem}

Assume that $F$ is a pure perverse sheaf of weight~$0$. Then the subquotients of the monodromy filtration on $\Psi(F)$ are pure, and the weight of the subquotient $\gr(\Psi(F))_i  = M_i/M_{i-1}$ is equal to~$i$. In other words, the monodromy filtration agrees with the weight filtration of~$\Psi(F)$. In particular, the weight of each subquotient as a Weil sheaf agrees with its Cartan weight with respect to the action of the Lefschetz-$\sl_2$.
\end{proposition}

\bigskip

Finally, recall on the one hand that the $i$\textit{-th primitive part} $P_i$ of $\Psi(F)$ is defined as the kernel of the map
$$N: \ \gr_i(\Psi(F)) \ \longto \ \gr_{i-2}(\Psi(F)) \, .$$
On the other hand, consider the filtration induced on the kernel
$$\ker \bigl(N: \Psi(F) \to \Psi(F)(-1) \bigr) \ \ \subset \ \ \Psi(F)$$
by the monodromy filtration on $\Psi(F)$ by means of intersecting the kernel with the monodromy filtration. We then have the following well-known lemma:

\begin{lemma}
\label{kernel and associated graded commute}
The $i$-th subquotient of the latter filtration is canonically isomorphic to the $i$-th primitive part $P_i$. Less precisely, the associated graded of the kernel of $N: \Psi(F) \to \Psi(F)(-1)$ agrees with the kernel of $N$ acting on the associated graded $\gr(\Psi(F))$.
\end{lemma}

\bigskip

\sssec{Intersection cohomology from nearby cycles}
As above let $Y$ be a scheme or stack, let $Y \to \BA^1$ be a map, and let $\IC_Y$ denote the IC-sheaf of $Y$, normalized to be pure of weight $0$ as mentioned in Subsection \ref{Notation and conventions}. Consider the monodromy operator
$$N: \ \Psi(\IC_Y) \ \longto \ \Psi(\IC_Y)(-1)$$
acting on the nearby cycles of the IC-sheaf of $Y$, and let $\ker(N) \subset \Psi(\IC_Y)$ and $\im(N) \subset \Psi(\IC_Y)(-1)$ denote its perverse kernel and image. Note furthermore that for the IC-sheaf of $Y$, the usual triangle for the map $\can: \Psi \to \Phi$ is in fact a short exact sequence of perverse sheaves
$$\xymatrix@+10pt{
0 \ar[r] & \IC_Y|^*_{Y|_{\{0\}}}[-1](-\tfrac{1}{2}) \ar[r] & \Psi(\IC_Y) \ar^{can}[r] & \Phi(\IC_Y) \ar[r] & 0 \, .
}$$

\noindent For example from Beilinson's gluing description (\cite{gluing perverse sheaves}) applied to the IC-sheaf of $Y$ one verifies:

\medskip

\begin{lemma}
\label{IC via Psi}
The above short exact sequence in fact coincides with the short exact sequence
$$\xymatrix@+10pt{
0 \ar[r] & \ker(N) \ar[r] & \Psi(\IC_Y) \ar^{N}[r] & \im(N) \ar[r] & 0 \, ,
}$$
i.e., we have
$$\ker(N) \ = \ \IC_Y|^*_{Y|_{\{0\}}}[-1](-\tfrac{1}{2}) \ ,$$
$$\im(N) \ = \ \Phi(\IC_Y) \ . \ \ \ \ \ \ \ \ \ \ \ \ \ $$

\noindent In particular, one can obtain the restriction $\IC_Y|^*_{Y|_{\{0\}}}[-1](-\tfrac{1}{2})$ and the vanishing cycles $\Phi(\IC_Y)$ from understanding the monodromy action on $\Psi(\IC_Y)$.
\end{lemma}

\bigskip

\sssec{Nearby cycles of the Picard-Lefschetz family of hyperbolas}
\label{Nearby cycles of the Picard-Lefschetz family of hyperbolas}

We now recall the well-known computation of the nearby cycles for the family of hyperbolas $xy = t$, i.e., the nearby cycles for the map
$$d: \ \BA^2 \longto \BA^1, \ \ (x,y) \longmapsto x \cdot y \, .$$
This example in fact plays a key role in the nearby cycles computation for $\VinBun_G$, as will become clear in the next two sections and in the proof of the main theorem about nearby cycles in Section \ref{Nearby cycles} below.

\medskip

To state the result, let $C = d^{-1}(\{0\})$ be the fiber of $d$ over $0 \in \BA^1$, i.e., the reducible node in $\BA^2$ formed by the union $C = C_x \cup C_y$ of the two coordinate axes $C_x$ and $C_y$. Let $p$ denote the origin of $\BA^2$, i.e., the intersection of $C_x$ and $C_y$, and let $\delta_p$ denote the pushforward of the constant sheaf $\Qellbar$ along the inclusion $p \into C$. Let $(\Qellbar)_C$ denote the constant sheaf on $C$ and let $i_{x,*}(\Qellbar)_{C_x}$ and $i_{y,*}(\Qellbar)_{C_y}$ denote the pushforwards of the constant sheaves from $C_x$ and $C_y$ to $C$. Thus the IC-sheaf of $C$ is equal to
$$\IC_C \ = \ i_{x,*}(\Qellbar)_{C_x} \ \oplus \ i_{y,*}(\Qellbar)_{C_y} \, .$$
Applying the nearby cycles functor to the IC-sheaf $\IC_{\BA^2} = \Qellbar[2](1)$ of $\BA^2$ one finds:

\medskip

\begin{lemma}
\label{Psi for the Picard-Lefschetz family of hyperbolas}
The weight-monodromy filtration on $\Psi(\IC_{\BA^2})$ is equal to
$$\Psi(\IC_{\BA^2}) \  \ \supsetneq \ \ (\Qellbar)_C[1](\tfrac{1}{2}) \ \ \supsetneq \ \ \delta_p(\tfrac{1}{2}) \ \ \supsetneq \ \ 0 \, ,$$
and the corresponding associated graded object equals
$$\gr \, \Psi(\IC_{\BA^2}) \ \ = \ \ \delta_p(-\tfrac{1}{2}) \ \oplus \ \IC_C \ \oplus \ \delta_p(\tfrac{1}{2}) \, .$$
Furthermore, the action of the monodromy operator $N$ on $\gr \, \Psi(\IC_{\BA^2})$ identifies $\gr_{1} = \delta_p(-\tfrac{1}{2})$ with $\gr_{-1}(-1) = \delta_p(\tfrac{1}{2})(-1) = \delta_p(-\tfrac{1}{2})$, and the action on $\gr_0 = \IC_C$ is trivial. In particular, as a representation of the Lefschetz-$\sl_2$ the direct sum $\gr_{-1} \oplus \gr_1$ is isomorphic to the standard representation of the Lefschetz-$\sl_2$.
\end{lemma}

\bigskip\bigskip

\ssec{Picard-Lefschetz oscillators}
\label{Picard-Lefschetz oscillators}

\sssec{Factorization structures}
\label{Factorization structures}

Assume we are given, for each $n \in \BZ_{\geq 0}$, a perverse sheaf $F_n \in D(X^{(n)})$ on the symmetric power $X^{(n)}$ of the curve $X$. Here and below we denote by
$$\add: \ X^{(n_1)} \times X^{(n_2)} \ \longto \ X^{(n)}$$
the map defined by adding effective divisors. Then we define a \textit{factorization structure} on the collection of perverse sheaves $F_n$ to be a collection of compatible isomorphisms
$$(\add^* F_n)\big|^*_{X^{(n_1)} \stackrel{\circ}{\times} X^{(n_2)}} \ \ \cong \ \ F_{n_1} \stackrel{\circ}{\boxtimes} F_{n_2}$$
for any $n_1 + n_2 = n$. If there is no ambiguity about which factorization structure is being considered on a given collection of perverse sheaves $F_n$, then we also abuse terminology and refer to the collection of perverse sheaves $F_n$ as \textit{factorizable}.

\bigskip

\sssec{External exterior powers}
\label{sssec External exterior powers}

Recall that to any local system $E$ on the curve $X$, placed in cohomological degree $0$, one can associate its $n$-\textit{th} \textit{external exterior power} $\Lambda^{(n)}(E)$ on the symmetric power of the curve $X^{(n)}$. Namely, the $n$-fold exterior product $E \boxtimes \cdots \boxtimes E$ on the $n$-th power $X^n$ carries a natural equivariant structure with respect to the action of the symmetric group $S_n$ on $X^n$. Thus its pushforward $p_* (E \boxtimes \cdots \boxtimes E)$ along the natural map
$$p: \ X^n \ \longto \ X^{(n)}$$
carries a natural action of $S_n$, and we define $\Lambda^{(n)}(E)$ by taking $S_n$-invariants of the pushforward $p_* (E \boxtimes \cdots \boxtimes E)$ against the sign character of $S_n$.

\medskip

This construction is functorial and satisfies the basic properties listed in the next lemma (see for example \cite[Sec. 5]{on de Jong} for proofs).

\medskip

\begin{lemma}
\label{external exterior powers}
\begin{itemize}
\item[]
\item[]
\item[(a)] Over the disjoint locus $\overset{\circ}{X} {}^{(n)}$ the $n$-th external exterior power $\Lambda^{(n)}(E)$ is again a local system.
\item[]
\item[(b)] The shifted object $\Lambda^{(n)}(E)[n]$ is a perverse sheaf. In fact, it is equal to the intermediate extension of its restriction to the disjoint locus.
\item[]
\item[(c)] Let $D = \sum_{k} n_k x_k \in X^{(n)}$ be a divisor on $X$, with the points $x_k$ distinct. Then the $*$-stalk of $\Lambda^{(n)}(E)$ at the point $D$ is equal to
$$\bigotimes_k \, \Lambda^{n_k}(E) \, .$$
\item[]
\item[(d)] The collection of perverse sheaves $\Lambda^{(n)}(E)[n]$ is factorizable in the sense of Subsection \ref{Factorization structures} above.

\end{itemize}
\end{lemma}

\bigskip

\sssec{Definition of Picard-Lefschetz oscillators}
\label{Definition of Picard-Lefschetz oscillators}

Let $V$ denote the $2$-dimensional standard representation of the Lefschetz-$\sl_2$:
$$V \ = \ \Qellbar(\tfrac{1}{2}) \oplus \Qellbar(-\tfrac{1}{2})$$
We denote by
$$\underline V \ := \ V \otimes \Qellbar_X$$
the corresponding constant local system of rank $2$ on the curve $X$ together with the induced action of the Lefschetz-$\sl_2$. For any integer $n \geq 1$ we then define the \textit{Picard-Lefschetz oscillator} $\CP_n$ on $X^{(n)}$ to be the $n$-th external exterior power of $\underline V \,$, shifted and twisted as follows:
$$\CP_n \ := \ \Lambda^{(n)} (\underline V) \, [n](\tfrac{n}{2})$$
Thus by Lemma \ref{external exterior powers} above $\CP_n$ is a perverse sheaf on $X^{(n)}$, and carries an action of the Lefschetz-$\sl_2$ by the functoriality of the external exterior power construction. Furthermore the definition and Lemma \ref{external exterior powers} together show:

\medskip

\begin{lemma}
\label{properties of PLOs}
Let the symmetric group $S_n$ act on the $n$-fold tensor power $V \otimes \cdots \otimes V$ by permuting the factors and additionally multiplying by the sign of the permutation, and consider the local system on the disjoint locus~$\overset{\circ}{X} {}^{(n)}$ associated to this representation. Then the IC-extension of this local system is equal to the Picard-Lefschetz oscillator $\CP_n$. In particular the perverse sheaf $\CP_n$ is semisimple. Finally, the natural factorization structure on the collection of Picard-Lefschetz oscillators $\CP_n$ respects the action of the Lefschetz-$\sl_2$.
\end{lemma}

\medskip

Our choice of the term Picard-Lefschetz oscillators is due, on the one hand, to the appearance of the sign character in the action of the symmetric group in Lemma \ref{properties of PLOs} above; and, on the other hand, due to the appearance of the representation $V$: For $n=1$ the Picard-Lefschetz oscillator $\CP_1$ equals, up to shifts and twists, the constant rank-$2$ local system on the curve~$X$ whose fiber is equal to the standard representation $V$ of the Lefschetz-$\sl_2$; the latter is precisely the summand of the associated graded of the nearby cycles sheaf of the Picard-Lefschetz family from Subsection \ref{Nearby cycles of the Picard-Lefschetz family of hyperbolas} above consisting of those summands supported on the singular locus $\{p\}$ of the map $d: \BA^2 \to \BA^1$.

\bigskip

\ssec{Nearby cycles for $\VinBun_G$}

\mbox{} \medskip

To state our main theorem about nearby cycles for $\VinBun_G$ we will need the following definition:

\medskip

\sssec{Placing Picard-Lefschetz oscillators on $\VinBun_G$}
\label{Placing Picard-Lefschetz oscillators on VinBun}

We now define versions of the Picard-Lefschetz oscillators on the strata closures of the defect stratification of the $B$-locus from Subsection~\ref{The defect stratification} above. More precisely, we define versions $\widetilde \CP_{n_1, k, n_2}$ of the Picard-Lefschetz oscillators $\CP_k$ on the relative compactifications
$$\barBun_{B^-, n_1} \ \underset{\Bun_T}{\times} \ \bigl( \, X^{(k)} \, \times \barBun_{B,n_2} \bigr)$$
from Subsection~\ref{Compactifying the maps $f_{n_1, k, n_2}$} above; the latter map onto the strata closures in $\VinBun_G$ via the compactified maps
$$\bar{f}_{n_1, k, n_2}: \ \barBun_{B^-, n_1} \ \underset{\Bun_T}{\times} \ \bigl( \, X^{(k)} \, \times \barBun_{B,n_2} \bigr) \ \longto \ \VinBun_{G,B}$$
introduced in Subsection \ref{compactifying the maps sssec} above.

\medskip

To state the definition, let $(n_1, k, n_2)$ be any triple with $n_1 = n_2 - k$, as before. Then we define $\widetilde \CP_{n_1, k, n_2}$ as
$$\widetilde \CP_{n_1, k, n_2} \ \ := \ \ \IC_{\barBun_{B^-, n_1}} \underset{ \ \Bun_T}{\boxtimes} \ \CP_k \, \boxtimes \IC_{\barBun_{B,n_2}} \, ,$$
i.e., as the $*$-restriction of the external product
$$\IC_{\barBun_{B^-, n_1}} \boxtimes \ \CP_k \, \boxtimes \IC_{\barBun_{B,n_2}}$$
from the product space
$$\barBun_{B^-, n_1} \times \, X^{(k)} \, \times \barBun_{B,n_2}$$
to the fiber product
$$\barBun_{B^-, n_1} \ \underset{\Bun_T}{\times} \ \bigl( \, X^{(k)} \, \times \barBun_{B,n_2} \bigr) \, ,$$
shifted by $[- \dim \Bun_T]$ and twisted by $(- \tfrac{\dim \Bun_T}{2})$. Since the IC-sheaf of $\barBun_B$ is constant for $G=\SL_2$ by Subsection \ref{Basic properties}, we can rephrase the definition of $\widetilde \CP_{n_1, k, n_2}$ as follows. Let $g$ denote the genus of the curve $X$ and let $s_k$ denote the integer
$$s_k \ := \ \dim \barBun_{B^-, n_1} + \dim \barBun_{B,n_2} - \dim \Bun_T \ = \ 3g-3 + 2k \, .$$
Furthermore let
$$p_{n_1, k, n_2}: \ \ \barBun_{B^-, n_1} \ \underset{\Bun_T}{\times} \ \bigl( \, X^{(k)} \, \times \barBun_{B,n_2} \bigr) \ \ \longto \ \ X^{(k)}$$
denote the natural forgetful map. Then we can equivalently define:
$$\widetilde \CP_{n_1, k, n_2} \ \ := \ \ p_{n_1, k, n_2}^* \CP_k \ [s_k](\tfrac{s_k}{2})$$
Finally, the action of the Lefschetz-$\sl_2$ on the Picard-Lefschetz oscillator $\CP_k$ induces an analogous action on $\widetilde \CP_{n_1, k, n_2}$.

\bigskip

\sssec{Main theorem about nearby cycles}

We can now state our main theorem about nearby cycles for $\VinBun_G$. Its proof will be given in Section \ref{Nearby cycles} below. Recall from Subsection \ref{The $G$-locus and the $B$-locus} that the $G$-locus $\VinBun_{G,G}$ is smooth, so that its IC-sheaf is constant up to shifts and twists. Applying the nearby cycles functor $\Psi$ to this shifted constant sheaf and passing to the associated graded of its weight-monodromy filtration, we obtain a perverse sheaf on the $B$-locus $\VinBun_{G,B}$ carrying the monodromy action of the Lefschetz $\sl_2$. The result then is:

\bigskip

\begin{theorem}
\label{main theorem for nearby cycles}
There exists an isomorphism of perverse sheaves
$$\gr \, \Psi (\IC_{\VinBun_{G,G}})     \ \ \ \cong \ \ \ \bigoplus_{(n_1, k, n_2)} \bar{f}_{n_1, k, n_2, *}\ \widetilde \CP_{n_1, k, n_2}$$
which identifies the action of the Lefschetz-$\sl_2$ on the right hand side via the Picard-Lefschetz oscillators with the monodromy action on the left hand side. As before the direct sum runs over all triples $(n_1, k, n_2)$ with $n_1, n_2 \in \BZ$, $k \in \BZ_{\geq 0}$, and $n_1 = n_2 - k$.
\end{theorem}

\medskip

\sssec{Remark}
By Subsection \ref{Compactifying the maps $f_{n_1, k, n_2}$} above, the summands $\bar{f}_{n_1, k, n_2, *}\ \widetilde \CP_{n_1, k, n_2}$
on the right hand side of Theorem \ref{main theorem for nearby cycles} are equal to the intermediate extension of the perverse sheaf
$$\IC_{\Bun_{B^-, n_1}} \underset{ \ \Bun_T}{\boxtimes} \ \CP_k \, \boxtimes \IC_{\Bun_{B,n_2}}$$
from the stratum $\sideset{_{(n_1, k, n_2)}}{_{G,B}}\VinBun$ to its closure in $\VinBun_{G,B}$.

\bigskip\bigskip

\ssec{Intersection cohomology of $\VinBun_G$}

\mbox{} \medskip

To state our main theorem about the IC-sheaf of $\VinBun_G$, we introduce the following notation. Given a representation $\rho$ of the symmetric group $S^k$, we denote by $\IC(\rho)$ the IC-extension of the corresponding local system on the disjoint locus of $X^{(k)}$. Furthermore, using the same notation as in the definition of $\widetilde \CP_{n_1, k, n_2}$ above, we define
$$\widetilde{\IC}(\rho)_{n_1, k, n_2} \ \ := \ \ \IC_{\barBun_{B^-, n_1}} \underset{ \ \Bun_T}{\boxtimes} \ \IC(\rho) \, \boxtimes \IC_{\barBun_{B,n_2}}$$
on the fiber product
$$\barBun_{B^-, n_1} \ \underset{\Bun_T}{\times} \ \bigl( \, X^{(k)} \, \times \barBun_{B,n_2} \bigr) \, .$$
Using the projection maps $p_{n_1, k, n_2}$ from Subsection \ref{Placing Picard-Lefschetz oscillators on VinBun} above one can equivalently define
$$\widetilde{\IC}(\rho)_{n_1, k, n_2} \ \ := \ \ p_{n_1, k, n_2}^* \, \IC(\rho) \, [s_k](\tfrac{s_k}{2}) $$
where the integers $s_k$ are defined as in Subsection \ref{Placing Picard-Lefschetz oscillators on VinBun} above. We can then state:

\bigskip

\begin{theorem}
\label{main theorem for intersection cohomology}
The associated graded with respect to the weight filtration of the restriction $\IC_{\VinBun_G} \big|^*_{\VinBun_{G,B}}[-1](-\tfrac{1}{2})$ is equal to:
$$\bigoplus_{(n_1, k, n_2, r)} \bar{f}_{n_1, k, n_2, *} \ \widetilde{\IC}(\rho_{k-r, r})_{n_1, k, n_2} \, \otimes \Qellbar (\tfrac{k}{2} - r) \, .$$

\medskip

\noindent Here we denote by $\rho_{(k-r,r)}$ the irreducible representation of $S_k$ corresponding to the Young diagram with $k-r$ boxes in the first column and~$r$ boxes in the second column. The direct sum runs over all quadruples $(n_1, k, n_2, r)$ where $n_1, n_2 \in \BZ$ and $k,r \in \BZ_{\geq 0}$, satisfying that $n_1 = n_2 - k$ and $0 \leq r \leq \tfrac{k}{2}$.
\end{theorem}

\bigskip

Theorem \ref{main theorem for intersection cohomology} yields an explicit answer for the primitive parts $P_i$; in general it is however not clear how to compute the IC-stalks from the $P_i$. In the present situation it is however possible, due to a geometric fact visible on the level of the local models we will construct in Section \ref{local models for VinBun} below. For this reason we comment on the computation of IC-stalks in Remark \ref{IC stalks} below. The proof of Theorem \ref{main theorem for intersection cohomology} will be given in Section \ref{Intersection cohomology} below.

\bigskip\bigskip\bigskip

\ssec{Stalks of the $*$-extension}

\mbox{} \medskip

We now state our result describing the stalks of the $*$-extension of the constant sheaf; we refer the reader to Subsections \ref{Stalks of the *-extension of the constant sheaf}, \ref{Relation to the geometric Langlands program}, and \ref{An application to Drinfeld's strange invariant bilinear form} of the introduction, as well as to Section \ref{An application: Computation of Drinfeld's function}, for motivation and for applications.

\medskip

To state the result, let
$$j_G: \ \VinBun_{G,G} \ \longinto \ \VinBun_G$$
denote the open immersion of the $G$-locus, and let $i_{n_1, k, n_2}$ denote the inclusion of the stratum
$$\sideset{_{(n_1, k, n_2)}}{_{G,B}}\VinBun \ = \ \Bun_{B^-, n_1} \ \underset{\Bun_T}{\times} \ \bigl( \, X^{(k)} \, \times \Bun_{B,n_2} \bigr)$$
into the $B$-locus $\VinBun_{G,B}$. Our result then expresses the $*$-restriction along $i_{n_1, k, n_2}$ of the $*$-extension along $j_G$ of the constant sheaf on the $G$-locus $\VinBun_{G,G}$ in terms of a certain complex $\widetilde \Omega_k$. This complex has already been studied in the work \cite{BG2} of Braverman and Gaitsgory, and is defined as follows.

\medskip

\sssec{Definition of $\widetilde \Omega_k$}
Let ${}_0Z^k$ denote the \textit{open Zastava space} from \cite{FFKM}, \cite{BFGM}; see Subsection \ref{Zastava spaces} below for its definition. As is explained in Subsection \ref{Zastava spaces}, the space ${}_0Z^k$ is smooth and comes equipped with a projection map to the $k$-th symmetric power of the curve
$$\pi_Z: \ {}_0Z^k \ \longto \ X^{(k)} \, .$$
We then define $\widetilde \Omega_k$ as the pushforward
$$\widetilde \Omega_k \ := \ \pi_{Z,!} \bigl( \IC_{{}_0Z^k}\bigr) \ = \ \pi_{Z,!} \bigl( (\Qellbar)_{{}_0Z^k}[\dim_{{}_0Z^k}](\tfrac{1}{2} \dim_{{}_0Z^k})\bigr) \, .$$
We refer the reader to Subsection \ref{Compactly supported cohomology of open Zastava spaces} below for a more detailed discussion of the complex $\widetilde \Omega_k$.
We will in fact express our result in terms of the Verdier dual
$$\BD \, \widetilde \Omega_k \ = \ \pi_{Z,*} \bigl( \IC_{{}_0Z^k}\bigr)$$
of $\widetilde \Omega_k$.

\bigskip

Using the same notation as before, our result reads:

\bigskip

\begin{theorem}
\label{ij}
The $*$-restriction of the $*$-extension of the IC-sheaf of the $G$-locus
$$i_{n_1, k, n_2}^* \, j_{G,*} \, \IC_{\VinBun_{G,G}}$$
is equal to
$$\IC_{\Bun_{B^-, n_1}} \underset{\Bun_T}{\boxtimes} \Bigl( \Bigl( \BD \widetilde \Omega_k [2k](k) \, \otimes \, H^*(\BA^1 \setminus \{0\})[1](\tfrac{1}{2}) \Bigl) \ \boxtimes \ \IC_{\Bun_{B, n_2}} \Bigr) \, .$$
\end{theorem}

\bigskip

Theorem \ref{ij} will be proven in Proposition \ref{omega tilde proposition} below, on the level of the \textit{local models} which we introduce in the next section.

\bigskip\bigskip\bigskip

\section{Local models for $\VinBun_G$}
\label{local models for VinBun}

\bigskip
\ssec{The absolute and relative local models}

\mbox{} \medskip

\sssec{Definition of the relative model}
\label{Definition of the relative model}

Let $n \in \BZ_{\geq 1}$. An $S$-point of the \textit{relative local model} $Y^n_{rel}$ consists of the data of a triple $(E_1, E_2, \varphi)$ on $X \times S$ as in the definition of $\VinBun_G$, together with a line subbundle $L_1 \longintointo E_1$ and a line quotient bundle $E_2 \longonto L_2$, satisfying the following conditions: For every geometric point $\bar{s} \to S$ we require the restriction to $X \times \bar{s}$ of the composite map
$$L_1 \ \longintointo \ E_1 \ \stackrel{\varphi}{\longto} \ E_2 \ \longonto \ L_2$$
to be an isomorphism generically on the curve $X \times \bar{s}$. Furthermore, for each $\bar{s} \to S$ we require the resulting injection of line bundles
$$L_1|_{X \times \bar{s}} \ \longinto \ L_2|_{X \times \bar{s}}$$
to be of relative degree $n$, i.e., we require it to correspond to an effective divisor of degree $n$ on the curve $X \times \bar{s}$. Note that these conditions in particular imply that the composite map $L_1 \to L_2$ is an injection of coherent sheaves on $X \times S$.

\medskip

\sssec{Definition of the absolute model}
Next consider the natural map
$$Y^n_{rel} \ \longto \ \Bun_T$$
defined by remembering only the line bundle $L_2$. We define the \textit{absolute local model} $Y^n$ as the fiber of this map over the trivial line bundle $\CO_X$; i.e., the absolute model $Y^n$ is obtained from the relative model $Y^n_{rel}$ by requiring the ``background'' line bundle $L_2$ to be the trivial line bundle. It is not hard to see that the absolute local model $Y^n$ is in fact a scheme.

\medskip

\sssec{More definitions}
\label{More definitions}

The following definitions and notation apply to both~$Y^n$ and $Y^n_{rel}$. We only state them for $Y^n$, the case of $Y^n_{rel}$ being analogous. By construction the space $Y^n$ admits a forgetful map to $\VinBun_G$, and in particular a natural map to $\BA^1 = T_{adj}^+$. Using the latter map we define the $G$-\textit{locus} $Y^n_G$ and the $B$-\textit{locus} $Y^n_B$ as for $\VinBun_G$ in Subsection \ref{The $G$-locus and the $B$-locus} above. The \textit{defect-free locus} ${}_0Y^n$ of $Y^n$ is defined exactly as for $\VinBun_G$; i.e., it is the inverse image of $\sideset{_0}{_G}\VinBun$ under the forgetful map $Y^n \to \VinBun_G$. As for $\VinBun_G$ we have:

\medskip

\begin{lemma}
The restriction of the map $Y^n_G \to \BA^1$ to the defect-free locus~${}_0Y^n_G$ is smooth; in particular the open subscheme ${}_0Y^n_G$ of $Y^n$ is smooth.
\end{lemma}

\medskip

The space $Y^n$ furthermore admits a natural projection map to the $n$-th symmetric power of the curve
$$\pi: \ Y^n \ \longto \ X^{(n)} \, .$$
Namely, recall that an $S$-point of $X^{(n)}$ consists of a line bundle $L$ on $X \times S$ together with a map of coherent sheaves $L \to \CO_{X \times S}$ which is injective of relative degree $n$ whenever restricted to $X \times \bar{s}$ for every geometric point $\bar{s} \to S$. The map $\pi$ is then defined by only remembering the composite map of line bundles $L_1 \to L_2 = \CO_{X \times S}$.

\medskip

As for $\VinBun_{G,B}$ every $k$-point in the $B$-locus $Y^n_B$ admits a unique factorization
$$L_1 \ \longintointo \ E_1 \longonto M_1 \longto M_2 \longintointo E_2 \ \longonto L_2$$
with notation as above. As before we call the effective divisor corresponding to $M_1 \into M_2$ the \textit{defect divisor} and its degree the \textit{defect}. The $B$-locus~$Y^n_B$ is stratified according to defect degrees just like $\VinBun_{G,B}$. To state the analogous result, we first recall:

\medskip

\sssec{Zastava spaces}
\label{Zastava spaces}

In \cite{FFKM}, \cite{BFGM} certain local models for the relative compactifications $\barBun_B$ from Subsection \ref{Compactifying the maps $f_{n_1, k, n_2}$}, the \textit{Zastava spaces}, were introduced. We recall now their definition for $G = \SL_2$.

\medskip

Let $n \in \BZ_{\geq 0}$. Then an $S$-point of the relative Zastava space $Z^n_{rel}$ consists of an $S$-point $L \into E$ of $\barBun_B$ together with a ``background'' line bundle $L'$ on $X \times S$ and a surjection $E \onto L'$, subject to the following conditions. First, one requires that for every geometric point $\bar{s} \to S$ the restriction of the composite map
$$L \ \longinto \ E \ \longonto \ L'$$
to $X \times \bar{s}$ is an isomorphism generically on the curve $X \times \bar{s}$. Second, the resulting injective map of line bundles $L \into L'$ on $X \times S$ is required to be of relative degree $n$ on each $X \times \bar{s}$.

\medskip

As before one defines an absolute version $Z^n$ of $Z^n_{rel}$ by forcing the ``background'' line bundle $L'$ to be the trivial line bundle. The absolute Zastava space $Z^n$ is in fact a scheme. Next, the notation
$$Z^n_{(\Bun_{T,d})}$$
refers to the relative version, but with the degree of the ``background'' line bundle~$L'$ being required to be equal to the integer $d$. Similarly, for the opposite Borel $B^-$, the space
$$Z^{-,n}_{(\Bun_{T,-n})}$$
parametrizes the data
$$L' \ \longintointo \ E \ \longto \ L$$
where now $L'$ is a ``background'' line subbundle of $E$ of degree $-n$, and the composite map $L' \to L$ is required to be an isomorphism generically on $X$ of relative degree~$n$.

\medskip

We denote by ${}_0Z^n$ the open subscheme of the absolute Zastava space obtained by requiring that the injection $L \into E$ is in fact a subbundle map, and similary for the relative versions. We will refer to ${}_0Z^n$ and its relative versions as the \textit{open Zastava space}. Finally, the absolute and relative Zastava spaces afford natural maps
$$\pi_Z: \ Z^n \ \longto \ X^{(n)}$$
defined as in Subsection \ref{More definitions} above; the absolute Zastava spaces are in fact \textit{factorizable}, in the sense of Subsection \ref{Factorization}, with respect to these maps.

\medskip

We refer the reader to \cite{FM}, \cite{FFKM}, and \cite{BFGM} for more background on the Zastava spaces.

\medskip

\sssec{The $G$-locus of $Y^n$ in terms of Zastava spaces}

Let $Y^n_{c=1}$ denote the fiber of the natural map $Y^n \to \BA^1$ over the element $c=1 \in \BA^1$. Then directly from the definitions one sees that $Y^n_{c=1}$ agrees with the open Zastava space ${}_0Z^n$. In fact we have the following analog for the $G$-locus $Y^n_G$ of Lemma \ref{G-locus product decomposition} for $\VinBun_G$. Given an $S$-point
$$L_1 \ \longintointo \ E_1 \ \stackrel{\cong}{\longto} \ E_2 \ \longonto \ \CO_{X \times S}$$
of $Y^n_G$ we can define an $S$-point
$$L_1 \ \longintointo \ E_1 \ \longonto \ \CO_{X \times S}$$
of ${}_0Z^n$ by composing the last two maps; furthermore, we obtain the $S$-point $\det \varphi$ of $\BA^1$. Then Lemma \ref{G-locus product decomposition} above implies:

\medskip

\begin{lemma}
\label{G-locus decomposition for the local models}
The natural map
$$Y^n_G \ \ \ \longto \ \ \ {}_0Z^n \times (\BA^1 \setminus \{0\})$$
defined by the above association is an isomorphism. Under this isomorphism, the natural map $Y^n_G \to \BA^1 \setminus \{0\}$ corresponds on the right hand side to the projection onto the second factor. The projection map $\pi: Y^n_G \to X^{(n)}$ corresponds on the right hand side to the projection onto the first factor followed by the projection map $\pi_Z: \, {}_0Z^n \to X^{(n)}$.
\end{lemma}

\medskip

\bigskip

\sssec{Stratification of the $B$-locus of the local model}
\label{Stratification of the $B$-locus of the local model}

The stratification of the $B$-locus of $\VinBun_G$ from Proposition \ref{defect stratification proposition} above takes the following form for the local model $Y^n$. Let $n_1, k, n_2 \in \BZ_{\geq 0}$ be non-negative integers satisfying $n_1 + k + n_2 = n$. Then as in Subsection~\ref{Stratification by defect} above there exist natural maps
$$\bar{f}_{n_1, k, n_2}: \ \ Z^{-,n_1}_{(\Bun_{T,-n})} \, \underset{\Bun_T}{\times} \bigl( \, X^{(k)} \times Z^{n_2} \bigr) \ \ \ \longto \ \ \ Y^n_B \, ,$$
and the analogous result is:

\medskip

\begin{corollary}
\label{stratification for local model}
The maps $\bar{f}_{n_1, k, n_2}$ are proper, and their restrictions

$$f_{n_1, k, n_2}: \ \ {}_0Z^{-,n_1}_{(\Bun_{T,-n})} \, \underset{\Bun_T}{\times} \bigl( \, X^{(k)} \times {}_0Z^{n_2} \bigr) \ \ \ \longto \ \ \ Y^n_B$$

\medskip

\noindent are isomorphisms onto smooth locally closed substacks
$${}_{(n_1,k,n_2)}Y^n_B \ \ \longinto \ \ Y^n_B \, .$$

\medskip

\noindent As the triples $(n_1, k, n_2)$ range over all triples of non-negative integers satisfying $n_1 + k + n_2 = n$, the substacks ${}_{(n_1,k,n_2)}Y^n_B$ form a stratification of the $B$-locus $Y^n_B$. On the level of $k$-points the closure of a stratum is equal to the finite disjoint union of strata

$${}_{(n_1,k,n_2)}\overline{Y^n_B} \ \ = \ \ \bigcup_{\substack{d_1 \geq 0 \\ d_2 \geq 0}} \ \,  {}_{(n_1 - d_1, k + d_1 + d_2, n_2 - d_2)}Y^n_B$$

\medskip

\noindent Given any non-negative integer $k \in \BZ_{\geq 0}$, the locus ${}_{\leq k}Y^n_B$ in $Y^n_B$ obtained by requiring the defect to be at most $k$ is open in $Y^n_B$. The locus~${}_kY^n_B$ obtained by requiring the defect to be exactly $k$ is locally closed, and isomorphic as schemes to the disjoint union
$${}_kY^n_B \ \ \ =  \ \ \ \coprod_{n_1, n_2} \ {}_{(n_1, k, n_2)}Y^n_B \, .$$
Finally, the locus ${}_nY^n_B$ of maximal defect is closed in $Y^n_B$ and isomorphic to the symmetric power $X^{(n)}$.
\end{corollary} 

\bigskip

\ssec{Restatements of the main theorems for the local models}
\label{Restatements of the main theorems for the local models}

\mbox{} \medskip

To prove the main theorem about nearby cycles, Theorem \ref{main theorem for nearby cycles} above, it suffices to establish its analog for the absolute local models $Y^n$, for all integers $n \geq 0$; the same holds for the other theorems in Section \ref{Statement of main theorems} above. This follows via a standard argument often referred to as the ``interplay principle'', carried out for example in \cite[Sec. 3, 8.1]{BFGM}, \cite[Sec. 4.3, 4.6, 4.7]{BG2}, where the analogous interplay between Drinfeld's compactification $\barBun_B$ and the Zastava spaces is used: One first compares the absolute and the relative local model to each other, and then compares the relative model to $\VinBun_G$.

\medskip

To state this analog of Theorem \ref{main theorem for nearby cycles} for $Y^n$, let $n_1, k, n_2 \in \BZ_{\geq 0}$ be non-negative integers satisfying $n_1 + k + n_2 = n$, and recall the compactified maps
$$\bar{f}_{n_1, k, n_2}: \ \ Z^{-,n_1}_{(\Bun_{T,-n})} \, \underset{\Bun_T}{\times} \bigl( \, X^{(k)} \times Z^{n_2} \bigr) \ \ \ \longto \ \ \ Y^n_B$$
from Subsection \ref{Stratification of the $B$-locus of the local model}. Similarly as before let
$$\widetilde \CP_{n_1, k, n_2} \ \ := \ \ \IC_{Z^{-,n_1}_{Bun_T,-n}} \underset{ \ \Bun_T} \boxtimes \ \ \CP_k \ \boxtimes \ \IC_{Z^{n_2}}$$
denote the $*$-restriction of the external product
$$\IC_{Z^{-,n_1}_{\Bun_T,-n}} \boxtimes \ \ \CP_k \ \boxtimes \ \IC_{Z^{n_2}}$$
from the product space
$$Z^{-,n_1}_{(\Bun_{T,-n})} \times X^{(k)} \times Z^{n_2}$$
to the fiber product
$$Z^{-,n_1}_{(\Bun_{T,-n})} \, \underset{\Bun_T}{\times} \bigl( \, X^{(k)} \times Z^{n_2} \bigr) \, ,$$
shifted by $[\dim \Bun_T]$ and twisted by $(\tfrac{\dim \Bun_T}{2})$. Since the Zastava spaces are smooth for $G=\SL_2$, we can equivalently define $\widetilde \CP_{n_1, k, n_2}$ as
$$\widetilde \CP_{n_1, k, n_2} \ \ = \ \ p_{n_1, k, n_2}^* \, \CP_k \, [2n - 2k](n-k)$$
where similarly to above we denote by
$$p_{n_1, k, n_2}: \ \ Z^{-,n_1}_{(\Bun_{T,-n})} \, \underset{\Bun_T}{\times} \bigl( \, X^{(k)} \times Z^{n_2} \bigr) \ \ \ \longto \ \ \ X^{(k)}$$
the forgetful map. The analog of Theorem \ref{main theorem for nearby cycles} then reads:

\bigskip

\begin{theorem}
\label{main theorem for nearby cycles, local version}
There exists an isomorphism of perverse sheaves
$$\gr \, \Psi (\IC_{Y^n_G})     \ \ \ \cong \ \ \ \bigoplus_{(n_1, k, n_2)} \bar{f}_{n_1, k, n_2, *}\ \widetilde \CP_{n_1, k, n_2}$$
which identifies the action of the Lefschetz-$\sl_2$ on the right hand side via the Picard-Lefschetz oscillators with the monodromy action on the left hand side. Here the direct sum runs over all triples $(n_1, k, n_2)$ of non-negative integers satisfying $n_1 + k + n_2 = n$.
\end{theorem}

\bigskip\bigskip\bigskip

\section{Geometry of the local models}
\label{Geometry of the local models}

\bigskip

\ssec{Factorization in families}
\label{Factorization}

\mbox{} \medskip

Unlike the Beilinson-Drinfeld affine Grassmannian (\cite{BD1}) or the Zastava spaces (\cite{BFGM}), the local models $Y^n$ are not literally factorizable. Instead, they are \textit{factorizable in families}, i.e., the fibers of the map $Y^n \to \BA^1$ are factorizable in a compatible way:

\bigskip

\sssec{Factorization in families}

The spaces $Y^n$ are \textit{factorizable in families} in the sense of the following lemma.

\medskip

\begin{proposition}
\label{factorization in families}
For any integers $n_1 + n_2 = n$ the natural map
$$X^{(n_1)} \stackrel{\circ}{\times} X^{(n_2)} \ \longto \ X^{(n)}$$
defined by adding effective divisors induces a cartesian square
$$\xymatrix@+10pt{
Y^{n_1} \underset{ \ \BA^1}{\stackrel{\circ}{\times}} Y^{n_2} \ar[r] \ar[d]_{\pi_{n_1} \times \pi_{n_2}}   &   Y^n \ar[d]^{\pi_n}             \\
X^{(n_1)} \stackrel{\circ}{\times} X^{(n_2)} \ar[r]                            &       X^{(n)}            \\
}$$
where the top horizontal arrow commutes with the natural maps to $\BA^1$.
\end{proposition}

\medskip

Broadly speaking, Proposition \ref{factorization in families} follows from the fact that generically on the curve $X$, the datum of a point of $Y^n$ is the trivial datum except for the determinant of the middle map $\varphi$. More precisely, Proposition \ref{factorization in families} will be a direct consequence of the following easy lemma:

\medskip

\begin{lemma}
\label{generic lemma behavior}
Let $k$ be a non-negative integer, let
$$L \longintointo E_1 \stackrel{\varphi}{\longto} E_2 \longonto \CO_{X \times S}$$
be an $S$-point of the local model $Y^k$, and let
$$d \ := \ \det(\varphi) \ \ \in \ \ \Gamma(X \times S, \CO_{X \times S}) \ = \ \Gamma(S, \CO_S) \ = \ \BA^1(S)$$
denote its image under the usual map $Y^k \to \BA^1$. Furthermore let $U \subset X \times S$ denote the dense open subscheme of $X \times S$ on which the composite map $L \to \CO_{X \times S}$ is an isomorphism. Then over $U$ the data of the above $S$-point takes the simple form
$$\CO_{U} \ \stackrel{i_1}{\longintointo} \ \CO_{U} \oplus \CO_{U} \ \stackrel{\Bigl(\begin{smallmatrix} 1&0 \\ 0& d|_U \end{smallmatrix}\Bigr)}{\longto} \ \CO_{U} \oplus \CO_{U} \ \stackrel{pr_1}{\longonto} \ \CO_{U} \, .$$
\end{lemma}

\bigskip

\begin{proof}
\noindent Composing the middle map $\varphi$ either with the rightmost or the leftmost arrow we obtain the splittings
$$\xymatrix@+10pt{
 L|_U \ar@{^(->}[r] \ar@/_2pc/[rr]^{\cong} &  E_1|_U \ar@{->>}[r] & \CO_{U} \! \! \! \! \! \! \! \! & \text{and} & \! \! \! \! \! \! \! \! L|_U \ar@{^(->}[r] \ar@/_2pc/[rr]^{\cong} &  E_2|_U \ar@{->>}[r] &  \CO_{U}
}$$
over the open subscheme $U$. These splittings in turn induce trivializations of the $\SL_2$-bundles $E_1$ and $E_2$ which are compatible with the middle map $\varphi$, so that $\varphi$ must be of the matrix form as above, but with an a priori unknown entry in the lower right corner. The fact that $\det(\varphi) = d$ on $X \times S$ however forces the entry in the lower right corner to be equal to $d|_U \in \Gamma(U, \CO_U)$.
\end{proof}

\bigskip

We can now prove Proposition \ref{factorization in families}:

\medskip

\begin{proof}[Proof of Proposition \ref{factorization in families}]
We need to construct a natural isomorphism
$$Y^{n_1} \underset{ \ \BA^1}{\stackrel{\circ}{\times}} Y^{n_2} \ \ \cong \ \ \Bigl( X^{(n_1)} \stackrel{\circ}{\times} X^{(n_2)} \Bigr) \underset{X^{(n)}}{\times} Y^n$$
which respects the forgetful maps to $X^{(n_1)} \stackrel{\circ}{\times} X^{(n_2)}$ and to $\BA^1$.
To do so, let us first define a map from the right hand side to the left hand side. Thus we are given an $S$-point
$$L \longintointo E_1 \stackrel{\varphi}{\longto} E_2 \longonto \CO_{X \times S}$$
of $Y^n$, an $S$-point $L_1 \into \CO_{X \times S}$ of $X^{(n_1)}$, and an $S$-point $L_2 \into \CO_{X \times S}$ of~$X^{(n_2)}$, such that the subsheaf $L_1 \otimes L_2 \into \CO_{X \times S}$ coincides with the subsheaf $L \into \CO_{X \times S}$ obtained from the $S$-point of $Y^n$. Let
$$d \ := \ \det(\varphi) \ \in \ \Gamma(X \times S, \CO_{X \times S}) \ = \ \Gamma(S, \CO_S) \, ,$$
and let $U$, $U_1$, $U_2$ denote the open subschemes of $X \times S$ on which the maps
$$L \longinto \CO_{X \times S}, \ \ L_1 \longinto \CO_{X \times S}, \ \ L_2 \longinto \CO_{X \times S}$$
are isomorphisms. Then by definition of the right hand side we have
$$U_1 \cap U_2 \ = \ U \ \ \ \text{and} \ \ \ U_1 \cup U_2 \ = \ X \times S \, .$$
We now define an $S$-point of~$Y^{n_1}$ by gluing together the required data on~$U_1$ and $U_2$. Namely, on the one hand we restrict the datum
$$L \longintointo E_1 \stackrel{\varphi}{\longto} E_2 \longonto \CO_{X \times S}$$
to the open subscheme $U_2$, and on the other hand we consider the datum
$$\CO_{U_1} \ \stackrel{i_1}{\longintointo} \ \CO_{U_1} \oplus \CO_{U_1} \ \stackrel{\Bigl(\begin{smallmatrix} 1&0 \\ 0& d|_{U_1} \end{smallmatrix}\Bigr)}{\longto} \ \CO_{U_1} \oplus \CO_{U_1} \ \stackrel{pr_1}{\longonto} \ \CO_{U_1}$$
over the open subscheme $U_1$. By Lemma \ref{generic lemma behavior}, these two data agree on the intersection $U_1 \cap U_2 = U$, and thus can be glued to form an $S$-point of~$Y^{(n_1)}$. We construct an $S$-point of $Y^{(n_2)}$ analogously, and by construction they together form an $S$-point of the left hand side as desired.

\medskip

We define a map from the left hand side to the right hand side in a similar fashion: Consider the $S$-points of $Y^{n_1}$ and $Y^{n_2}$ arising from a given $S$-point of the left hand side, and let $d \in \Gamma(S, \CO_S)$ be their common image in $\BA^1(S)$. These $S$-points of $Y^{n_1}$ and $Y^{n_2}$ give rise to open subschemes $U_1$ and $U_2$ of $X \times S$ defined exactly as in Lemma \ref{generic lemma behavior}, and by definition of the left hand side we have that $U_1 \cup U_2 = X \times S$. Furthermore, by Lemma~\ref{generic lemma behavior} the restriction of the data on $X \times S$ comprising the $S$-point of $Y^{n_1}$ to the intersection $U_1 \cap U_2$ agrees with the restriction of the data comprising the~$S$-point of $Y^{n_2}$. Thus the two data can be glued to form an $S$-point of $Y^n$, and we have constructed the converse map. Finally, it is immediate from the constructions that the two maps are inverse to each other and respect the forgetful maps to $X^{(n_1)} \stackrel{\circ}{\times} X^{(n_2)}$ and to $\BA^1$.
\end{proof}

\medskip

\sssec{Factorization of the fibers}
For a scalar $c \in \BA^1$ let $Y^n_c$ denote the fiber of the map $Y^n \to \BA^1$ over $c$. Thus $Y^n_{c=0}$ is equal to the $B$-locus $Y^n_B$ of $Y^n$, and $Y^n_{c=1}$ is equal to the open Zastava space ${}_0Z^n$ from Subsection \ref{Zastava spaces} above. Furthermore, since the top horizontal arrow in Proposition \ref{factorization in families} commutes with the natural maps to $\BA^1$, we find:

\medskip

\begin{corollary}
The spaces $Y^n_c$ are factorizable in the usual sense, i.e., the addition of effective divisors induces a cartesian square
$$\xymatrix@+10pt{
Y^{n_1}_c \stackrel{\circ}{\times} Y^{n_2}_c \ar[r] \ar[d]_{\pi_{n_1} \times \pi_{n_2}}   &   Y^n_c \ar[d]^{\pi_n}             \\
X^{(n_1)} \stackrel{\circ}{\times} X^{(n_2)} \ar[r]                            &       X^{(n)}            \\
}$$

\medskip

\noindent In particular, the $B$-locus $Y^n_B$ is factorizable in the usual sense.
\end{corollary}

\medskip

\ssec{Embedding, section, and contraction}
\label{Embedding, section, and contraction}

\mbox{} \medskip

In this section we construct a $\BG_m$-action on $Y^n$ which contracts $Y^n$ onto a section of the projection map $\pi: Y^n \to X^{(n)}$. This action can be constructed in various ways; here we construct it via a specific embedding of $Y^n$ into a product of Beilinson-Drinfeld affine Grassmannians which we discuss first. This embedding will also be used in Subsection \ref{Explicit equations and generalized Picard-Lefschetz families} below to derive explicit equations for the local models $Y^n$.

\medskip

\sssec{Embeddings for Zastava spaces}
\label{Embeddings for Zastava spaces}

Let $\Gr_G^n \longto X^{(n)}$ denote the Beilinson-Drinfeld affine Grassmannian for $G=\SL_2$, which parametrizes $\SL_2$-bundles on the curve~$X$ together with a trivialization away from an effective divisor of degree $n$. Recall from \cite{BFGM} that the absolute Zastava space $Z^n$ from Subsection \ref{Zastava spaces} affords a natural locally closed embedding
$$Z^n \ \longinto \ \Gr_G^n$$
which commutes with the natural projections to $X^{(n)}$. On $k$-points, this embedding associates to a point
$$L \ \longinto \ E \ \longonto \ \CO_X$$
of $Z^n$ the $\SL_2$-bundle $E$ together with the trivialization of $E$ obtained by splitting the surjection $E \onto \CO_X$ away from the zero locus of the composite map $L \to \CO_X$.

\medskip

Next consider the Zastava space parametrizing the data
$$L' \ \longintointo \ E \ \longto \ \CO_X$$
with notation as in previous sections. Note that here the map on the right is allowed to have zeroes, the line bundle on the right is fixed to be $\CO_X$, while the ``background'' line bundle $L'$ on the left is allowed to vary. Unlike the absolute Zastava space from Subsection \ref{Zastava spaces} above, this Zastava space is obtained from the relative Zastava space $_{\Bun_{T,-n}}Z^{-,n}$ by forcing the line bundle $L$ ``on the right''  to be equal to $\CO_X$; we denote this Zastava space by $\tilde{Z}^{-,n}$ for simplicity. An embedding of $\tilde Z^{-,n}$ into $\Gr_G^n$ is defined exactly as for $Z^n$.

\bigskip

\sssec{Sections for Zastava spaces}
\label{Sections for Zastava spaces}

Next we briefly review some constructions for the Zastava space $Z^n$ from \cite{BFGM}; we will use these constructions in Subsections \ref{The section for the local models $Y^n$} and \ref{The contraction for the local models $Y^n$} below to make similar constructions for the local models $Y^n$.
First, recall that the projection map
$$\pi_Z: \ Z^n \ \longto \ X^{(n)}$$
admits a natural section $s_Z$
which on $k$-points sends an effective divisor $D$ to the point
$$\CO_X(-D) \ \stackrel{i_1}{\longinto} \ \CO_X \oplus \CO_X \ \stackrel{pr_1}{\longonto} \ \CO_X$$
of the Zastava space $Z^n$.

\medskip

The case of $\tilde Z^{-,n}$ is analogous: The projection
$$\pi_{Z^-}: \ \tilde Z^{-,n} \ \longto \ X^{(n)}$$
admits a natural section $s_{Z^-}$ defined by sending an effective divisor $D \in X^{(n)}$ to the point
$$\CO_X(-D) \ \stackrel{i_1}{\longintointo} \ \CO_X(-D) \oplus \CO_X(D) \ \stackrel{pr_1}{\longto} \ \CO_X$$
of $\tilde Z^{-,n}$.

\bigskip

\sssec{Contractions for Zastava spaces}
\label{Contractions for Zastava spaces}

Next recall from \cite{MV} that any cocharacter $\check\lambda: \BG_m \to T$ naturally gives rise to an action of $\BG_m$ on the Beilinson-Drinfeld affine Grassmannian $\Gr_G^n$ which leaves the forgetful map $\Gr_G^n \to X^{(n)}$ invariant. It is shown in \cite{BFGM} that the $(-2 \check\rho)$-action of $\BG_m$ preserves the subspace $Z^n$; moreover, it contracts $Z^n$ onto the section $s_Z$, i.e., the action map extends to a map
$$\BA^1 \times Z^n \ \longto \ Z^n$$
such that the composition
$$Z^n \ = \ \{ 0 \} \times Z^n \ \longinto \ \BA^1 \times Z^n \ \longto \ Z^n$$
is equal to the composition of projection and section
$$Z^n \ \stackrel{\pi_Z}{\longto} \ X^{(n)} \ \stackrel{s_Z}{\longto} \ Z^n \, .$$

\medskip

The next lemma provides a modular interpretation of the $(- 2 \check\rho)$-action of~$\BG_m$ on $Z^n$, which will be used below; it can be proven by chasing through the definitions.

\medskip

\begin{lemma}
\label{modular action 1}
The action of an element $a \in \BG_m(S) = \Gamma(S,\CO_S)^\times$ on an $S$-point
$$L \ \stackrel{i}{\longinto} \ E \ \stackrel{p}{\longonto} \CO_X$$
of the Zastava space $Z^n$ via the $(- 2 \check\rho)$-action of $\BG_m$ yields the point
$$L \ \stackrel{a \cdot i}{\longinto} \ E \ \stackrel{\frac{1}{a} \cdot p \ }{\longonto} \ \CO_X \, .$$
\end{lemma}

\bigskip

Similarly, the $(2\check\rho)$-action of $\BG_m$ on $\Gr^n_G$ contracts $\tilde Z^{-,n}$ onto the section~$s_{Z^-}$ from Subsection \ref{Sections for Zastava spaces} above. Just as for $Z^n$ we have the following modular interpretation:

\medskip

\begin{lemma}
\label{modular action 2}
The action of an element $a \in \BG_m(S) = \Gamma(S,\CO_S)^\times$ on an $S$-point
$$L \ \stackrel{i}{\longintointo} \ E \ \stackrel{p}{\longto} \CO_X$$
of the Zastava space $\tilde Z^{-,n}$ via the $(2 \check\rho)$-action of $\BG_m$ yields the point
$$L \ \stackrel{\frac{1}{a} \cdot i}{\longintointo} \ E \ \stackrel{a \cdot p}{\longto} \ \CO_X \, .$$
\end{lemma}

\bigskip

\sssec{Embeddings for the local models $Y^n$}
\label{Embeddings for the local models $Y^n$}

Combining the embeddings of~$Z^n$ and $\tilde Z^{-,n}$ from Subsection \ref{Embeddings for Zastava spaces} above, we obtain a locally closed embedding
$$\tilde Z^{-,n} \underset{ \ X^{(n)}}{\times} Z^n \ \ \longinto \ \ \Gr_G^n \underset{ \ X^{(n)}}{\times} \Gr_G^n \, .$$

\noindent We now construct the embedding of $Y^n$ mentioned above by in turn constructing a closed immersion
$$\tau: \ Y^n \ \longinto \ \tilde Z^{-,n} \underset{ \ X^{(n)}}{\times} Z^n \, .$$
Namely, if
$$L \ \longintointo \ E_1 \ \stackrel{}{\longto} \ E_2 \ \longonto \ \CO_{X \times S}$$
is an $S$-point of $Y^n$, we can on the one hand compose the middle map $\varphi$ with the surjection on the right and obtain the $S$-point
$$L \ \longintointo \ E_1 \ \longto \ \CO_{X \times S}$$
of $\tilde Z^{-,n}$. On the other hand, composing $\varphi$ with the subbundle map on the left yields an $S$-point
$$L \ \longto \ E_2 \ \longonto \ \CO_{X \times S}$$
of $Z^n$, and by construction the two points in fact lie in the fiber product over $X^{(n)}$ above; we have thus defined the map $\tau$.

\medskip

\begin{lemma}
\label{tau closed immersion}
The map $\tau$ is a closed immersion.
\end{lemma}

\begin{proof}
Given an $S$-point of~$\tilde Z^{-,n} \times_{X^{(n)}} Z^n$, represented by the outer rhombus in the next diagram, we show that there is at most one dotted arrow $\varphi$ making both triangles commute.
$$\xymatrix@+10pt{
 & E_1 \ar[dr]^{h_1} \ar@{..>}^{\varphi}[dd] & \\
 L \ar@{^(->}[ur]^{g_1} \ar[dr]_{g_2} & & \CO_{X \times S} \\
 & E_2 \ar@{->>}[ur]_{h_2} & \\
}$$
Since the existence of such an arrow is a closed condition, this will prove the lemma. To prove the uniqueness of the dotted arrow, we form the difference $\delta: E_1 \to E_2$ of any given two such dotted arrows, and show that $\delta = 0$. Namely, by the commutativity assumptions for each dotted arrow, the map $\delta$ descends to a map $\bar\delta: E_1/L \to E_2$ whose composite with $h_2$ is $0$. Thus the map $\bar\delta$ factors through the kernel of $h_2$, which itself is the trivial line bundle, and we need to show that the resulting map $E_1/L \to \CO_{X \times S}$ is zero. 

\medskip

To do so, observe first that since $g_1$ is a subbundle map, the quotient $E_1/L$ is itself a line bundle; its restriction to any $X \times \bar{s}$ has degree $n \geq 1$. We prove the above vanishing by showing that in fact the vector space of maps
$$\Hom_{\CO_{X \times S}}(E_1/L, \CO_{X \times S}) \ = \ H^0(X \times S, (E_1/L)^*)$$
vanishes, where $(E_1/L)^*$ denotes the dual line bundle of $E_1/L$. For the latter, it suffices to show that the sheaf pushforward $R^0p_*((E_1/L)^*)$ along the projection map $p: X \times S \to S$ vanishes. By the theorem on cohomology and base change, this in turn can be checked on the geometric fibers of the projection $p$, where it holds for degree reasons.
\end{proof}

\bigskip

\sssec{The section for the local models $Y^n$}
\label{The section for the local models $Y^n$}

Next we construct a section of the projection map
$$\pi: \ Y^n \ \longto \ X^{(n)} \, .$$
First recall that an $S$-point of $X^{(n)}$ consists of a line bundle $L$ on $X \times S$ together with a map of coherent sheaves $L \to \CO_{X \times S}$ which is injective of relative degree $n$ whenever restricted to $X \times \bar{s}$ for every geometric point $\bar{s} \to S$. The latter condition automatically forces the map $L \to \CO_{X \times S}$ to be injective. Furthermore, let $L^*$ denote the dual line bundle of $L$ on $X \times S$. Then we define the section
$$s: \ X^{(n)} \ \longto \ Y^n$$
by associating to an $S$-point $L \to \CO_{X \times S}$ of $X^{(n)}$ the $S$-point
$$L \ \ \stackrel{i_1}{\longinto} \ \ L \oplus L^* \ \ \stackrel{\varphi}{\longto} \ \ \CO_{X \times S} \oplus \CO_{X \times S} \ \ \stackrel{pr_1}{\longonto} \ \ \CO_{X \times S}$$
of $Y^n$, where the map $\varphi$ in the middle is defined as the composition
$$L \oplus L^* \ \ \stackrel{pr_1}{\longonto} \ \ L \ \ \longinto \ \CO_{X \times S} \ \ \stackrel{i_1}{\longinto} \ \ \CO_{X \times S} \oplus \CO_{X \times S} \, .$$

\medskip

\noindent It is clear from the definitions that the map $s$ is indeed a section of $\pi$. Furthermore, by construction the section $s$ factors through the $B$-locus $Y^n_B$ of $Y^n$. In fact we have the following two lemmas, both of which follow easily from the definitions:

\medskip

\begin{lemma}
The section $s$ induces an isomorphism of $X^{(n)}$ with the stratum of maximal defect ${}_nY^n_B$.
\end{lemma}

\begin{lemma}
\label{agreement of sections}
The section $s_{Z^-} \times s_Z$ of the projection
$$\pi_{Z^-} \times \pi_Z: \ \ \tilde Z^{-,n} \underset{ \ X^{(n)}}{\times} Z^n \ \longto \ X^{(n)}$$
induced by the sections $s_Z$ and $s_{Z^-}$ from Subsection \ref{Sections for Zastava spaces} factors through the closed subspace $Y^n$, and in fact agrees with the section $s$.
\end{lemma}

\medskip

\sssec{The contraction for the local models $Y^n$}
\label{The contraction for the local models $Y^n$}

We now construct a $\BG_m$-action on $Y^n$ which contracts it onto the section $s$, in the sense of Subsection \ref{Contractions for Zastava spaces} above. One can construct this action in various ways; here we construct it using the embedding from Subsection~\ref{Embeddings for the local models $Y^n$}.
Namely, let us define a $\BG_m$-action on the fiber product $\Gr_G^n \times_{X^{(n)}} \Gr_G^n$
by acting on the first factor via the cocharacter $2 \check\rho$ of $G = \SL_2$ and on the second factor via the cocharacter $-2 \check\rho$.

\medskip

\begin{lemma}
\label{contraction lemma}
This $\BG_m$-action preserves the locally closed subspace $Y^n$ and contracts it onto the section $s$.
\end{lemma}

\begin{proof}
We first show that the action indeed preserves $Y^n$. By Subsection \ref{Contractions for Zastava spaces}, we need to show that an $S$-point of $\tilde Z^{-,n} \times_{X^{(n)}} Z^n$ which lies in $Y^n$ still lies in $Y^n$ after acting by an element $a \in \BG_m(S) = \Gamma(S, \CO_S)^\times$. In view of the embedding of Lemma \ref{tau closed immersion} and the modular descriptions of Lemma \ref{modular action 1} and Lemma \ref{modular action 2}, we have to show that if the outer rhombus of the diagram
$$\xymatrix@+10pt{
 & E_1 \ar[dr]^{h_1} \ar@{..>}^{\varphi}[dd] & \\
 L \ar@{^(->}[ur]^{g_1} \ar[dr]_{g_2} & & \CO_{X \times S} \\
 & E_2 \ar@{->>}[ur]_{h_2} & \\
}$$
admits a dotted arrow $\varphi$ as shown, then the same holds for the following rhombus:
$$\xymatrix@+10pt{
 & E_1 \ar[dr]^{a \cdot h_1} \ar@{..>}[dd] & \\
 L \ar@{^(->}[ur]^{\frac{1}{a} \cdot g_1} \ar[dr]_{a \cdot g_2} & & \CO_{X \times S} \\
 & E_2 \ar@{->>}[ur]_{\frac{1}{a} \cdot h_2} & \\
}$$
This can indeed be achieved by defining the dotted arrow as $a^2 \cdot \varphi$, and hence we have shown that $Y^n$ is preserved by the $\BG_m$-action. The second statement follows from the construction together with Lemma \ref{agreement of sections}.
\end{proof}

\medskip

\sssec{The contraction principle and preservation of weights}

Having constructed a $\BG_m$-action on $Y^n$ which contracts $Y^n$ onto the section $s$ of the projection map $\pi$ in the sense of Subsection \ref{Contractions for Zastava spaces} above, we arrive at the following consequences for the restriction along the section $s$. First, the well-known \textit{contraction principle} (see for example \cite[Sec. 3]{Hyperbolic localization} or \cite[Sec. 5]{BFGM}) for contracting $\BG_m$-actions states:

\begin{lemma}
\label{contraction principle}
For any $\BG_m$-monodromic object $F \in D(Y^n)$ there exists a natural isomorphism
$$s^* F \ \cong \ \pi_* F \, .$$
\end{lemma}

\medskip

Since by \cite[Sec. 5]{BBD} the $*$-pullback does not increase the weights and the $*$-pushforward does not decrease the weights, we obtain:

\medskip

\begin{corollary}
\label{preservation of purity}
Let $F \in D(Y^n)$ be $\BG_m$-monodromic, and assume in addition that $F$ is pure of some weight $w$. Then the complex $s^* F = \pi_* F$ is again pure of weight $w$.
\end{corollary}

\medskip

\ssec{Explicit equations and generalized Picard-Lefschetz families}
\label{Explicit equations and generalized Picard-Lefschetz families}

\mbox{} \medskip

In this section we use the embedding~$\tau$ from Subsection \ref{Embeddings for the local models $Y^n$} to find explicit equations for the fibers of the projection $\pi: Y^n \to X^{(n)}$. We will primarily be concerned with the fiber of~$\pi$ over the point $nx \in X^{(n)}$. The case of a general point $\sum n_k x_k \in X^{(n)}$ follows from this case via factorization.

\medskip

\sssec{Fibers of Zastava spaces}
Following for example \cite{MV}, we use the following notation for the semi-infinite orbits in the affine Grassmannian $\Gr_G = SL_2 (\mathsf{k}(\!(t)\!))/SL_2(\mathsf{k}[ \! [t]\!])$. Given any integer $i \in \BZ$ the $N(\mathsf{k}(\!(t)\!))$-orbit of the point
$$\Bigl(\begin{smallmatrix} t^i & 0 \\ 0 & \ t^{-i} \end{smallmatrix} \Bigr)$$
in $\Gr_G$ will be denoted by $S^i$, and its $N^-(\mathsf{k}(\!(t)\!))$-orbit by $T^i$. Using the modular interpretation of these orbits it is not hard to show (see for example \cite{BFGM}):

\begin{lemma}
By passing to the fibers over the point $nx \in X^{(n)}$, the embeddings of $Z^n$ and $\tilde Z^{-,n}$ into $\Gr_G^n$ from Subsection \ref{Embeddings for Zastava spaces} above induce identifications

$$Z^n|_{nx} \ \ \cong \ \ \overline{S^n} \cap T^0$$
and
$$\tilde Z^{-,n}|_{nx} \ \ \cong \ \ S^n \cap \overline{T^0} \, .$$
\end{lemma}

\bigskip

To make the above intersections of semi-infinite orbits more explicit, we will write matrix representatives for elements of $\Gr_G = SL_2 (\mathsf{k}(\!(t)\!))/SL_2(\mathsf{k}[ \! [t]\!])$. We have the following well-known lemma:

\begin{lemma}
\label{affinespaces}
The following two maps are isomorphisms:
\begin{itemize}
\item[]
\item[(a)]
$$\BA^n \ \longto \ \overline{S^n} \cap T^0$$

$$(a_{-n}, \ldots, a_{-1}) \ \longmapsto \Bigl(\begin{smallmatrix} 1 & 0 \\ \sum a_i t^i & 1 \end{smallmatrix} \Bigr)$$
\item[]

\item[(b)] 
$$\BA^n \ \longto \ S^n \cap \overline{T^0}$$

$$(b_0, \ldots, b_{n-1}) \ \longmapsto \Bigl(\begin{smallmatrix} t^n & \sum b_i t^i \\ 0 & t^{-n} \end{smallmatrix} \Bigr)$$
\end{itemize}
\end{lemma}

\bigskip

\sssec{The fibers of the local models $Y^n$}

Let $\BY^n$ denote the fiber of the projection $\pi: Y^n \to X^{(n)}$ over the point $nx \in X^{(n)}$. Let ${}_0\BY^n$ denote the \textit{defect-free} open subscheme of $\BY^n$, i.e., the open subscheme obtained by intersecting $\BY^n$ with the defect-free locus ${}_0Y^n$ of $Y^n$.

\medskip

We will now use the closed embedding $\tau$ from Subsection \ref{Embeddings for the local models $Y^n$} to find equations for $\BY^n$ and ${}_0\BY^n$. In fact, when describing the embedding on the level of fibers, the exposition seems to be clearer if one at first uses a slight variant~$\tilde \tau$ of the embedding $\tau$ where one slightly enlarges the target; we will remove this ``redundancy'' afterwards (see Corollary \ref{standardcoordinates} below).

\medskip

Namely, instead of $\tau$ we will at first use the closed embedding into the larger target
$$\tilde \tau: \ \ Y^n \ \ \longinto \ \ \, \tilde Z^{-,n} \underset{ \ X^{(n)}}{\times} Z^n \ \times \BA^1 \, ,$$
where the map to the last factor $\BA^1 = T_{adj}^+$ is the usual map $Y^n \to \BA^1$. Thus over the point $nx \in X^{(n)}$ we obtain a closed embedding
$$\BY^n \ \longinto \ (S^n \cap \overline{T^0}) \times (\overline{S^n} \cap T^0) \times \BA^1 \, .$$
Denote by $\Mat_{2\times2}$ the affine space of $2\times2$ matrices over $\mathsf{k}$, i.e., the Vinberg semigroup of $G = \SL_2$. Then one can verify directly from the modular interpretation of $Y^n$:

\bigskip

\begin{lemma}
\label{towards coordinates}
\begin{itemize}
\item[]
\item[]
\item[(a)] The above embedding identifies $\BY^n$ with the closed subscheme of the product
$$(S^n \cap \overline{T^0}) \times (\overline{S^n} \cap T^0) \times \BA^1$$
consisting of those elements $(M_1, M_2, d)$ which satisfy that
$$M_1^{-1} \bigl(\begin{smallmatrix} 1&0\\ 0&d \end{smallmatrix} \bigr) M_2 \ \in \ \Mat_{2\times2}(\mathsf{k}[\![t]\!]).$$
Note that this condition is indeed independent of the choice of representatives for $M_1$ and $M_2$.
\item[]
\item[(b)] The open subscheme $_0\BY^n$ of $\BY^n$ is obtained by additionally requiring that evaluation of the matrix
$$M_1^{-1} \bigl(\begin{smallmatrix} 1&0\\ 0&d \end{smallmatrix} \bigr) M_2 \ \in \ \Mat_{2\times2}(\mathsf{k}[\![t]\!])$$
at $t = 0$ does not yield the zero matrix $\bigl(\begin{smallmatrix} 0&0\\ 0&0 \end{smallmatrix} \bigr)$. Note that this condition is again independent of the choice of representatives for $M_1$ and $M_2$.
\end{itemize}
\end{lemma}

\bigskip

Using the isomorphisms of Lemma \ref{affinespaces} above we obtain:

\medskip

\begin{lemma}
\label{preliminary coordinates}
Via Lemma \ref{affinespaces}, consider $\BY^n$ as a closed subscheme of the affine space $\BA^n \times \BA^n \times \BA^1$ with coordinates $(b_0, \ldots, b_{n-1}, a_{-n}, \ldots, a_{-1}, d)$. Then $\BY^n$ is defined by the following $n$ equations:
\begin{align*}
a_{-n} b_0 & = d \\
a_{-n} b_1 + a_{-n +1} b_0 & = 0 \\
a_{-n} b_2 + a_{-n+1} b_1 + a_{-n+2} b_0 & = 0 \\
\vdots & \\
a_{-n} b_{n-1} + a_{-n+1} b_{n-2} + \cdots + a_{-1} b_0 & = 0
\end{align*}
(In other words, if we set $a=\sum_i a_i t^i$ and $b=\sum_i b_i t^i$, where $(a_i)_i$ and $(b_i)_i$ are indexed as above, then we require that $a \cdot b = d t^{-n}$.)

\end{lemma}

\medskip

\begin{proof}
In the notation of Lemma \ref{affinespaces} and Lemma \ref{towards coordinates}, let
$$g = a_{-n} t^{-n} + \ldots + a_{-1}t^{-1},$$
$$f = b_0 t^0 + \ldots + b_{n-1} t^{n-1},$$
$$M_1 = \bigl(\begin{smallmatrix} 1&0\\ g&1 \end{smallmatrix} \bigr),$$
$$M_2 = \Bigl(\begin{smallmatrix} t^n & f \\ 0 & t^{-n} \end{smallmatrix} \Bigr).$$
where we allow all coefficients to be valued in an arbitrary $\mathsf{k}$-algebra.
We then have
$$M_1^{-1} \bigl(\begin{smallmatrix} 1&0\\ 0&d \end{smallmatrix} \bigr) M_2 \ = \ \Bigl(\begin{smallmatrix} t^n & f \\ - g t^n & -gf + d t^{-n} \end{smallmatrix} \Bigr).$$
Observe that all matrix entries except the one in the lower right corner are integral automatically. The entry in the lower right corner equals
$$-gf + d t^{-n} \ = \ d t^{-n} - \sum_{k=-n}^{n-2} \Bigl( \sum_{i+j=k}a_i b_j \Bigr) t^k$$
Thus the integrality condition of Lemma \ref{towards coordinates} (a) translates to the asserted equations.
\end{proof}

\bigskip

Using the first equation in Lemma \ref{preliminary coordinates} one eliminates the last coordinate and obtains finally:

\medskip

\begin{corollary}
\label{standardcoordinates}
The scheme $\BY^n$ is equal to the closed subscheme of the affine space $\BA^n \times \BA^n$ defined by the following $(n-1)$ equations:
\begin{align*}
a_{-n} b_1 + a_{-n +1} b_0 & = 0 \\
a_{-n} b_2 + a_{-n+1} b_1 + a_{-n+2} b_0 & = 0 \\
\vdots & \\
a_{-n} b_{n-1} + a_{-n+1} b_{n-2} + \cdots + a_{-1} b_0 & = 0
\end{align*}
\end{corollary}

\bigskip\bigskip

For the open subscheme $_0\BY^n$ of $\BY^n$ we have:

\medskip

\begin{lemma}
\label{equations for the open part of the fiber}
The open subscheme $_0\BY^n$ of $\BY^n$ is obtained by removing from~$\BY^n$ the closed subscheme defined by additionally requiring that
$$a_{-n} = 0 = b_0$$
and
$$a_{-n+1} b_{n-1} + \cdots + a_{-2} b_2 + a_{-1} b_1 = 0.$$
\end{lemma}

\medskip

\begin{proof}
We continue to use the notation from the proof of Lemma \ref{preliminary coordinates}, and simply write out the condition stated in Lemma \ref{towards coordinates} (b). Namely, evaluating the matrix
$$\Bigl(\begin{smallmatrix} t^n & f \\ - g t^n & -gf + d t^{-n} \end{smallmatrix} \Bigr) \in \Mat_{2\times2}(\mathsf{k}[\![t]\!])$$
at $t=0$ yields the matrix
$$\Bigl(\begin{smallmatrix} 0 & b_0 \\ -a_{-n} & \ -\sum_i a_i b_i \end{smallmatrix} \Bigr),$$
and the assertion follows.
\end{proof}

\bigskip

In a similar fashion one easily checks:

\medskip

\begin{lemma}
\label{composite map and contraction in coordinates}
In terms of the coordinates of Corollary \ref{standardcoordinates}, the composite map
$$\BY^n \ \longto \ Y^n \ \longto \ \BA^1$$
sends a point $(a_i, b_j)_{i,j}$ to the scalar $a_{-n} b_0$. Furthermore, the contracting $\BG_m$-action from Lemma \ref{contraction lemma} acts quadratically on each coordinate:
$$c \, \cdot \, (a_i, b_j)_{i,j} \ \ = \ \ (c^2 a_i, c^2 b_j)_{i,j}$$
\end{lemma}

\bigskip

\sssec{Equations for other fibers}
\label{general point}
Let $\sum_{k=1}^m n_k x_k$ be a point of $X^{(n)}$, with the $x_k$ distinct. Then since the space $Y^n$ factorizes in families in the sense of Proposition \ref{factorization in families}, the fiber of $Y^n$ over the point $\sum n_k x_k$ is equal to the iterated fiber product
$$\BY^{n_1} \ \underset{\BA^1}{\times} \ \BY^{n_2} \ \underset{\BA^1}{\times} \ \cdots \ \underset{\BA^1}{\times} \ \BY^{n_m} \, ,$$
and one can write explicit equations for the latter space using Corollary~\ref{standardcoordinates} and Lemma \ref{composite map and contraction in coordinates}.

\bigskip

\sssec{The classical Picard-Lefschetz situation for $\BY^1$}
\label{The classical Picard-Lefschetz situation for the fiber}

Specializing to $n=1$ in Corollary \ref{standardcoordinates}, Lemma \ref{equations for the open part of the fiber}, and Lemma \ref{composite map and contraction in coordinates}, we see that the family $\BY^1 \to \BA^1$ recovers the classical Picard-Lefschetz family of hyperbolas degenerating to a node: The space $\BY^1$ is isomorphic to the affine plane $\BA^2$ with the two coordinates $(a_{-1}, b_0)$, and the map $\BY^1 \to \BA^1$ sends a point $(a_{-1}, b_0)$ to the product $a_{-1} \cdot b_0 \in \BA^1$. The $B$-locus $\BY^1_B$ consists of the union of the two coordinate axes.

\medskip

Next we make explicit the stratification of the $B$-locus $\BY^1_B$ induced by the defect stratification of $Y^1_B$, using the analogous notation for the strata. By Corollary \ref{stratification for local model} above the $B$-locus $\BY^1_B$ is stratified by the three strata
$${}_{(1,0,0)}\BY^1_B \, , \ \ {}_{(0,1,0)}\BY^1_B \, , \ \text{and} \ \ {}_{(0,0,1)}\BY^1_B \, .$$
In terms of the Picard-Lefschetz family, the strata ${}_{(1,0,0)}\BY^1_B$ and ${}_{(0,0,1)}\BY^1_B$ form the two axes of the node $\BY^1_B$, with the point of their intersection removed from both. Similarly, the stratum of maximal defect ${}_{(0,1,0)}{\BY^1_B}$ corresponds to the point in which the axes meet. As prescribed by Corollary \ref{stratification for local model} above the closure of either of the strata ${}_{(1,0,0)}\BY^1_B$ and ${}_{(0,0,1)}\BY^1_B$ is obtained by adding the stratum of maximal defect ${}_{(0,1,0)}{\BY^1_B}$.

\bigskip\bigskip\bigskip

\section{Nearby cycles}
\label{Nearby cycles}

In this section we prove the main theorem about nearby cycles for the local models, Theorem \ref{main theorem for nearby cycles, local version}, and hence also Theorem \ref{main theorem for nearby cycles} for $\VinBun_G$, as explained in Subsection \ref{Restatements of the main theorems for the local models} above. The general structure of the argument occupies Subsections \ref{The base case $n=1$} through \ref{The isomorphism on the stratum of maximal defect} below. However, we postpone two key statements to separate subsections later in the text, hoping that this might bring out the structure of the argument better than a logically linear proof. Moreover, each of the two statements requires proof techniques that are somewhat different from the present subsection, possibly justifying the separate treatment.

\medskip

We will prove Theorem \ref{main theorem for nearby cycles, local version} by induction on the integer $n$ appearing in its formulation. We remark that our inductive procedure will implicitly also show that the full nearby cycles are in fact unipotent; this is needed to invoke the factorization of the nearby cycles (see for example \cite[Sec. 5]{Jantzen conjectures}) during the induction step. Alternatively, it is not hard to see that the total space $\VinBun_G$ admits a $\BG_m$-action which is compatible under the natural map to $\BA^1$ with the standard action of $\BG_m$ on $\BA^1$; this in turn forces the full nearby cycles to be unipotent, yielding a different proof.

\ssec{The base case $n=1$}
\label{The base case $n=1$}

\mbox{} \medskip

We begin by establishing the base case $n=1$ of the induction by an explicit calculation involving the geometry of the local models and the equations from Subsection \ref{Explicit equations and generalized Picard-Lefschetz families} above.

\begin{proposition}
\label{base case proposition}
Theorem \ref{main theorem for nearby cycles, local version} holds for the case $n=1$.
\end{proposition}

\begin{proof}
Since $n=1$ the computation of $\gr \Psi (\IC_{Y^1_G})$ reduces to the computation of $\gr \Psi (\IC_{\BY^1_G})$, where $\BY^1_G$ denotes the $G$-locus of the fiber~$\BY^1$ of the projection map $\pi$ studied in Subsection \ref{Explicit equations and generalized Picard-Lefschetz families} above. We use the same notation as in Subsection \ref{The classical Picard-Lefschetz situation for the fiber} above for the stratification of $\BY^1_B$ induced by the defect stratification of~$Y^1_B$. We then have to show that on the $B$-locus $\BY^1_B$ there exists an isomorphism

$$\gr \, \Psi (\IC_{\BY^1_G}) \ \ \ \cong \ \ \ \Qellbar_{{}_{(1,0,0)}\overline{\BY^1_B}} \, [1](\tfrac{1}{2}) \ \oplus \ \ \Bigl( V \otimes \Qellbar_{{}_{(0,1,0)}\BY^1_B} \Bigr) \ \oplus \ \Qellbar_{{}_{(0,0,1)}\overline{\BY^1_B}} \, [1](\tfrac{1}{2})$$

\medskip

\noindent which respects the action of the Lefschetz-$\sl_2$. To see this, recall first that the explicit description of $\BY^1_B$ in coordinates in Subsection \ref{The classical Picard-Lefschetz situation for the fiber} above shows that the strata closures ${}_{(1,0,0)}\overline{\BY^1_B}$ and ${}_{(0,0,1)}\overline{\BY^1_B}$ form the two axes of the reducible node $\BY^1_B$; similarly, the stratum ${}_{(0,1,0)}{\BY^1_B}$ corresponds to the point in which the axes meet. Furthermore, by Subsection \ref{The classical Picard-Lefschetz situation for the fiber} above the family $\BY^1 \to \BA^1$ is precisely the Picard-Lefschetz family of hyperbolas, and hence the required calculation is precisely the assertion of Lemma \ref{Psi for the Picard-Lefschetz family of hyperbolas} above.
\end{proof}

\bigskip

\ssec{Reduction to the stratum of maximal defect}

\mbox{} \medskip

In the present and the next subsection we deal with the induction step from $n-1$ to $n$. The present subsection reduces the assertion of Theorem \ref{main theorem for nearby cycles, local version} to a simpler version which takes place entirely on the stratum of maximal defect~${}_nY^n_B$; the latter case will be established in the next subsection, modulo the two separate statements referred to earlier.

\bigskip

From now on we abbreviate the right hand side of Theorem \ref{main theorem for nearby cycles, local version} by
$$C_n \ \ := \ \ \bigoplus_{(n_1, k, n_2)} \bar{f}_{n_1, k, n_2, *} \ \widetilde \CP_{n_1, k, n_2} \, .$$
We first record:

\medskip

\begin{lemma}
\label{semisimplicity}
The perverse sheaf $C_n$ is semisimple. The perverse sheaf $\gr \, \Psi (\IC_{Y^n_G})$ becomes semisimple after forgetting the Weil structure.
\end{lemma}

\begin{proof}
For $C_n$ it suffices to show that the perverse sheaves $\widetilde \CP_{n_1, k, n_2}$ are semisimple since the compactified maps $\bar{f}_{n_1, k, n_2, *}$ are finite. This follows from Lemma \ref{properties of PLOs} and the definition of $\widetilde \CP_{n_1, k, n_2}$ as a shifted and twisted pullback of the Picard-Lefschetz oscillator $\CP_k$ in Subsection \ref{Placing Picard-Lefschetz oscillators on VinBun}. The assertion about $\gr \, \Psi (\IC_{Y^n_G})$ follows from Gabber's theorem (Proposition \ref{Gabber's theorem} above), together with the decomposition theorem of Beilinson, Berstein, and Deligne from \cite{BBD} for pure perverse sheaves.
\end{proof}

\medskip

Over the course of proving Theorem \ref{main theorem for nearby cycles, local version} we will establish that $\gr \, \Psi (\IC_{Y^n_G})$ is in fact semisimple also as a perverse Weil sheaf; this is of course a posteriori also a consequence of Theorem \ref{main theorem for nearby cycles, local version}.

\medskip

\sssec{Splitting according to loci of support}
\label{Splitting according to loci of support}

Lemma \ref{semisimplicity} is already sufficient to split the perverse sheaves $\gr \, \Psi (\IC_{Y^n_G})$ and $C_n$ into direct sums according to where their simple constituents are supported: Namely, we can decompose
$$\gr \, \Psi (\IC_{Y^n_G}) \ \ = \ \ \Bigl(\gr \, \Psi (\IC_{Y^n_G}) \Bigr)_{\text{on} \, {}_nY^n_B} \ \ \bigoplus \ \ \Bigl(\gr \, \Psi (\IC_{Y^n_G}) \Bigr)_{\text{not} \, \text{on} \, {}_nY^n_B}$$
where all simple constituents of the first summand are supported on the stratum ${}_nY^n_B$ and where all simple constituents of the second summand are not supported on ${}_nY^n_B$. Analogously we write
$$C_n \ \ = \ \ \bigl( C_n \bigr)_{\text{on} \, {}_nY^n_B} \ \ \bigoplus \ \ \bigl( C_n \bigr)_{\text{not} \, \text{on} \, {}_nY^n_B} \, .$$
By definition both direct sum decompositions are respected by the action of the Lefschetz-$\sl_2$. Thus to prove Theorem \ref{main theorem for nearby cycles, local version}, it suffices to construct

\medskip

(a) \textit{the isomorphism on the stratum of maximal defect}
$$\Bigl(\gr \, \Psi (\IC_{Y^n_G}) \Bigr)_{\text{on} \, {}_nY^n_B} \ \cong \ \bigl( C_n \bigr)_{\text{on} \, {}_nY^n_B}$$

(b) \textit{the isomorphism away from the stratum of maximal defect}
$$\Bigl(\gr \, \Psi (\IC_{Y^n_G}) \Bigr)_{\text{not} \, \text{on} \, {}_nY^n_B} \ \cong \ \bigl( C_n \bigr)_{\text{not} \, \text{on} \, {}_nY^n_B}$$
where both isomorphisms need to respect the action of the Lefschetz-$\sl_2$. The existence of the isomorphism (b) away from the stratum of maximal defect follows directly from the induction hypothesis, as we explain in the next paragraph. For the remainder of Section \ref{Nearby cycles} we will then be concerned with establishing the isomorphism (a) on the stratum of maximal defect.

\medskip

\sssec{The isomorphism away from the stratum of maximal defect}
\label{The isomorphism away from the stratum of maximal defect}

Recall from Corollary \ref{stratification for local model} the open subscheme ${}_{\leq (n-1)}Y^n_B$ of the $B$-locus $Y^n_{B}$ defined by allowing the defect degree to be at most $n-1$. Then as in Subsection \ref{Restatements of the main theorems for the local models} the induction hypothesis implies the validity of Theorem \ref{main theorem for nearby cycles, local version} after restricting to the open subscheme~${}_{\leq (n-1)}Y^n_B$, i.e.:

\medskip

\begin{lemma}
\label{interplay away from worst stratum}
Assume the main theorem about nearby cycles holds for the integer $n-1$.
Then on the open subscheme ${}_{\leq (n-1)}Y^n_B$ of $Y^n_B$ there exists an isomorphism of perverse sheaves
$$\gr \, \Psi (\IC_{Y^n_G}) \Big|_{{}_{\leq (n-1)}Y^n_B}^{*}    \ \ \ \cong \ \ \ \bigoplus_{(n_1, k, n_2)} \Bigl( \bar{f}_{n_1, k, n_2, *}\ \widetilde \CP_{n_1, k, n_2} \Bigr) \Big|_{{}_{\leq (n-1)}Y^n_B}^{*}$$
which is compatible with the action of the Lefschetz-$\sl_2$.
\end{lemma}

\medskip

\begin{proof}
This follows via a standard argument often referred to as the ``interplay principle''; see for example \cite[Sec. 3, 8.1]{BFGM}, \cite[Sec. 4.3, 4.6, 4.7]{BG2}. We now outline the argument in our case.
By the induction hypothesis, the desired isomorphism has already been constructed up to defect $n-1$, i.e., on $Y^{n-1}_B$ and on $\sideset{_{\leq (n-1)}}{_{G,B}}\VinBun$, and we need to construct it on ${}_{\leq (n-1)}Y^n_B$. For a sufficiently large positive integer $m$, the natural map $Y^m_{rel} \to \VinBun_G$ is smooth; as this map is also compatible with the natural maps to the affine line $\BA^1$ and the defect stratifications of its source and target, the pullback along this map of the isomorphism on $\sideset{_{\leq (n-1)}}{_{G,B}}\VinBun$ yields the desired isomorphism on ${}_{\leq (n-1)}Y^m_{rel, B}$. By the smooth-local equivalence between $Y^m$ and $Y^m_{rel}$, we obtain the desired isomorphism on ${}_{\leq (n-1)}Y^m_B$, although the integer $m$ may be larger than the desired integer $n$. To deduce the assertion for the integer $n$ instead of the integer $m$, note that the etale factorization map
$$Y^n_B \ \stackrel{\circ}{\times} \ Y^{m-n}_B \ \longto \ Y^m_B$$
restricts to a map on open subspaces
$${}_{\leq (n-1)}Y^n_B \ \stackrel{\circ}{\times} \ {}_0Y^{m-n}_B \ \longto \ {}_{\leq (n-1)}Y^m_B \ ;$$
pulling back the isomorphism from ${}_{\leq (n-1)}Y^m_B$ along the latter map then yields the desired isomorphism on ${}_{\leq (n-1)}Y^n_B$ since both sides of the isomorphism are compatible with the factorization structure and constant along ${}_0Y^{m-n}_B$.
\end{proof}

\medskip

Combining Lemma \ref{interplay away from worst stratum} and Lemma \ref{semisimplicity} above, we already know that the restriction of $(\gr \, \Psi (\IC_{Y^n_G}))_{\text{not} \, \text{on} \, {}_nY^n_B}$ to ${}_{\leq (n-1)}Y^n_B$ is semisimple. Since by definition none of the simple constituents of $(\gr \, \Psi (\IC_{Y^n_G}))_{\text{not} \, \text{on} \, {}_nY^n_B}$ are supported on the complement of ${}_{\leq (n-1)}Y^n_B$, it must in fact be equal to the intermediate extension of its restriction to ${}_{\leq (n-1)}Y^n_B$. Thus applying the intermediate extension functor to the isomorphism in Lemma \ref{interplay away from worst stratum} yields the desired isomorphism (b) above.

\bigskip

\ssec{The isomorphism on the stratum of maximal defect}
\label{The isomorphism on the stratum of maximal defect}

\mbox{} \medskip

Recall from Corollary \ref{stratification for local model} that the stratum of maximal defect ${}_nY^n_B$ is canonically identified with the symmetric power $X^{(n)}$ of the curve $X$ via the natural projection map $Y^n \to X^{(n)}$. Throughout this section we will identify ${}_nY^n_B$ and $X^{(n)}$ without further mention. Observe furthermore that by definition of $C_n$ we have
$$\bigl( C_n \bigr)_{\text{on} \, {}_nY^n_B} \ \ = \ \ \CP_n$$
where $\CP_n$ denotes the $n$-th Picard-Lefschetz oscillator as in Subsection \ref{Definition of Picard-Lefschetz oscillators} above.

\medskip

\medskip

\begin{lemma}
\label{factorization structures on both sides}
The objects $(\gr \, \Psi (\IC_{Y^n_G}))_{\text{on} \, {}_nY^n_B}$ and $(C_n)_{\text{on} \, {}_nY^n_B}$ admit natural factorization structures which respect the action of the Lefschetz-$\sl_2$.
\end{lemma}

\begin{proof}
For $(C_n)_{\text{on} \, {}_nY^n_B} = \CP_n$ this was already dealt with in Lemma \ref{external exterior powers} and Lemma \ref{properties of PLOs} above. We now prove the assertion for $(\gr \, \Psi (\IC_{Y^n_G}))_{\text{on} \, {}_nY^n_B}$. First, the factorization-in-families of the local models $Y^n$ from Proposition \ref{factorization in families} above, together with the compatibility of $\gr \Psi$ with fiber products (see for example \cite[Sec. 5]{Jantzen conjectures}), shows that on $Y^{n_1}_B \overset{\circ}{\times} Y^{n_2}_B$ we have:
$$\gr \, \Psi(\IC_{Y^n_G}) \Big|^*_{Y^{n_1}_B \overset{\circ}{\times} Y^{n_2}_B} \ \ = \ \ \gr \, \Psi(\IC_{Y^{n_1}_G}) \, \overset{\circ}{\boxtimes} \, \gr \, \Psi(\IC_{Y^{n_2}_G})$$
Here the left hand side denotes the $*$-pullback of $\gr \, \Psi(\IC_{Y^n_G})$ along the etale factorization map
$$Y^{n_1}_B \overset{\circ}{\times} Y^{n_2}_B \ \longto \ Y^n_B \, .$$
The above identification respects the action of the Lefschetz-$\sl_2$ due to the comment in Proposition \ref{factorization in families} about compatibility with respect to maps to~$\BA^1$.

\medskip

Next we claim that the above identification in fact induces an identification of the desired summands:
$$(\gr \, \Psi(\IC_{Y^n_G}))_{\text{on} \, {}_nY^n_B} \Big|^*_{Y^{n_1}_B \overset{\circ}{\times} Y^{n_2}_B} \, = \, (\gr \, \Psi(\IC_{Y^{n_1}_G}))_{\text{on} \, {}_{n_1}Y^{n_1}_B} \, \overset{\circ}{\boxtimes} \, (\gr \, \Psi(\IC_{Y^{n_2}_G}))_{\text{on} \, {}_{n_2}Y^{n_2}_B}$$
Indeed, using the identification without requirements on the support and the fact that the stratum of maximal defect ``factorizes'' in the sense that the square
$$\xymatrix@+10pt{
{}_{n_1}Y^{n_1}_B \stackrel{\circ}{\times} {}_{n_2}Y^{n_2}_B \ar[r] \ar[d]   &   {}_nY^n_B \ar[d]           \\
Y^{n_1}_B \stackrel{\circ}{\times} Y^{n_2}_B \ar[r]   &   Y^n_B
}$$

\medskip

\noindent is cartesian, one can identify both the left hand side and the right hand side with the direct summand of the perverse sheaf
$$\gr \, \Psi(\IC_{Y^n_G}) \Big|^*_{Y^{n_1}_B \overset{\circ}{\times} Y^{n_2}_B}$$
consisting of those simple constituents supported on ${}_{n_1}Y^{n_1}_B \stackrel{\circ}{\times} {}_{n_2}Y^{n_2}_B$.
\end{proof}

\medskip

\sssec{Simple constituents supported on the diagonal}
Let $\Delta_X$ denote the main diagonal
$$\Delta_X = X \longinto X^{(n)}$$
of the symmetric power $X^{(n)}$. The following lemma is crucial to our approach to the proof of Theorem \ref{main theorem for nearby cycles, local version}:

\begin{lemma}
\label{no summands over the diagonal lemma}
For $n \geq 2$, none of the simple constituents of the perverse sheaves $(\gr \, \Psi (\IC_{Y^n_G}))_{\text{on} \, {}_nY^n_B}$ and $(C_n)_{\text{on} \, {}_nY^n_B}$ are supported on the diagonal $\Delta_X$.
\end{lemma}

\medskip

For $(\gr \, \Psi (\IC_{Y^n_G}))_{\text{on} \, {}_nY^n_B}$ the proof of Lemma \ref{no summands over the diagonal lemma} is the topic of Subsection~\ref{Fighting simples on the main diagonal} below. The case of $(C_n)_{\text{on} \, {}_nY^n_B}$ however follows directly from the definitions: Since $(C_n)_{\text{on} \, {}_nY^n_B} = \CP_n$, this follows from the fact that the Picard-Lefschetz oscillator $\CP_n$ is the intermediate extension of a local system on the disjoint locus of $X^{(n)}$ by Lemma \ref{properties of PLOs}.

\medskip

In the next lemma and below \textit{the union of all diagonals in}~$X^{(n)}$ refers to the natural closed subscheme of $X^{(n)}$ complementary to the disjoint locus of $X^{(n)}$. From the definition of factorizability one verifies:

\begin{lemma}
\label{only generic supports}
Let $F_n$ be a factorizable collection of perverse sheaves on $X^{(n)}$, and assume that for each $n \geq 2$ none of the simple constituents of $F_n$ is supported on the main diagonal $\Delta_X$. Then for any $n \geq 2$ none of the simple constituents of $F_n$ is supported on the union of all diagonals in~$X^{(n)}$. In particular, none of the simple constituents of $(\gr \, \Psi (\IC_{Y^n_G}))_{\text{on} \, {}_nY^n_B}$ and $(C_n)_{\text{on} \, {}_nY^n_B}$ is supported on the union of all diagonals in $X^{(n)}$.
\end{lemma}

\medskip

\sssec{Generic agreement}

Lemma \ref{only generic supports} above shows that it suffices to construct the desired isomorphism (a) above after restricting $(\gr \, \Psi (\IC_{Y^n_G}))_{\text{on} \, {}_nY^n_B}$ and $(C_n)_{\text{on} \, {}_nY^n_B}$ to the disjoint locus $\overset{\circ}{X} {}^{(n)}$ of the symmetric power. Indeed, by Lemma \ref{only generic supports} both perverse sheaves are the intermediate extensions of their restrictions to the disjoint locus, and hence the desired isomorphism can be obtained by intermediate extension as well. As a first step towards the isomorphism on the disjoint locus, we construct the following weaker version. Let
$$\add_{disj}: \ \overset{\circ}{X} {}^n \ \longto \overset{\circ}{X} {}^{(n)}$$
denote the addition map from the disjoint locus of the cartesian product to the disjoint locus of the symmetric product of the curve $X$. Then directly from the definition of factorization we obtain:

\begin{lemma}
\label{isomorphisms after pullback}
The isomorphism for $n=1$ from Subsection \ref{The base case $n=1$} and the factorization structure from Lemma \ref{factorization structures on both sides} above together yield isomorphisms of the pullbacks
$$\add_{disj}^* \Bigl(\gr \, \Psi (\IC_{Y^n_G}) \Bigr)_{\text{on} \, {}_nY^n_B} \ \ \cong \ \ \add_{disj}^* \bigl( C_n \bigr)_{\text{on} \, {}_nY^n_B}$$
for any $n \geq 2$ which respect the action of the Lefschetz-$\sl_2$.
\end{lemma}

\bigskip

Since the addition map $\add_{disj}$ is a torsor for the symmetric group $S_n$, we note that the pullbacks $\add_{disj}^*(\gr \, \Psi (\IC_{Y^n_G}))_{\text{on} \, {}_nY^n_B}$ and $\add_{disj}^*(C_n)_{\text{on} \, {}_nY^n_B}$ carry natural $S_n$-equivariant structures. Thus to construct the desired isomorphism on the stratum of maximal defect, it suffices to prove that the isomorphism between the pullbacks constructed in Lemma \ref{isomorphisms after pullback} above in fact respects the $S_n$-equivariant structures.

\medskip

This is one of the key computations of the present article, and is in fact part of how the Picard-Lefschetz oscillators were found in the first place. Fortunately, in the present situation it suffices to verify the case $n=2$:

\begin{lemma}
\label{transpositions generate}
To show that the isomorphisms
$$\add_{disj}^* \Bigl(\gr \, \Psi (\IC_{Y^n_G}) \Bigr)_{\text{on} \, {}_nY^n_B} \ \ \cong \ \ \add_{disj}^* \bigl( C_n \bigr)_{\text{on} \, {}_nY^n_B}$$
constructed in Lemma \ref{isomorphisms after pullback} above respect the natural $S_n$-equivariant structures for all $n \geq 2$, it in fact suffices to verify the case $n=2$.
\end{lemma}

\begin{proof}
Since the symmetric group $S_n$ is generated by transpositions, it suffices to verify that the above isomorphisms respect the equivariant structure for any transposition in $S_n$. Without loss of generality we may assume the transposition under consideration interchanges the elements $1,2 \in \{1, \ldots, n\}$. Factoring the map $\add_{disj}$ as the composition
$$\overset{\circ}{X} {}^{n} \ \longto \ \overset{\circ}{X} {}^{(2)} \times \overset{\circ}{X} {}^{n-2} \ \longto \ \overset{\circ}{X} {}^{(n)}$$
and using that the isomorphisms in Lemma \ref{isomorphisms after pullback} are constructed via the factorization structures then reduces the assertion to the case $n=2$.
\end{proof}

\medskip

The required assertion in the case $n=2$ will be dealt with in Subsection \ref{Finding the Picard-Lefschetz oscillators} and Subsection~\ref{Picard-Lefschetz oscillators via intersection cohomology} below. Namely, we will give two different proofs, one via an abstract calculation in the Grothendieck group, and another one via a direct computation of the IC-sheaf of the local model $Y^2$ based on the explicit geometry available from Section~\ref{Geometry of the local models} above. This completes the proof of the main theorem about nearby cycles for the local models, Theorem \ref{main theorem for nearby cycles, local version}, and hence also Theorem \ref{main theorem for nearby cycles} for $\VinBun_G$, modulo Subsections \ref{Fighting simples on the main diagonal} through \ref{Picard-Lefschetz oscillators via intersection cohomology} below.

\bigskip

\ssec{Fighting simples on the main diagonal}
\label{Fighting simples on the main diagonal}

\mbox{} \medskip

In this subsection we prove Lemma \ref{no summands over the diagonal lemma} for $(\gr \, \Psi (\IC_{Y^n_G}))_{\text{on} \, {}_nY^n_B}$. As the lemma is part of the induction step of the inductive proof of the main theorem, Theorem \ref{main theorem for nearby cycles, local version}, we are allowed to assume the validity of Theorem \ref{main theorem for nearby cycles, local version} for the integer $n-1$ in the course of the proof of Lemma \ref{no summands over the diagonal lemma}. Using the inductive hypothesis, the question about simples on the diagonal can be translated into a similar question about the cohomology of the Zastava spaces; as remarked earlier, answering this question also establishes Theorem \ref{ij} above. In the next two subsections we first discuss a description of the cohomology of the Zastava spaces as well as its implications for the local models, and only then proceed to the actual proof.

\bigskip

\sssec{Compactly supported cohomology of open Zastava spaces}
\label{Compactly supported cohomology of open Zastava spaces}

As in Subsection \ref{Zastava spaces} let ${}_0Z^n$ denote the open Zastava space, and as before let
$$\pi_Z: \ {}_0Z^n \ \longto \ X^{(n)}$$
denote the projection map. We now record some information about the object
$$\widetilde \Omega_n \ := \ \pi_{Z,!} \bigl( \IC_{{}_0Z^n}\bigr) \ = \ \pi_{Z,!} \bigl( (\Qellbar)_{{}_0Z^n}[\dim_{{}_0Z^n}](\tfrac{1}{2} \dim_{{}_0Z^n})\bigr)$$
and will then apply it to the proof of Lemma \ref{no summands over the diagonal lemma}. In fact, we will only need to understand $\widetilde \Omega_n$ on a fairly coarse level, namely on the level of the Grothendieck group. An expression of $\widetilde \Omega_n$ on the level of the Grothendieck group can fortunately be extracted from the work \cite{BG2} of Braverman and Gaitsgory. The study of $\widetilde \Omega_n$ as an object of the derived category is much more involved, and has been carried out by Sam Raskin in the forthcoming article \cite{The geometric principal series category}.

\medskip

To state the description of $\widetilde\Omega_n$ in the Grothendieck group, let
$$\add: \ X^{(i)} \times X^{(j)} \ \longto \ X^{(n)}$$
denote the addition map of effective divisors as before; however, unlike before, here we do not restrict to the disjoint locus of the product $X^{(i)} \times X^{(j)}$. Furthermore, using the notation from Subsection \ref{sssec External exterior powers} we denote by
$$\Lambda^{(j)} ({\Qellbar}_X)[j](j)$$
the $j$-th external exterior power of the constant local system on the curve~$X$, shifted and twisted as indicated. The following description of $\widetilde \Omega_n$ in the Grothendieck group then follows directly from Corollary 4.5 of \cite{BG2}:

\medskip

\begin{lemma}
\label{Omega tilde in the Grothendieck group}
In the Grothendieck group on $X^{(n)}$ we have:
$$\widetilde \Omega_n \ \ = \ \ \sum_{i+j = n} add_* \Bigl( {\Qellbar}_{X^{(i)}} \; \boxtimes \ \Lambda^{(j)} ({\Qellbar}_X)[j](j) \Bigr)$$
Here the sum runs over all pairs of integers $(i,j)$ with $0 \leq i, j, \leq n$ and $i+j=n$.
\end{lemma}

\bigskip

\sssec{Stalks of the extension of the constant sheaf}

We now explain how the object $\widetilde \Omega_n$ arises in a sheaf-theoretic computation on the local models $Y^n$. To do so, let
$$j_G: \ Y^n_G \ \longinto \ Y^n$$
denote the open inclusion of the $G$-locus of $Y^n$, and recall from Subsection \ref{The section for the local models $Y^n$} that the inclusion of the stratum of maximal defect ${}_nY^n_B \into Y^n$ agrees with the section $s: X^{(n)} \into Y^n$ under the identification ${}_nY^n_B = X^{(n)}$. Furthermore let $H^*_c(\BA^1 \setminus \{0\})$ denote the compactly supported cohomology of $\BA^1 \setminus \{0\}$. Then, using the geometry of the local models $Y^n$ discussed in Section~\ref{Geometry of the local models} above, we can now prove:

\medskip

\begin{proposition}
\label{omega tilde proposition}
$$s^!  \, j_{G,!}  \, \IC_{Y^n_G} \ \ = \ \ \widetilde \Omega_n \, \otimes \, H^*_c(\BA^1 \setminus \{0\})[1](\tfrac{1}{2})$$
\end{proposition}

\medskip

\begin{proof}
Let $\pi_G$ denote the restriction of the projection map $\pi: Y^n \to X^{(n)}$ to the $G$-locus $Y^n_G$. Then $j_{G,!} \IC_{Y^n_G}$ is $\BG_m$-equivariant for the $\BG_m$-action from Subsection \ref{The contraction for the local models $Y^n$} above, and hence the contraction principle from Lemma \ref{contraction principle} above yields:
$$s^!  \, j_{G,!}  \, \IC_{Y^n_G} \ = \ \pi_! \, j_{G,!}  \, \IC_{Y^n_G} \ = \ \pi_{G,!} \, \IC_{Y^n_G}$$

\medskip

\noindent On the other hand, in terms of the product decomposition of $Y^n_G$ from Lemma \ref{G-locus decomposition for the local models} above we have
$$\IC_{Y^n_G} \ \ = \ \ \IC_{{}_0Z^n} \, \boxtimes \, \IC_{\BA^1 \setminus \{0\}} \, .$$
Thus the compatibility of the projection maps in Lemma \ref{G-locus decomposition for the local models} implies that
$$\pi_{G,!} \, \IC_{Y^n_G} \ \ = \ \ \pi_{Z,!} (\IC_{{}_0Z^n}) \, \otimes \, H^*_c(\BA^1 \setminus \{0\})[1](\tfrac{1}{2}) \ \ = \ \ \widetilde \Omega_n \, \otimes \, H^*_c(\BA^1 \setminus \{0\})[1](\tfrac{1}{2}) \, ,$$
as desired.
\end{proof}

\bigskip

\sssec{The proof}
\label{no summands proof}

We can now proceed to the proof of Lemma \ref{no summands over the diagonal lemma} for the perverse sheaf $(\gr \, \Psi (\IC_{Y^n_G}))_{\text{on} \, {}_nY^n_B}$. In line with the appearance of the lemma in the induction step of the inductive proof of the main theorem, Theorem \ref{main theorem for nearby cycles, local version}, we are allowed to assume the validity of Theorem \ref{main theorem for nearby cycles, local version} for the integer $n-1$ in the course of the proof.

\bigskip

\begin{proof}[Proof of Lemma \ref{no summands over the diagonal lemma} for $(\gr \, \Psi (\IC_{Y^n_G}))_{\text{on} \, {}_nY^n_B}$]

\mbox{} \medskip

Let $i_B$ denote the inclusion of the $B$-locus $Y^n_B$ into $Y^n$, and as before let~$j_G$ denote the open immersion of the $G$-locus $Y^n_G$ into $Y^n$. Then on $Y^n_B$ the usual triangle for the map $\can: \Psi \to \Phi$ applied to the object $j_{G,*} \IC_{Y^n_G}$ takes the form
$$i_B^* \, j_{G,*} \, \IC_{Y^n_G} \, [-1](-\tfrac{1}{2}) \ \longto \ \Psi(\IC_{Y^n_G}) \ \stackrel{N}{\longto} \ \Psi(\IC_{Y^n_G})(-1) \ \stackrel{+1}{\longto} \, .$$

\smallskip

\noindent To better understand the object $(\gr \, \Psi (\IC_{Y^n_G}))_{\text{on} \, {}_nY^n_B}$ we will use the $*$-pullback of the above triangle along $s$, i.e., the triangle
$$s^* \, j_{G,*} \, \IC_{Y^n_G} \, [-1](-\tfrac{1}{2}) \ \longto \ s^* \, \Psi(\IC_{Y^n_G}) \ \stackrel{N}{\longto} \ s^* \, \Psi(\IC_{Y^n_G})(-1) \ \stackrel{+1}{\longto} \, .$$

\smallskip

\noindent More precisely, we will exploit the relation the latter triangle induces in the Grothendieck group of perverse sheaves on $X^{(n)}$. Namely, let the image of $s^* \, \Psi(\IC_{Y^n_G})$ in the Grothendieck group on $X^{(n)}$ be expressed uniquely as a minimal $\BZ$-linear combination of simple perverse sheaves. Then we claim that none of the simple perverse sheaves occurring in this expression is supported on the main diagonal~$\Delta_X$ of $X^{(n)}$. Indeed, since the third term of the last triangle is a non-trivial twist of the middle term $s^* \, \Psi(\IC_{Y^n_G})$, it suffices to prove the analogous claim for the first term. However, by Proposition \ref{omega tilde proposition} above, the first term is Verdier dual to the object
$$\widetilde \Omega_n \, \otimes \, H^*_c(\BA^1 \setminus \{0\})[2](1).$$
Thus it in turn suffices to show the analogous claim for $\widetilde \Omega_n$: If the image of $\widetilde \Omega_n$ in the Grothendieck group of perverse sheaves on $X^{(n)}$ is expressed as a minimal $\BZ$-linear combination of simple perverse sheaves, then none of the simples occurring in this expression is supported on the main diagonal~$\Delta_X$. This last claim however follows directly from the explicit description of $\widetilde \Omega_n$ on the level of the Grothendieck group in Lemma \ref{Omega tilde in the Grothendieck group} above: Each summand in this description is equal to the intermediate extension to $X^{(n)}$ of a local system on the disjoint locus of~$X^{(n)}$.

\medskip

We have now established that none of the simple perverse sheaves occurring in the minimal description of the image of $s^* \, \Psi(\IC_{Y^n_G})$ in the Grothendieck group on $X^{(n)}$ is supported on the main diagonal $\Delta_X$. From this we now deduce that the same holds for the image of $(\gr \, \Psi (\IC_{Y^n_G}))_{\text{on} \, {}_nY^n_B}$ in the Grothendieck group; since $(\gr \, \Psi (\IC_{Y^n_G}))_{\text{on} \, {}_nY^n_B}$ is a perverse sheaf, this completes the proof.

\medskip

To make the required deduction, note first that on the level of the Grothen-dieck group the objects $\Psi(\IC_{Y^n_G})$ and $\gr \, \Psi (\IC_{Y^n_G})$ coincide. Thus any simple perverse sheaf occurring in the minimal description of the image of the direct sum
$$s^* \, \gr \, \Psi (\IC_{Y^n_G}) \ \ = \ \ \Bigl(\gr \, \Psi (\IC_{Y^n_G}) \Bigr)_{\text{on} \, {}_nY^n_B} \ \ \bigoplus \ \ s^* \, \Bigl(\gr \, \Psi (\IC_{Y^n_G}) \Bigr)_{\text{not} \, \text{on} \, {}_nY^n_B}$$
in the Grothendieck group of $X^{(n)}$ cannot be supported on the main diagonal~$\Delta_X$. Hence to establish the desired claim for the first summand, it suffices to establish the analogous claim for the second summand. The second summand can however be dealt with via the induction hypothesis: Since we are allowed to assume the validity of Theorem \ref{main theorem for nearby cycles, local version} for the integer $n-1$, we may apply Lemma \ref{interplay away from worst stratum} above and may hence make use of the identification
$$\Bigl(\gr \, \Psi (\IC_{Y^n_G}) \Bigr)_{\text{not} \, \text{on} \, {}_nY^n_B} \ \cong \ \bigl( C_n \bigr)_{\text{not} \, \text{on} \, {}_nY^n_B}$$
from Subsection \ref{The isomorphism away from the stratum of maximal defect} above. This in turn reduces the assertion about the second summand to the analogous assertion for the object $s^* \, (C_n)_{\text{not} \, \text{on} \, {}_nY^n_B}$. The next lemma however provides a strengthening of this last assertion, and therefore completes the proof of Lemma \ref{no summands over the diagonal lemma}.
\end{proof}

\medskip

\begin{lemma}
\label{restriction along s of the usual summands}
For each triple $(n_1, k, n_2)$ as in Theorem \ref{main theorem for nearby cycles, local version} the object $s^* \, \bar{f}_{n_1, k, n_2, *}\ \widetilde \CP_{n_1, k, n_2}$ is a direct sum of cohomologically shifted simple perverse sheaves on $X^{(n)}$, none of which is supported on the main diagonal $\Delta_X$.
\end{lemma}

\begin{proof}
First observe that the square
$$\xymatrix@+10pt{
X^{(n_1)} \times X^{(k)} \times X^{(n_2)} \, \ar^{\text{section} \ \ \ \ \ \ }[r] \ar^{\add}[d] & \, Z^{-,n_1}_{(\Bun_{T,-n})} \, \underset{\Bun_T}{\times} \bigl( \, X^{(k)} \times Z^{n_2} \bigr) \ar^{\bar{f}_{n_1, k, n_2}}[d] \\
X^{(n)} \ \ar^{s}[r] & \ Y^n_B \\
}$$
is cartesian, where the top arrow is the natural map formed by combining the three section maps $s_{Z^-}$, $s$, and $s_Z$ from Subsection \ref{Embedding, section, and contraction} above, and where the left horizontal map is the addition map of effective divisors. Next, from the definition of $\widetilde \CP_{n_1, k, n_2}$ and the properness of $\bar{f}_{n_1, k, n_2}$ we obtain that
$$s^* \, \bar{f}_{n_1, k, n_2, *}\ \widetilde \CP_{n_1, k, n_2} \ \ = \ \ \add_* \, (\Qellbar \boxtimes \, \CP_k \boxtimes \Qellbar) \, [2n - 2k](n-k) \, .$$
Then the finiteness of the map $\add$ and the properties of $\CP_k$ stated in Lemma \ref{properties of PLOs} together yield the assertion.
\end{proof}

\bigskip

\ssec{Finding the Picard-Lefschetz oscillators}
\label{Finding the Picard-Lefschetz oscillators}

\mbox{} \medskip

In this subsection we give the first proof that the isomorphism
$$\add_{disj}^* \Bigl(\gr \, \Psi (\IC_{Y^2_G}) \Bigr)_{\text{on} \, {}_2Y^2_B} \ \ \cong \ \ \add_{disj}^* \bigl( C_2 \bigr)_{\text{on} \, {}_2Y^2_B}$$
constructed in Lemma \ref{isomorphisms after pullback} indeed respects the natural $S_2$-equivariant structure; this completes the proof of Theorem \ref{main theorem for nearby cycles, local version}. In Subsection \ref{Picard-Lefschetz oscillators via intersection cohomology} below we give a second proof. The first proof is an abstract calculation in the Grothendieck group and essentially a refinement of the arguments of Subsection \ref{Fighting simples on the main diagonal}, exploiting the specific expression for $\widetilde\Omega$ given in Lemma \ref{Omega tilde in the Grothendieck group} above. The second proof does not require this specific expression, but instead deduces the required assertion from an intersection cohomology computation for the space $Y^2$ which relies on the geometry of the local models developed in Section~\ref{Geometry of the local models} above; this is in fact how the Picard-Lefschetz oscillators were found originally. We include this second proof as it also provides an example of a direct IC-sheaf computation without passing through the nearby cycles, and might illuminate how one can work with the local models in very explicit terms.

\bigskip

\sssec{The first proof of the compatibility}
We will show that the images of $(\gr \, \Psi (\IC_{Y^2_G}))_{\text{on} \, {}_2Y^2_B}$ and $(C_2)_{\text{on} \, {}_2Y^2_B} = \CP_2$ in the Grothendieck group of perverse sheaves on $X^{(2)}$ agree; since both are in fact semisimple perverse sheaves, this will prove the claim. When writing expressions in the Grothendieck group on $X^{(2)}$, we will for notational simplicity denote by $\Qellbar$ and by $\sign$ the IC-extensions to $X^{(2)}$ of the constant and sign local systems on the disjoint locus of $X^{(2)}$. First, from the definition one finds that
$$\CP_2 \ = \ \sign(1) + \sign(0) + \sign(-1) + \Qellbar(0)$$
in the Grothendieck group.

\medskip

To compute $(\gr \, \Psi (\IC_{Y^2_G}))_{\text{on} \, {}_2Y^2_B}$ we exploit the relation in the Grothendieck group induced by the triangle
$$s^* \, j_{G,*} \, \IC_{Y^2_G} \, [-1](-\tfrac{1}{2}) \ \longto \ s^* \, \Psi(\IC_{Y^2_G}) \ \stackrel{N}{\longto} \ s^* \, \Psi(\IC_{Y^2_G})(-1) \ \stackrel{+1}{\longto}$$
from the proof in Subsection \ref{no summands proof} above in the case $n=2$. Namely, as a first step we use this triangle to show that in the Grothendieck group the difference
$$(\gr \, \Psi (\IC_{Y^2_G}))_{\text{on} \, {}_2Y^2_B} \ - \ (\gr \, \Psi (\IC_{Y^2_G}))_{\text{on} \, {}_2Y^2_B}(-1)$$
is equal to
$$\Qellbar(0) - \Qellbar(-1) + \sign(1) - \sign(-2) \, .$$
To see this, we need to compute the images of the first term of the triangle and of $(\gr \, \Psi (\IC_{Y^2_G}))_{\text{not} \, \text{on} \, {}_2Y^2_B}$ in the Grothendieck group: For the image of the first term of the triangle we find the expression
$$\Qellbar(1) - 2 \Qellbar(0) + \Qellbar(-1) - \sign(0) + 2 \sign(-1) - \sign(-2)$$
by Lemma \ref{Omega tilde in the Grothendieck group}, Lemma \ref{omega tilde proposition}, and the fact that
$$H^*_c(\BA^1 \setminus \{0\}) \ = \ \Qellbar[-2](-1) \oplus \Qellbar[-1](0) \, .$$
For $(\gr \, \Psi (\IC_{Y^2_G}))_{\text{not} \, \text{on} \, {}_2Y^2_B}$ we first invoke Lemma \ref{interplay away from worst stratum} and then compute its image in the Grothendieck group to be
$$\Qellbar(1) - \sign(1) - 2 \Qellbar(0) - 2 \sign(0)$$
by using the cartesian square from the proof of Lemma \ref{restriction along s of the usual summands} above. We have now established the above formula for
$$(\gr \, \Psi (\IC_{Y^2_G}))_{\text{on} \, {}_2Y^2_B} \ - \ (\gr \, \Psi (\IC_{Y^2_G}))_{\text{on} \, {}_2Y^2_B}(-1) \, .$$
But since $(\gr \, \Psi (\IC_{Y^2_G}))_{\text{on} \, {}_2Y^2_B}$ is perverse, it can be reconstructed from this difference by induction, starting with the lowest weight; executing this algorithm yields
$$(\gr \, \Psi (\IC_{Y^2_G}))_{\text{on} \, {}_2Y^2_B} \ = \ \sign(1) + \sign(0) + \sign(-1) + \Qellbar(0)$$
as desired, completing the proof.

\bigskip\bigskip

\ssec{Picard-Lefschetz oscillators via IC-stalks for small defect}
\label{Picard-Lefschetz oscillators via intersection cohomology}

\mbox{} \medskip

In this subsection we give the aforementioned second proof of the correctness of the $S_2$-equivariant structure. This proof is more technical than the previous one, and may safely be skipped by the reader. We include it for two reasons: First, because it illustrates how to work with the local models in very explicit terms; while an abstract and quick proof is usually preferable, being able to work out basic examples very explicitly may often be useful in applications, and the present subsection provides various worked out examples for small defect. Second, this proof makes use of an explicit IC-sheaf computation for small defects; this illustrates the intricacies of computing the IC-sheaf directly, possibly motivating the more abstract approach via the monodromy action taken in Section \ref{Intersection cohomology}. It furthermore shows that, maybe rather surprisingly, the IC-sheaf does not factorize, as was already mentioned in the introduction.

\medskip

The idea of the present proof is as follows. First, we compute the IC-stalks in the case of defect $2$ in a direct manner. This computation relies on the explicit equations for the local models derived earlier; the equations are used to compute cohomology groups of various loci in the local models via basic algebraic topology. We then use this calculation to determine the desired $S_2$-equivariant structure: Our knowledge of the IC-stalks lets us understand the perverse kernel of the monodromy operator $N$ on the associated graded of the nearby cycles; exploiting the compatibility of the $S_2$-action and the action of the Lefschetz-$\sl_2$, this knowledge is sufficient to verify the correctness of the $S_2$-equivariant structure.

\medskip

\sssec{Intersection cohomology for $Y^2$}
Our goal is to show:

\medskip

\begin{proposition}
\label{IC-restriction for Y^2}
The restriction of the IC-sheaf of $Y^2$ to the stratum of maximal defect ${}_2Y^2_B = X^{(2)}$ equals:
$$s^* \IC_{Y^2} \ \ = \ \ \Qellbar_{X^{(2)}}[3](\tfrac{3}{2}) \ \oplus \ \Qellbar_{X^{(2)}}[5](\tfrac{5}{2})$$
\end{proposition}

\bigskip

We begin with the following lemma:

\medskip

\begin{lemma}
\label{IC-restriction two points no simple on diagonal}
None of the simple perverse sheaves occurring in the minimal $\BZ$-linear combination of $s^*\IC_{Y^2}$ in the Grothendieck group of perverse sheaves on $X^{(2)}$ is supported on the diagonal $\Delta_X$ of $X^{(2)}$.
\end{lemma}

\medskip

\begin{proof}
On the level of the Grothendieck group the object $s^* \IC_{Y^2}$ agrees up to twist and sign with the $*$-restriction along $s$ of the associated graded $\gr \IC_{Y^2}|^*_{Y^2_B}[-1](-\tfrac{1}{2})$. The latter associated graded object is however a subobject of the perverse sheaf $\gr \Psi(\IC_{Y^2_G})$ by Subsection \ref{Preliminaries about nearby and vanishing cycles}. Hence the claim will follow once we establish the analogous claim for each restriction $s^*P$ for each simple perverse sheaf $P$ occurring in $\gr \Psi(\IC_{Y^2_G})$. To prove the latter, we split $\gr \Psi(\IC_{Y^2_G})$ as a direct sum as in Subsection \ref{Splitting according to loci of support} above: For each simple $P$ occurring in the summand $(\gr \, \Psi (\IC_{Y^2_G}))_{\text{on} \, {}_2Y^2_B}$ the needed assertion is then precisely Lemma \ref{no summands over the diagonal lemma} above for $n=2$. For each simple $P$ occurring in the summand $(\gr \, \Psi (\IC_{Y^2_G}))_{\text{not} \, \text{on} \, {}_2Y^2_B}$ the needed assertion follows from the validity of Theorem \ref{main theorem for nearby cycles, local version} for $n=1$ and Lemma~\ref{restriction along s of the usual summands} above.
\end{proof}

\medskip

As before we denote by
$$\add: \ X \stackrel{\circ}{\times} X \ \longto X^{(2)}$$

\noindent the addition map of effective divisors. Complementarily to Lemma \ref{IC-restriction two points no simple on diagonal} we now prove:

\medskip

\begin{lemma}
\label{IC-restriction two points add pullback}

$$\add^* s^! \IC_{Y^2} \ = \ \Qellbar_{X \stackrel{\circ}{\times} X}[1](\tfrac{1}{2}) \ \oplus \ \Qellbar_{X \stackrel{\circ}{\times} X}[-1](-\tfrac{1}{2})$$
\end{lemma}

\medskip

\begin{proof}
By the contraction principle (Lemma \ref{contraction principle} above) and the factorization in families (Proposition \ref{factorization in families}) above, we have to show that
$$(\pi_1 \times \pi_1)_! \IC_{Y^1 \underset{ \ \BA^1}{\stackrel{\circ}{\times}} Y^1} \ \ = \ \ \ \Qellbar_{X \stackrel{\circ}{\times} X}[1](\tfrac{1}{2}) \ \oplus \ \Qellbar_{X \stackrel{\circ}{\times} X}[-1](-\tfrac{1}{2})$$
on the disjoint locus of $X \times X$. As in the proof of Proposition \ref{base case proposition} above it suffices to verify this at the level of $*$-stalks. To do this, note that the contracting $\BG_m$-action on $Y^1$ induces a $\BG_m$-action on the fiber product $Y^1 \underset{ \ \BA^1}{\stackrel{\circ}{\times}} Y^1$ by Lemma \ref{composite map and contraction in coordinates}; this action respects the product projection $\pi_1 \times \pi_1$ and contracts the fiber product onto the product section $s_1 \times s_1$. Thus, applying the contraction principle again, we are left to verify that the $!$-stalk of the IC-sheaf of the subvariety $\BY^1 \times_{\BA^1} \BY^1 \subset \BA^4$ at the origin $0 \in \BA^4$ is equal to
$$\IC_{ \, \BY^1 \underset{ \ \BA^1}{\times} \BY^1} \Big|^!_0 \ \ \ = \ \ \ \Qellbar[-1](-\tfrac{1}{2}) \ \oplus \ \Qellbar[-3](-\tfrac{3}{2}) \, .$$
But the equations from Subsection \ref{Explicit equations and generalized Picard-Lefschetz families} show that the subvariety $\BY^1 \times_{\BA^1} \BY^1 \subset \BA^4$ is precisely the affine quadric cone defined by the equation $XY = ZW$ in~$\BA^4$; the standard calculation of the IC-stalk at the vertex of the cone, for example via a resolution of singularities, then yields the result.
\end{proof}

\medskip

By Corollary \ref{preservation of purity} above we already know that $s^*\IC_{Y^2}$ is pure of weight~$0$; combining this with Lemma \ref{IC-restriction two points no simple on diagonal} and Lemma \ref{IC-restriction two points add pullback} above, we conclude that
$$s^! \IC_{Y^2} \ \ = \ \ L_1[-1](-\tfrac{1}{2}) \oplus L_2[-3](-\tfrac{3}{2})$$

\medskip

\noindent where $L_1$ and $L_2$ can be either equal to the shifted constant sheaf $(\Qellbar)_{X^{(2)}}[2](1)$ or to the IC-extension of the sign local system from the disjoint locus in $X^{(2)}$. Thus to prove Proposition \ref{IC-restriction for Y^2} above, we have to prove that both $L_1$ and~$L_2$ are equal to the constant sheaf, i.e., we have to rule out the appearance of sign local systems. To do so, we will ``compute over the diagonal'', for which we will utilize our concrete understanding of the space $\BY^2$ in coordinates. More precisely, since the stalk of the IC-extension of the sign local system at a point on the diagonal $\Delta_X$ of $X^{(2)}$ vanishes, Proposition \ref{IC-restriction for Y^2} will follow from the following lemma:

\medskip

\begin{lemma}
The $*$-stalk of $s^! \IC_{Y^2}$ at any point on the diagonal $\Delta_X$ is equal to
$$\Qellbar[1](\tfrac{1}{2}) \ \oplus \ \Qellbar[-1](-\tfrac{1}{2}).$$
\end{lemma}

\medskip

\begin{proof}
Since $s^! \IC_{Y^2} = \pi_! \IC_{Y^2}$ by the contraction principle, Lemma \ref{contraction principle} above, the above $*$-stalk is equal to the compactly supported cohomology
$$H_c^*(\BY^2, \IC_{Y^2}|^*_{\BY^2})$$
of the restriction $\IC_{Y^2}|^*_{\BY^2}$. We thus have to show that these cohomology groups are $1$-dimensional in the relevant degrees $1$ and $-1$; in doing so, the weights are irrelevant, so we suppress them from the notation throughout the proof. To compute these cohomology groups, we will use the long exact sequence in compactly supported cohomology
$$H_c^*({}_{\leq 1}\BY^2, \IC_{Y^2}|^*_{{}_{\leq 1}\BY^2}) \ \longto \ H_c^*(\BY^2, \IC_{Y^2}|^*_{\BY^2}) \ \longto \ H_c^*({}_2\BY^2_B, \IC_{Y^2}|^*_{{}_2\BY^2_B})$$
associated to the pair $({}_{\leq 1} \BY^2, {}_2\BY^2_B)$ consisting of the complementary open and closed subvarieties
$${}_{\leq 1} \BY^2 \ \stackrel{\text{open}}{\longinto} \ \BY^2 \ \stackrel{\text{closed}}{\longotni} \ {}_2\BY^2_B \, .$$
Observe that the open subvariety ${}_{\leq 1} \BY^2$ complementary to ${}_2\BY^2_B$ consists of the $G$-locus $\BY^2_G$ as well as the strata ${}_0 \BY^2_B$ and ${}_1 \BY^2_B$ of the $B$-locus.

\medskip

To analyze the term $H_c^*({}_2\BY^2_B, \IC_{Y^2}|^*_{{}_2\BY^2_B})$ in this sequence, observe that ${}_2\BY^2_B$ consists of precisely one point, which we will denote by $p$; hence this term is simply equal to the stalk $\IC_{Y^2}|^*_p$. But applying Verdier duality to our preliminary knowledge of $s^! \IC_{Y^2}$ in terms of $L_1$ and $L_2$ above, we already know that this stalk must be concentrated in cohomological degrees $-3$ and~$-5$.

\medskip

To analyze the term $H_c^*({}_{\leq 1}\BY^2, \IC_{Y^2}|^*_{{}_{\leq 1}\BY^2})$, note first that since $\BY^1$ and hence also $Y^1$ are smooth by Subsection~\ref{Explicit equations and generalized Picard-Lefschetz families}, the open locus ${}_{\leq 1}Y^2 \subset Y^2$ is smooth as well. Since ${}_{\leq 1}\BY^2 \subset {}_{\leq 1}Y^2$ and since $\dim Y^2 = 5$ we hence conclude that
$$\IC_{Y^2}\big|^*_{{}_{\leq 1}\BY^2} \ \ = \ \ \Qellbar_{{}_{\leq 1}\BY^2}[5] \, .$$

\medskip

Combining the last two observations, the long exact sequence shows:
$$H_c^1(\BY^2, \IC_{Y^2}|^*_{\BY^2}) \ \ = \ \ H_c^1({}_{\leq 1}\BY^2, \Qellbar [5])$$
$$H_c^{-1} (\BY^2, \IC_{Y^2}|^*_{\BY^2}) \ \ = \ \ H_c^{-1}({}_{\leq 1}\BY^2, \Qellbar [5])$$
Thus the proof of the lemma is completed by the computation of the compactly supported cohomology groups of the variety ${}_{\leq 1}\BY^2$ on the right hand side in the next lemma.
\end{proof}

\medskip

Continuing to suppress the weights from the notation due to their irrelevance for the present question, we conclude the proof of Proposition \ref{IC-restriction for Y^2} by showing:

\begin{lemma}
$$H_c^6({}_{\leq 1}\BY^2, \Qellbar) \ \ = \ \ \Qellbar$$
$$H_c^4({}_{\leq 1}\BY^2, \Qellbar) \ \ = \ \ \Qellbar$$
\end{lemma}

\medskip

\begin{proof}
From Subsection~\ref{Explicit equations and generalized Picard-Lefschetz families} above it follows that the subvariety $\BY^2 \subset \BA^4$ is the affine quadric cone in $\BA^4$ defined by the equation $XY + ZW = 0$; the closed subvariety ${}_2\BY^2_B$ corresponds precisely to the vertex of the cone. The open subvariety ${}_{\leq 1}\BY^2$ thus forms a $\BG_m$-bundle over the smooth quadric surface in $\BP^3$ and is hence smooth itself. Since ${}_{\leq 1}\BY^2$ is $3$-dimensional and irreducible, the first claim follows. For the second claim, we can by Poincare duality equivalently compute $H^2({}_{\leq 1}\BY^2, \Qellbar)$. The latter cohomology group can in turn be shown to be isomorphic to $\Qellbar$ using the Gysin sequence for the first Chern class of the $\BG_m$-bundle ${}_{\leq 1}\BY^2$ over the quadric surface in $\BP^3$.
\end{proof}

\medskip

\sssec{The second proof of the compatibility}

We now give the second proof that the isomorphism
$$\add_{disj}^* \Bigl(\gr \, \Psi (\IC_{Y^2_G}) \Bigr)_{\text{on} \, {}_2Y^2_B} \ \ \cong \ \ \add_{disj}^* \bigl( C_2 \bigr)_{\text{on} \, {}_2Y^2_B}$$
constructed in Lemma \ref{isomorphisms after pullback} is compatible with the equivariant structures on both sides with respect to the symmetric group $S_2 = \BZ/2\BZ$. More precisely, we will relate the last question to the intersection cohomology computation in Proposition \ref{IC-restriction for Y^2} of the previous subsection, and play the symmetries coming from the $S_2$-action and from the action of the Lefschetz-$\sl_2$ off of each other.

\medskip

To make the task more explicit, observe first that the $S_2$-equivariant structure on the pullback $\add_{disj}^*(\gr \, \Psi (\IC_{Y^2_G}))_{\text{on} \, {}_2Y^2_B}$ corresponds to a representation of the symmetric group $S_2$ on the tensor product $V \otimes V$ of the standard representation
$$V \ = \ \Qellbar(\tfrac{1}{2}) \oplus \Qellbar(-\tfrac{1}{2})$$
of the Lefschetz-$\sl_2$ with itself, as in Subsection \ref{Definition of Picard-Lefschetz oscillators}. In particular this action of $S_2$ must commute with the action of the Lefschetz-$\sl_2$. We now have to verify that the action of the non-trivial element $\sigma \in S_2 = \BZ/2\BZ$ on $V \otimes V$ is given by flipping the two factors and multiplying by $-1$.

\medskip

Denote by $U_k$ the irreducible representation of the Lefschetz-$\sl_2$ of highest weight $k \in \BZ_{\geq 0}$. Then since by definition $V = U_1$ the tensor product $V \otimes V$ decomposes as the direct sum
$$V \otimes V \ = \  \Lambda^2V \oplus \Sym^2V \ = \ U_0 \oplus U_2 \, .$$
Since the action of $\sigma$ commutes with the action of the Lefschetz-$\sl_2$, the action of $\sigma$ respects this direct sum decomposition; as the summands are irreducible as representations of the Lefschetz-$\sl_2$, the action of $\sigma$ on each of the summands must then be given by multiplication by either $+1$ or $-1$. We have to show that $\sigma$ acts by $+1$ on $U_0$ and by $-1$ on $U_2$. To determine these signs, it is of course enough to know how $\sigma$ acts on the lowest weight lines $M_0 = \Qellbar$ of $U_0$ and $M_2 = \Qellbar(1)$ of~$U_2$. It is precisely these signs on the lowest weight lines that we can access via the intersection cohomology of $Y^2$, as we discuss next.

\medskip

Let $\gr(\IC_{Y^2}|^*_{Y^2_B}[-1](-\tfrac{1}{2}))$ denote the associated graded perverse sheaf with respect to the weight filtration on $\IC_{Y^2}|^*_{Y^2_B}[-1](-\tfrac{1}{2})$, and let
$$(\gr(\IC_{Y^2}|^*_{Y^2_B}[-1](-\tfrac{1}{2})))_{ \text{on} \,{}_2Y^2_B}$$
denote its direct summand consisting of those simples which are supported on the stratum of maximal defect ${}_2Y^2_B$. By Lemma \ref{kernel and associated graded commute} and Lemma \ref{IC via Psi} the latter object is precisely the perverse kernel of the monodromy operator $N$ acting on $(\gr \, \Psi (\IC_{Y^2_G}))_{\text{on} \, {}_2Y^2_B}$. Its pullback to the disjoint locus of $X \times X$ hence corresponds to the $\BZ/2\BZ$-subrepresentation
$$M_0 \oplus M_2 \ \subset \ U_0 \oplus U_2 \ = \ V \otimes V$$
formed by the direct sum of the lowest weight lines $M_0$ and $M_2$. Since $M_0 = \Qellbar$ and $M_2 = \Qellbar(1)$ are of different weight, the signs by which $\sigma$ acts on $M_0$ and $M_2$ can thus be read off from the simple summands appearing in
$$(\gr(\IC_{Y^2}|^*_{Y^2_B}[-1](-\tfrac{1}{2})))_{ \text{on} \,{}_2Y^2_B} \, ,$$
or even its restriction to the disjoint locus in $X^{(2)}$. Namely, from the Tate twists of the local systems on the right hand side in the next lemma we conclude that $\sigma$ acts by $+1$ on $M_0$ and by $-1$ on $M_2$, completing the proof.

\bigskip

\begin{lemma}
On the disjoint locus $\overset{\circ}{X}{}^{(2)}$ we have
$$(\gr(\IC_{Y^2}|^*_{Y^2_B}[-1](-\tfrac{1}{2})))_{ \text{on} \,{}_2Y^2_B} \Big|^*_{\overset{\circ}{X}{}^{(2)}} \ \ \ = \ \ \ \Bigl(\Qellbar \oplus \sign (1)\Bigr)_{\overset{\circ}{X}{}^{(2)}}[2](1)\, ,$$
where $\sign(1)$ denotes the sign local system on $\overset{\circ}{X}{}^{(2)}$ twisted by $1$.
\end{lemma}

\begin{proof}
For readability we erase from the notation all symbols indicating a restriction to the disjoint locus of $X^{(2)}$, throughout the proof.
Since $(\gr(\IC_{Y^2}|^*_{Y^2_B}[-1](-\tfrac{1}{2})))_{ \text{on} \,{}_2Y^2_B}$ is a semisimple perverse sheaf, it suffices to perform the necessary calculation in the Grothendieck group of perverse sheaves on $X^{(2)}$. But in the Grothendieck group the latter object is equal to the difference
$$s^* \IC_{Y^2}[-1](-\tfrac{1}{2}) \ \ \ - \ \ \ s^*(\gr(\IC_{Y^2}|^*_{Y^2_B})[-1](-\tfrac{1}{2}))_{ \text{not} \, \text{on} \, {}_2Y^2_B} \, .$$
We first compute the second term: As before we invoke the validity of Theorem \ref{main theorem for nearby cycles, local version} for $n=1$ and apply Lemma \ref{interplay away from worst stratum} above; the second term is thus equal to the $*$-pullback along $s$ of the kernel of the action of the monodromy operator $N$ on $(C_2)_{ \text{not} \, \text{on} \, {}_2Y^2_B}$. Using the exact same cartesian diagram as in the proof of Lemma \ref{restriction along s of the usual summands} above one then computes that this second term is equal to
$$3 \cdot \Qellbar(1) \ + \ \sign(1) \ - \ 2 \cdot \Qellbar(1) \ - \ 2 \cdot \sign(1) \ \ \ = \ \ \ \Qellbar(1) \ - \ \sign(1) \, .$$
Here, for notational brevity, we write $\Qellbar$ and $\sign$ for the perverse sheaves $\Qellbar[2](1)$ and $\sign[2](1)$ of weight $0$.
However, by Proposition \ref{IC-restriction for Y^2} above, the first term is equal to
$$\Qellbar(0) \ + \ \Qellbar(1) \, .$$
Taking the difference of the two terms we find the desired expression
$$\Qellbar(0) \ + \ \sign(1) \, .$$
\end{proof}

\bigskip\bigskip\bigskip

\section{Intersection cohomology}
\label{Intersection cohomology}

In this Section we comment on how Theorem \ref{main theorem for intersection cohomology} about the intersection cohomology follows from Theorem \ref{main theorem for nearby cycles}, as well as on how Theorem \ref{main theorem for intersection cohomology} in turn can be used to compute the IC-stalks. As Theorem \ref{ij} follows from Proposition \ref{omega tilde proposition} by the same argument as in Subsection \ref{Restatements of the main theorems for the local models}, the present section concludes the proof of the main theorems stated in Section \ref{Statement of main theorems} above.

\bigskip

\ssec{Intersection cohomology from nearby cycles}

\sssec{Classical Schur-Weyl duality}
Recall that, working over an algebraically closed field of characteristic $0$, the irreducible representations of the symmetric group $S_k$ are in one-to-one correspondence with Young diagrams consisting of precisely $k$ boxes. Furthermore, any Young diagram with at most $m$ rows, but an arbitrary number of boxes, gives rise to an irreducible representation of the general linear group $\GL_m$. For a Young diagram $D$ with precisely $k$ boxes and at most $m$ rows we denote $\rho_D$ and by $U_D$ the corresponding irreducible representations of $S_k$ and $\GL_m$.

\medskip

Let now $U_{taut}$ denote the tautological $m$-dimensional representation of $\GL_m$. The $m$-fold tensor product $U_{taut} \otimes \ldots \otimes U_{taut}$ carries the diagonal action of $\GL_m$ as well as the permutation action of the symmetric group $S_k$, and these actions commute. The classical Schur-Weyl duality then states:

\medskip

\begin{lemma}[Classical Schur-Weyl duality]
\label{classical Schur-Weyl duality}
As a bi-representation of $\GL_m$ and $S_k$ the $k$-fold tensor product $U_{taut} \otimes \ldots \otimes U_{taut}$ decomposes as
$$U_{taut} \otimes \ldots \otimes U_{taut} \ \ = \ \ \bigoplus_D U_D \otimes \rho_D$$
where the sum runs over all Young diagrams $D$ consisting of precisely $k$ boxes and at most $m$ rows.
\end{lemma}

\medskip

We now apply this in the following context:

\sssec{Decomposing the Picard-Lefschetz oscillators}

As in Subsection \ref{Definition of Picard-Lefschetz oscillators} above let
$$V \ = \ \Qellbar(\tfrac{1}{2}) \oplus \Qellbar(-\tfrac{1}{2}) \, .$$
As in Lemma \ref{properties of PLOs} above let $V \otimes \ldots \otimes V$ denote the $k$-fold tensor product of $V$, together with the action of $S_k$ defined by permuting the factors and multiplying with the sign of the permutation. Since $V$ is precisely the tautological $2$-dimensional representation of the Lefschetz-$\sl_2$, the appropriate variant of Lemma \ref{classical Schur-Weyl duality} above yields:

\medskip

\begin{lemma}
\label{decomposing PLO}
As a bi-representation of $S_k$ and the Lefschetz-$\sl_2$ the $k$-fold tensor product $V \otimes \ldots \otimes V$ decomposes as
$$V \otimes \ldots \otimes V \ \ = \ \ \bigoplus_{0 \, \leq \, r \, \leq \, \tfrac{k}{2}} U_{k-2r} \otimes \rho_{(k-r, r)} \, .$$
Here we denote by $U_{k-2r}$ the irreducible representation of the Lefschetz-$\sl_2$ of highest weight $k-2r$ and by $\rho_{(k-r,r)}$ the irreducible representation of $S_k$ corresponding to the Young diagram with $k-r$ boxes in the first column and~$r$ boxes in the second column.
\end{lemma}

\bigskip

\sssec{Proof of Theorem \ref{main theorem for intersection cohomology}}
Observe first that Lemma \ref{decomposing PLO} yields an explicit direct sum decomposition into simple perverse sheaves of the Picard-Lefschetz oscillator $\CP_k$ by Lemma \ref{properties of PLOs} above. We however only need the following consequence:

\begin{lemma}
\label{PLO kernel}
The perverse kernel $\ker(N)$ of the monodromy operator $N$ acting on the Picard-Lefschetz oscillator $\CP_k$ is equal to the IC-extension of the local system on the disjoint locus of $X^{(k)}$ corresponding to the following representation of the symmetric group $S_k$:

$$\bigoplus_{0 \, \leq \, r \, \leq \, \tfrac{k}{2}} \rho_{(k-r,r)} \otimes \Qellbar(\tfrac{k}{2} - r)$$
Here, as before, we denote by $\rho_{(k-r,r)}$ the irreducible representation of $S_k$ corresponding to the Young diagram with $k-r$ boxes in the first column and~$r$ boxes in the second column; here the second tensor factor indicates the appropriate Tate twist.
\end{lemma}

\bigskip

To prove Theorem \ref{main theorem for intersection cohomology} it suffices, by Lemma \ref{kernel and associated graded commute} above, to compute the kernel of the monodromy operator $N$ on $\gr \Psi (\IC_{\VinBun_{G,G}})$. Using Theorem~\ref{main theorem for nearby cycles} and the fact that the maps $\bar{f}_{n_1, k, n_2}$ are finite, the assertion thus follows from Lemma \ref{PLO kernel} above.

\medskip

\sssec{Remark}
\label{IC stalks}
Theorem \ref{main theorem for intersection cohomology} also provides an algorithm to compute IC-stalks: As in Subsection \ref{Restatements of the main theorems for the local models} above, to compute the IC-stalks of $\VinBun_G$ along the strata of defect $k$ we can equivalently compute the IC-stalks of the local model $Y^k$ along the stratum of maximal defect ${}_kY^k$. Thus it suffices to derive an explicit formula for the restriction $s^*\IC_{Y^k}$ of the IC-sheaf of $Y^k$ along the section~$s$. To do so, note that Theorem \ref{main theorem for intersection cohomology} yields an explicit formula for $s^*\IC_{Y^k}$ in the Grothendieck group. But by Corollary \ref{preservation of purity} the complex $s^*\IC_{Y^k}$ is pure of weight $0$, which makes it possible to reconstruct cohomological shifts from the image of $s^*\IC_{Y^k}$ in the Grothendieck group and to algorithmically reconstruct its stalks. We do not know whether the resulting, rather involved, formulas can be put in a simple form; for applications, the description in Theorem \ref{main theorem for intersection cohomology} appears to be more valuable.

\bigskip\bigskip\bigskip

\section{An application: Computation of Drinfeld's function}
\label{An application: Computation of Drinfeld's function}

This section is separate from the main text. Its goal is to use the results of this article to explicitly describe a certain function introduced by Drinfeld which arises in the theory of automorphic forms. We refer the reader to Subsection \ref{An application to Drinfeld's strange invariant bilinear form} of the introduction, as well as to \cite{DrW}, for how this computation is applied in Drinfeld's and Wang's work on the \textit{strange invariant bilinear form} on the space of automorphic forms.

\ssec{The statement}

\sssec{The question}
Let $\BF_q$ be a finite field with $q$ elements, let $X$ be a smooth projective curve over $\BF_q$, let $G = \SL_2$ over $\BF_q$, and consider $\Bun_G$ over $\BF_q$. As above let
$$\Delta: \ \Bun_G \ \stackrel{\Delta}{\longto} \ \Bun_G \times \Bun_G$$
denote the diagonal morphism of $\Bun_G$ and let $\Qellbar_{\Bun_G}$ denote the constant sheaf on $\Bun_G$. Then we will answer the following question:

\medskip

\begin{question}
Under the sheaf-function correspondence, what is the function on $\Bun_G \times \Bun_G$ corresponding to the pushforward $\Delta_* \Qellbar_{\Bun_G}$? I.e., what is the trace of the action of the geometric Frobenius on the $*$-stalks of the pushforward $\Delta_* \Qellbar_{\Bun_G}$ at $\BF_q$-points of $\Bun_G \times \Bun_G$?
\end{question}

\medskip

Following a suggestion of Drinfeld, we will use the compactification $\barBun_G$ of the diagonal $\Delta$: Answering the above question amounts to understanding the $*$-stalks of the pushforward of the constant sheaf $\Qellbar_{\Bun_G}$ along the natural map
$$b: \ \Bun_G \ \longto \ \barBun_G \, .$$

\medskip

\sssec{Notation}
\label{Notation function section}
To state the answer to the above question we need to introduce the following notation. First, given an $\BF_q$-point $(E_1, E_2)$ of $\Bun_G \times \Bun_G$, we denote by $\Isom_{\SL_2}(E_1,E_2)(\BF_q)$ the set of vector bundle isomorphisms $E_1 \to E_2$ of determinant $1$, i.e., the set of isomorphisms as $\SL_2$-bundles. Next, let $\varphi: E_1 \to E_2$ be a non-zero morphism of vector bundles which is not an isomorphism. Factoring $\varphi$ as
$$E_1 \ \longonto \ M_1 \ \longinto \ M_2 \ \longintointo \ E_2$$
as in Subsection \ref{Definition of the defect} above we associate to $\varphi$ its defect divisor $D_{\varphi}$, which forms an $\BF_q$-point of the symmetric power $X^{(n)}$ for some integer $n$. The defect divisor $D_{\varphi}$ can be written as a sum
$$D_{\varphi} \ = \ \sum_{k} n_{k, \varphi} \, x_{k, \varphi}$$
where the $x_{k, \varphi}$ are distinct closed points of the curve $X$ over $\BF_q$. We then denote by $d_{k,\varphi}$ the degree of the residue field extension at the point $x_{k, \varphi}$. We can now state:

\bigskip
 
\sssec{The answer}
With the above notation we have:

\begin{proposition}
\label{answer}
Let $(E_1, E_2)$ be an $\BF_q$-point of $\Bun_G \times \Bun_G$. Then the trace of the geometric Frobenius on the $*$-stalk at~$(E_1, E_2)$ of the pushforward $\Delta_* \Qellbar$ is equal to:
$$|\Isom_{\SL_2}(E_1,E_2)(\BF_q)|  \ \ \ \ \ \ - \sum_{\substack{\varphi \ \in \ \Hom(E_1, E_2)(\BF_q) \\ \varphi \ \text{is not an isomorphism} \\ \varphi \neq 0}} \prod_k (1-q^{d_{k,\varphi}})$$
\end{proposition}

\bigskip

\ssec{Reduction to a trace computation on $\barBun_G$}
\label{Reduction to a trace computation on barBun_G}

\mbox{} \medskip

In this subsection we deduce Proposition \ref{answer} above, via the Lefschetz trace formula, from a computation on $\barBun_G$ stated in Proposition \ref{trace at point in B-locus} below; in the next subsection we will then prove Proposition \ref{trace at point in B-locus}.
We remark that the presentation in the present subsection would be simpler if our definition of $\barBun_G$ was such that it contains $\Bun_G$ as a dense open substack. In particular, it would then be unnecessary to distinguish between the case where the characteristic is equal or not equal to $2$, and Lemma \ref{trace at point in B-locus in local model} and Lemma \ref{gamma^*b_*} below would become redundant. (The present definition of $\barBun_G$ is adopted in this article since this definition appears to be the most natural in the case of an arbitrary reductive group $G$; see the follow-up articles \cite{Sch1}, \cite{Sch2}.)

\medskip

To state Proposition \ref{trace at point in B-locus}, recall from Subsection \ref{Compactifying the diagonal} the natural map
$$b: \ \Bun_G \ \longto \ \barBun_G \, ,$$
and let $z = (E_1, E_2, L, \varphi)$ be an $\BF_q$-point of the $B$-locus of $\barBun_G$. Exactly as in Subsection \ref{Notation function section} above we can associate to the point $z$, via the defect divisor of the map $\varphi$, the collection of closed points $x_{k, \varphi}$ and integers $d_{k,\varphi}$. With this notation we have:

\medskip

\begin{proposition}
\label{trace at point in B-locus}
The trace of the geometric Frobenius on the $*$-stalk of $b_*\Qellbar$ at $z$ is equal to
$$(1-q) \cdot \prod_k (1-q^{d_{k,\varphi}})\, .$$
\end{proposition}

\bigskip

From this Proposition \ref{answer} above follows via the Lefschetz trace formula:

\bigskip

\begin{proof}[Proof of Proposition \ref{answer}]
Consider the diagram
$$\xymatrix@+10pt{
   &   \Bun_G \ar^{b}[d]  \ar@/_-2pc/[dd]^{\Delta} \\
\BP(\Hom(E_1, E_2)) \ar[d] \ar^{ \ \ \ g}[r] &   \barBun_G \ar^{\bar\Delta}[d] \\
\Spec \BF_q \ar^{(E_1, E_2) \ \ \ }[r] & \Bun_G \times \Bun_G
}$$
where the square is cartesian. Since $\bar\Delta$ is proper, we can compute the desired trace via the Lefschetz trace formula applied to $g^*b_* \Qellbar$ on the projectivization $\BP(\Hom(E_1, E_2))$ of $\Hom(E_1, E_2)$: The desired trace is equal to

$$\sum_{z \ \in \ \BP(\Hom(E_1, E_2))(\BF_q)} \tr(\Frob, z^*b_* \Qellbar) \, .$$

\noindent To rewrite this formula, we abuse notation and denote again by $z$ the $\BF_q$-point of $\barBun_G$ obtained via the map $g$ from the $\BF_q$-point $z$ of $\BP(\Hom(E_1, E_2))$. We then split the sum according to whether the $\BF_q$-point $z$ lies in the $G$-locus or the $B$-locus of $\barBun_G$, i.e., according to whether the corresponding map~$\varphi$ is an isomorphism or not. The summand corresponding to the $B$-locus is computed by Proposition \ref{trace at point in B-locus} above and contributes the second term in the formula in Proposition \ref{answer}. To compute the summand corresponding to the $G$-locus, recall first that the map
$$b: \ \Bun_G \ \longto \ \barBun_G$$
is not an open immersion: It forms an etale cover of degree $2$ of the $G$-locus of $\barBun_G$ if the characteristic is not $2$, and defines a radicial map onto the $G$-locus if the characteristic is equal to $2$. To avoid having to distinguish these two cases, let
$$\BP \Isom_{\GL_2}(E_1, E_2) \ \subset \ \BP \Hom(E_1,E_2)$$
denote the quotient by $\BG_m$ of the space of isomorphisms of vector bundles $\Isom_{\GL_2}(E_1, E_2)$, and let
$$r: \ \Isom_{\SL_2}(E_1,E_2) \ \longto \ \BP \Isom_{\GL_2}(E_1, E_2)$$
denote the natural map. Then the restriction of $g^* b_* \Qellbar$ to the open subscheme $\BP \Isom_{\GL_2}(E_1, E_2)$ is equal to $r_* \Qellbar$. Since the map $r$ is finite we conclude that the contribution of the $G$-locus is equal to
$$\sum_{z \ \in \ \BP(\Isom(E_1, E_2))(\BF_q)} \tr(\Frob, z^*b_* \Qellbar) \ \ \ = \ \ \ |\Isom_{\SL_2}(E_1,E_2)(\BF_q)| \, ,$$
contributing the first term in the formula in Proposition \ref{answer}.
\end{proof}

\bigskip
 
\ssec{Proof of the trace computation via local models}

\mbox{} \medskip

We now prove Proposition \ref{trace at point in B-locus} above. First, we reduce the assertion to the analogous assertion for $\VinBun_G$, or equivalently for the local models $Y^n$, stated in Lemma \ref{trace at point in B-locus in local model} below. A minor reduction step is necessary since the map $b$ above is not an open immersion. We then prove Lemma \ref{trace at point in B-locus in local model}, using the results of Subsection \ref{Fighting simples on the main diagonal} above. To state the lemma let
$$j_G: \ Y^n_G \ \longinto \ Y^n$$
denote the open inclusion of the $G$-locus of $Y^n$, and let $z$ be an $\BF_q$-point of the stratum ${}_nY^n_B = X^{(n)}$ of maximal defect. Then with the exact same notation as in the previous two subsections we have:

\begin{lemma}
\label{trace at point in B-locus in local model}
The trace of the geometric Frobenius on the $*$-stalk of $j_{G,*} \, \Qellbar$ at the point $z$ is equal to
$$(1-q) \cdot \prod_k (1-q^{d_{k, \varphi}}) \, .$$
\end{lemma}

\medskip

\sssec{Reduction to the lemma}
As in Subsection \ref{Restatements of the main theorems for the local models} above, knowing Lemma \ref{trace at point in B-locus in local model} above for all integers $n \geq 0$ is equivalent to knowing the analogous assertion for $\VinBun_G$. To deduce Proposition \ref{trace at point in B-locus} from the latter, we first assume that the characteristic is not equal to $2$.

\medskip

Denote by $\sign$ the sign local system on $\BA^1 \setminus \{0\}$, and denote by $v^* \sign$ its pullback to the $G$-locus $\VinBun_{G,G}$ along the natural map
$$v: \ \VinBun_{G,G} \ \longto \ \BA^1 \setminus \{0\} \, .$$
Furthermore, denote by
$$\gamma: \ \VinBun_G \ \longto \ \barBun_G$$
the natural forgetful map, and let $j_{\VinBun_{G,G}}$ denote the open inclusion of the $G$-locus of $\VinBun_G$. Then chasing through the definitions one finds:

\begin{lemma}
\label{gamma^*b_*}
$$\gamma^*b_*{\Qellbar}_{\Bun_G} \ \ = \ \  j_{\VinBun_{G,G},*} \Bigl( {\Qellbar}_{\VinBun_{G,G}} \, \oplus \, v^* \sign \Bigr)$$
\end{lemma}

\medskip

Since the Frobenius traces on the $*$-stalks of $j_{\VinBun_{G,G},*} \, {\Qellbar}_{\VinBun_{G,G}}$ can be computed on the local models $Y^n$, Proposition \ref{trace at point in B-locus} follows from Lemma \ref{trace at point in B-locus in local model} above once we show that the second summand in Lemma \ref{gamma^*b_*} does not contribute. More precisely, letting $i_{\VinBun_{G,B}}$ denote the inclusion of the $B$-locus of $\VinBun_G$, we need to show:

\begin{lemma}
\label{sign disappears}
$$i_{\VinBun_{G,B}}^* \ j_{\VinBun_{G,G},*} \ v^* \sign \ = \ 0$$
\end{lemma}

\begin{proof}
As before it suffices to prove the analogous statement on the local models $Y^n$. Thus we have to show that
$$s^* \, j_{G, *} \, v^* \sign \ = \ 0 \, ,$$
where $s$ and $j_G$ are as before and $v$ denotes the natural map
$$Y^n_G \ \longto \ \BA^1 \setminus \{0\} \, .$$
To prove this, note first that Lemma \ref{composite map and contraction in coordinates} above shows that $j_{G, *} \, v^* \sign$ is naturally $\BG_m$-equivariant for the contracting $\BG_m$-action constructed in Subsection \ref{Embedding, section, and contraction} above. Applying the contraction principle (see Lemma \ref{contraction principle}) and Lemma \ref{G-locus decomposition for the local models} above, the desired vanishing follows from the fact that the sign local system on $\BA^1 \setminus \{0\}$ has trivial cohomology.
\end{proof}

\medskip

This concludes the reduction step under the assumption that the characteristic is not equal to $2$. If the characteristic is equal to $2$, then the map $b$ defines a radicial map from $\Bun_G$ to the $G$-locus $\barBun_{G,G}$. Thus the summand $v^* \sign$ does not appear in Lemma \ref{gamma^*b_*}, and Lemma \ref{sign disappears} is not even needed.

\bigskip

\sssec{Proof of Lemma \ref{trace at point in B-locus in local model}}
We begin by recalling the following trace computation. As before let $\Lambda^{(n)} (\Qellbar_X)$ denote the $n$-th external exterior power on $X^{(n)}$ of the constant local system $\Qellbar_X$ on the curve $X$ over $\BF_q$. Let $D$ be an $\BF_q$-point of $X^{(n)}$. As before we write
$$D \ = \ \sum_k n_k x_k$$
for certain distinct closed points $x_k$ of the curve $X$ and all $n_k \geq 1$, and we let $d_k$ denote the degree of the residue field extension at $x_k$. We then have:

\begin{lemma}
\label{trace lemma}
The trace of the geometric Frobenius on the $*$-stalk of $\Lambda^{(n)} ({\Qellbar}_X)$ is $0$ unless all $n_k$ are equal to $1$. If all~$n_k$ are equal to $1$, then the trace is equal to
$$\prod_k (-1)^{d_k + 1}.$$
\end{lemma}

\medskip

To prove Lemma \ref{trace at point in B-locus in local model}, we first apply Verdier duality to both sides of the equation in Proposition \ref{omega tilde proposition} above, and then combine the result with Lemma \ref{Omega tilde in the Grothendieck group} above to obtain an expression for $s^*j_{Y^n_G*} \, \Qellbar$ in the Grothendieck group. Applying Lemma \ref{trace lemma} above and taking into account that
$$H^*_c(\BA^1 \setminus \{0\}) \ = \ \Qellbar[-2](-1) \oplus \Qellbar[-1](0)$$
and that
$$\dim Y^n_G \ = \ 2n +1 \,, $$
we find the following formula for the trace of the geometric Frobenius on the $*$-stalk of $j_{Y^n_G*} \, \Qellbar$ at a point $D \in X^{(n)}(\BF_q) = {}_nY^n_B(\BF_q)$:
$$(1-q) \ \cdot \ \sum_{i+j=n} \ \sum_{D_1 + D_2 = D} \ (-1)^j q^j \cdot \prod_{x \, \in \, \supp(D_2)} (-1)^{\deg(x)+1}$$
Here the second sum runs over all pairs $(D_1, D_2)$ of $\BF_q$-points $D_1 \in X^{(i)}$, $D_2 \in X^{(j)}$ such that $D_1 + D_2 = D$ and such that the effective divisor $D_2$ is simple, i.e., each closed point occurring in $D_2$ appears with multiplicity~$1$; furthermore, we write $x \in \supp(D_2)$ to denote that a closed point $x$ of the curve $X$ occurs in $D_2$, and we let $\deg(x)$ denote the degree of the residue field extension at the point~$x$.

\medskip

To reformulate the above formula, let
$$D \ = \ \sum_{k} n_k x_k$$
for certain distinct closed points $x_k$ of the curve $X$, as before. Then the datum of a pair $(D_1, D_2)$ with the above properties is equivalent to the datum of a subset $S$ of the set of closed points $\{x_k\}$ occurring in the effective divisor~$D$. We can then rewrite the above formula as
$$(1-q) \ \cdot \ \sum_{S} (-1)^{|S|} \cdot q^{\sum_{x \in S} \deg(x)}$$
where the sum ranges over all subsets $S$ of the set of closed points occurring in the effective divisor $D \in X^{(n)}(\BF_q)$. We do allow the set $S$ to be the empty set, and in this case the corresponding summand is equal to $1$.

\medskip

Finally, to deduce the formula in Lemma \ref{trace at point in B-locus in local model} from the above preliminary formula, recall that the elementary symmetric polynomials in the variables $X_1, \ldots, X_m$ are precisely the coefficients appearing in the expansion of the product
$$\prod_{k=1}^m \, (T + X_k)$$
as a polynomial in $T$. Taking $m$ to be the number of closed points appearing in the effective divisor $D$, setting $T=1$, and setting $X_k = -q^{\deg x_k}$ transforms the preliminary formula to the desired one in Lemma \ref{trace at point in B-locus in local model}.

\end{document}